\documentclass[11pt]{amsart}
\usepackage{a4wide}
\usepackage{fullpage,amsmath,amssymb,amsfonts,latexsym, amsthm,array}
\usepackage{tabularx}
\usepackage{bbm}
\usepackage{tikz}
\usetikzlibrary{shadows}
\usepackage{todonotes}
\usepackage[T1]{fontenc} 

\newtheorem{defi}{Definition}[section]
\newtheorem{theo}[defi]{Theorem}
\newtheorem{lem}[defi]{Lemma}
\newtheorem{cor}[defi]{Corollary}
\newtheorem{rem}[defi]{Remark}
\newtheorem{prop}[defi]{Proposition}
\newtheorem{example}[defi]{Example}
\usepackage[ruled,vlined,norelsize]{algorithm2e}

\usepackage[utf8]{inputenc}
\usepackage{xspace}

\usepackage{url}

\newcommand{\C}{\ensuremath{\mathcal{C}}\xspace}
\newcommand{\SC}{\ensuremath{\mathcal{S}_{\C}}\xspace}
\newcommand{\cP}{\ensuremath{\mathcal{P}}\xspace}

\newcommand{\Z}{\ensuremath{\mathcal{Z}}\xspace}
\newcommand{\X}{\ensuremath{\mathcal{X}}\xspace}

\def\cal{\mathcal}

\newcommand{\s}{\ensuremath{{\mathcal S}}\xspace}

\newcommand{\si}{\ensuremath{\sigma}\xspace}

\newcommand{\wc}{\ensuremath{{\hat{\mathcal{C}}}}\xspace}
\newcommand{\Wc}{\ensuremath{{\hat{\mathcal{C}}}}\xspace}

\newcommand{\params}[1]{$($#1\rotatebox[origin=c]{180}{$($}}

\newcommand{\calA}{{\mathcal A}}
\newcommand{\cA}{\ensuremath{\mathcal{A}}\xspace}

\newcommand{\Sys}{\ensuremath{\mathcal{E}}\xspace}

\newcommand{\D}{\ensuremath{{\mathcal D}}\xspace}
\newcommand{\E}{\ensuremath{{\mathcal E}}\xspace}

\newcommand{\Bstar}{\ensuremath{{B^\star}}\xspace}
\newcommand{\rs}{restriction\xspace}
\newcommand{\rss}{restrictions\xspace}
\newcommand{\rt}{restriction term\xspace}
\newcommand{\rts}{restriction terms\xspace}
\newcommand{\ssi}{if and only if\xspace}

\makeatletter
\def\@tocline#1#2#3#4#5#6#7{\relax
  \ifnum #1>\c@tocdepth 
  \else
    \par \addpenalty\@secpenalty\addvspace{#2}%
    \begingroup \hyphenpenalty\@M
    \@ifempty{#4}{%
      \@tempdima\csname r@tocindent\number#1\endcsname\relax
    }{%
      \@tempdima#4\relax
    }%
    \parindent\z@ \leftskip#3\relax \advance\leftskip\@tempdima\relax
    \rightskip\@pnumwidth plus4em \parfillskip-\@pnumwidth
    #5\leavevmode\hskip-\@tempdima #6\nobreak\relax
    \dotfill\hbox to\@pnumwidth{\@tocpagenum{#7}}\par 
    \nobreak
    \endgroup
  \fi}
\makeatother

\title{An algorithm computing combinatorial specifications of permutation classes}

\author{Frédérique Bassino}
       \address{Université Paris 13, Sorbonne Paris Cité, LIPN, CNRS UMR 7030, F-93430 Villetaneuse, France}
       \email{bassino@lipn.univ-paris13.fr}

 \author{Mathilde Bouvel}
       \address{Institut für Mathematik, Universität Zürich, Winterthurerstr. 190, CH-8057 Zürich, Switzerland}
       \email{mathilde.bouvel@math.uzh.ch}

 \author{Adeline Pierrot}
 \address{LRI, Université Paris-Sud, Bat. 650 Ada Lovelace, 91405 Orsay Cedex, France}
       \email{adeline.pierrot@lri.fr}

\author{Carine Pivoteau}
        \address{LIGM UMR 8049, Universit\'e Paris Est and CNRS, Marne-la-Vall\'ee, France}
       \email{pivoteau@univ-mlv.fr}

\author{Dominique Rossin}
\address{LIX UMR 7161, \'Ecole Polytechnique and CNRS, Palaiseau, France}
      \email{dominique.rossin@lix.polytechnique.fr}

\begin{document}
\begin{abstract}
This article presents a methodology that automatically derives a combinatorial specification for a permutation class~$\mathcal{C}$, given its basis~$B$ of excluded patterns and the set of simple permutations in~$\mathcal{C}$, when these sets are both finite.
This is achieved considering both pattern avoidance and pattern containment constraints in permutations.
The obtained specification yields a system of equations satisfied by the generating function of~$\mathcal{C}$, this system being always positive and algebraic. 
It also yields a uniform random sampler of permutations in~$\mathcal{C}$. 
The method presented is fully algorithmic.
\end{abstract}

\maketitle

\tableofcontents

\section{Introduction}

Permutation classes (and the underlying pattern order on permutations) were defined in the seventies, 
and since then the enumeration of specific permutation classes (\emph{i.e.}, sets of permutations closed under taking patterns) 
has received a lot of attention. 
In this context, as in many in combinatorics, a recursive description of the permutations belonging to the class 
is often the key towards their enumeration. 
This recursive description is \emph{a priori} specific to the class studied. 
But more recently, the substitution decomposition (along with other general frameworks, see~\cite[and references therein]{Vatter2014}) 
has been introduced for the study of permutation classes: 
it provides a general and systematic approach to their study, with a recursive point of view. 
This tool has already proved useful in solving many enumerative problems~\cite[among others]{AA05,AlAtBr11,AlBr14,AtSaVa12}, 
but also in other areas like algorithmics~\cite{BBCP07,br06}.

The goal of the current paper is to systematize even more the use of substitution decomposition 
for describing recursively and enumerating permutation classes. 
Our main result is an algorithm that computes a combinatorial specification (in the sense of Flajolet and Sedgewick~\cite{FlSe09}) 
for any permutation class containing finitely many simple permutations. 
Note that this problem has been addressed already in~\cite{AA05,BHV08a}, however with much less focus on the algorithmic side. 
Moreover, we introduce in this article a generalization of permutation classes that we call restrictions:
while every permutation class is characterized by a set of forbidden patterns, 
a restriction is described giving a set of forbidden patterns and a set of mandatory patterns.
Our algorithm also allows to compute a specification for restrictions containing finitely many simple permutations.

The article is organized as follows. 
We start by recalling the necessary background in Section~\ref{sec:PermClasses}: 
permutation classes, substitution decomposition, and the symbolic method. 
Section~\ref{sec:detailed_intro} gives a more detailed presentation of our results. 
Here, we dedicate specific attention to explaining the differences between our work and those of~\cite{AA05,BHV08a}, 
and to putting our result in a more global algorithmic context 
(namely, we describe an algorithmic chain from the basis $B$ of a class~$\C$ to random sampling of permutations in~$\C$). 
With the next sections, we enter the technical part of our work. 
After briefly solving the case of substitution-closed classes in Section~\ref{sec:substitution-closed}, 
we explain in two steps how to obtain a combinatorial specification for other classes~$\C$. 
Section~\ref{sec:ambiguous} gives an algorithm producing an ambiguous system of combinatorial equations describing~$\C$.
Next, Section~\ref{sec:disambiguation} describes how to adapt this algorithm to obtain a combinatorial specification for~$\C$. 
Finally, Section~\ref{sec:ex} illustrates the whole process on examples.

\section{Some background on permutations and combinatorial specifications}
\label{sec:PermClasses}

\subsection{Permutation patterns and permutation classes}

A permutation $\sigma$ of size $|\sigma| = n$ is a bijective map from $[1..n] = \{1,\ldots ,n\}$ to itself.
We represent a permutation by a word $\sigma= \sigma_1 \sigma_2 \ldots \sigma_n$, 
where each letter $\sigma_i$ denotes the image of $i$ under $\sigma$. 
We denote $\varepsilon$ the only permutation of size $0$; 
$\varepsilon$ is also called the empty permutation.

\begin{defi}
\label{def:subsequence}
For any sequence $s$ of $k$ distinct integers, the \emph{normalization} of $s$ is the permutation $\pi$ of size $k$ which is order-isomorphic to $s$, 
\emph{i.e.}, $s_{\ell} < s_{m}$ whenever $\pi_{\ell} < \pi_{m}$.

For any permutation $\sigma$ of size $n$, and any subset $I = \{i_1,\ldots, i_k\}$ of $\{1, \ldots, n\}$ with $i_1 < \ldots < i_k$, 
$\sigma_I$ denotes the permutation of size $k$ obtained by normalization of the sequence $\sigma_{i_1} \ldots \sigma_{i_k}$. 
\end{defi}

\begin{defi}
A permutation $\pi$ is a \emph{pattern} of a permutation $\sigma$ if and only if 
there exists a subset $I$ of $\{1, \ldots, |\sigma|\}$ such that $\sigma_I = \pi$. 
We also say that $\sigma$ \emph{contains} or \emph{involves} $\pi$, and we write $\pi \preccurlyeq \sigma$.
A permutation $\sigma$ that does not contain $\pi$ as a pattern is said to {\em avoid} $\pi$.
\end{defi}

\begin{example}
The permutation $\sigma=316452$ contains the pattern $2431$ whose occurrences are $3642$ and $3652$.
But $\sigma$ avoids the pattern $2413$ as none of its subsequences of length $4$ is order-isomorphic to $2413$.
\end{example}

The pattern containment relation $\preccurlyeq$ is a partial order on permutations,
and permutation classes are downsets under this order.
In other words:
\begin{defi}
A set $\C$ of permutations is a \emph{permutation class} if and only
if for any $\sigma \in \C$, if $\pi \preccurlyeq \sigma$, then we also have $\pi \in \C$.
\end{defi}

Throughout this article, we take the convention that a permutation class only contains permutations of size $n \geq 1$, 
\emph{i.e.}, $\varepsilon \notin \C$ for any permutation class $\C$. 

\medskip

Every permutation class $\C$ can be characterized by a unique antichain $B$
({\em i.e.,} a unique set of pairwise incomparable elements)
such that a permutation $\sigma$ belongs to $\C$ if and only if it avoids every pattern in $B$ (see for example~\cite{AA05}).
The antichain $B$ is called the {\it basis} of $\C$, and we write $\C = Av(B)$.
The basis of a class $\C$ may be finite or infinite;
it is described as the permutations that do not belong to $\C$ and that are minimal in the sense of $\preccurlyeq$ for this criterion.

\subsection{Simple permutations and substitution decomposition of permutations}

The description of permutations in the framework of constructible structures (see Section~\ref{ssec:ConstructibleClasses}) that will be used in this article 
relies on the substitution decomposition of permutations.
Substitution decomposition is a general method,
adapted to various families of discrete objects~\cite{MoRa84},
that is based on core items and relations, and in which every object can be recursively decomposed into core objects using relations.
In the case of permutations, the core elements are simple permutations and the relations are substitutions. 

\begin{defi}
An \emph{interval} of a permutation $\sigma$ of size $n$ is a non-empty subset $\{i,\ldots ,(i+\ell-1)\}$ of consecutive integers of $\{1,\ldots ,n\}$
whose images by $\sigma$ also form a set of consecutive integers. 
The \emph{trivial} intervals of $\sigma$ are $\{1\}, \ldots, \{n\}$ and $\{1,\ldots, n\}$. 
The other intervals of $\sigma$ are called \emph{proper}. 
\label{defn:interval}
\end{defi}

\begin{defi}
\label{def:block}
A \emph{block} (resp.\ \emph{normalized block}) of a permutation $\sigma$ is 
any sequence $\sigma_{i_1} \ldots \sigma_{i_m}$ (resp.\ any permutation $\sigma_I$) for $I=\{i_1, \ldots, i_m\}$ an interval of $\sigma$. 
\end{defi}

\begin{defi}
A permutation $\sigma$ is \emph{simple} when it is of size at least $4$ and it contains no interval,
except the trivial ones. 
\label{defn:simplepermutation}
\end{defi}

Note that no permutation of size $3$ has only trivial intervals (so that the condition on the size is equivalent to ``at least $3$'').

\begin{rem}
The permutations $1$, $12$ and $21$ also have only trivial intervals, 
and are considered simple in many articles. 
Nevertheless, for our notational convenience in this work, we prefer to consider that they are \emph{not} simple. 
\end{rem}

For a detailed study of simple permutations, in particular from an enumerative point of view, we refer the reader to~\cite{AA05,AAK03,Bri08}.
Let us only mention that the number of simple permutations of size $n$ is asymptotically equivalent to $\frac{n!}{e^2}$ as $n$ grows.

\medskip

Let $\sigma$ be a permutation of size $n$ and $\pi^{1},\ldots, \pi^{n}$ be $n$ permutations of size $p_1, \ldots, p_n$ respectively.
Define the \emph{substitution} $\sigma[\pi^{1}, \pi^{2} ,\ldots, \pi^{n}]$ of $\pi^{1},\pi^{2} , \ldots, \pi^{n}$ in $\sigma$ (also called \emph{inflation} in~\cite{AA05})
to be the permutation obtained by concatenation of $n$ sequences of integers $S^1, \ldots , S^n$ from left to right,
such that for every $i,j$, the integers of $S^i$ form a block, 
are ordered in a sequence order-isomorphic to $\pi^{i}$,
and $S^i$ consists of integers smaller than $S^j$ if and only if $\sigma_i < \sigma_j$. 
The interested reader may find a formal definition in~\cite[Definition 0.25]{theseAdeline}.
When a permutation $\tau$ may be written as $\tau = \sigma[\pi^{1}, \pi^{2} ,\ldots, \pi^{n}]$, 
we also say that $\sigma[\pi^{1}, \pi^{2} ,\ldots, \pi^{n}]$ provides a {\em block decomposition} of $\tau$.

\begin{example}
The substitution $ 1\, 3\, 2 [2\, 1, 1\, 3\, 2, 1]$ gives the permutation $ 2\, 1\, \, 4\, 6\, 5\,  \, 3$. 
In particular,  $ 1\, 3\, 2 [2\, 1, 1\, 3\, 2, 1]$ is a block decomposition of $2\, 1\, 4\, 6\, 5\, 3$. 
\end{example}

When substituting in $\sigma=12$ or $21$, we often use $\oplus$ (resp.\ $\ominus$) to denote the permutation $12$ (resp.\ $21$). 

\begin{defi}
  A permutation $\pi$ is \emph{$\oplus$-indecomposable}
  (resp.\ \emph{$\ominus$-indecomposable}) if it cannot be written as
  $\oplus[\pi^{1},\pi^{2}]$ (resp.\ $\ominus[\pi^{1},\pi^{2}]$).
\end{defi}

Simple permutations, together with $\oplus$ and $\ominus$, 
are enough to describe all permutations through their \emph{substitution decomposition}:

\begin{theo}[Proposition 2 of~\cite{AA05}]
Every permutation $\pi$ of size $n$ with
  $n \geq 2$ can be uniquely decomposed as either:
\begin{itemize}
\item $\oplus[\pi^{1},\pi^{2}]$, with $\pi^{1}$ $\oplus$-indecomposable,
\item $\ominus[\pi^{1},\pi^{2}]$, with $\pi^{1}$ $\ominus$-indecomposable,
\item $\sigma[\pi^{1},\pi^{2},\ldots,\pi^{k}]$ with $\sigma$ a simple permutation of size $k$.
\end{itemize}
\label{thm:decomp_perm_AA05}
\end{theo}

\begin{rem}
The simple permutation $\sigma$ in the third item of Theorem~\ref{thm:decomp_perm_AA05} is a pattern of the permutation $\pi$.
Hence, as soon as $\pi$ belongs to some permutation class $\C$, then so does $\sigma$.
\label{rem:simple_pattern}
\end{rem}

Theorem~\ref{thm:decomp_perm_AA05} provides the first step in the decomposition of a permutation $\pi$.
To obtain its full decomposition, we can recursively decompose the permutations $\pi^{i}$ in the same fashion,
until we reach permutations of size $1$.
This recursive decomposition can naturally be represented by a tree,
that is called the substitution decomposition tree (or {\em decomposition tree} for short) of $\pi$.

\begin{defi}
\label{defn:deccompositionTrees}
The \emph{substitution decomposition tree} $T$ of a permutation $\pi$ is
the unique ordered tree encoding the substitution decomposition of $\pi$,
where each internal node is either labeled by $\oplus,\ominus$ -- those nodes are called {\em linear} --
or by a simple permutation $\sigma$ -- {\em prime} nodes.
\end{defi}

Note that in decomposition trees, 
linear nodes are always binary, and the left child of a node labeled by $\oplus$ (resp.\ $\ominus$) 
may not be labeled  $\oplus$ (resp.\ $\ominus$),
since $\pi^{1}$ is $\oplus$-indecomposable (resp.\ $\ominus$-indecomposable)
in the first (resp.\ second) item of Theorem~\ref{thm:decomp_perm_AA05}. 

\begin{example}
The permutation $\pi = 6\,9\,8\,7\,3\,11\,5\,4\,10\,17\,1\,2\,14\,16\,13\,15\,12$ can be recursively decomposed as 
\begin{eqnarray*}
        \pi     & =& 2413[476519328,1,12,35241]\\ 
                & =&2413[31524[\oplus[1,\ominus[1,\ominus[1,1]]],1,1,\ominus[1,1],1]],1,\oplus[1,1],\ominus[2413[1,1,1,1],1]]
\end{eqnarray*}
and its decomposition tree is given in Figure~\ref{fig:tree}.
\end{example}

\begin{figure}[htbp]
\begin{center}
\begin{tikzpicture}[
    level/.style={sibling distance=20mm/#1},
   edge from parent/.style={very thick,draw=black!70},
    simple/.style={rectangle, draw=none, rounded corners=1mm, fill=white, text centered, text=black,anchor=north,inner sep=5pt},
    linear/.style={circle, draw=none, fill=white, text centered, anchor=north, text=black,inner sep=2pt},
    every node/.style={circle, draw=none, fill=black, text centered, anchor=north, text=white,inner sep=0},
    level distance=6mm
]
\node[simple] {$2\,4\,1\,3$}
	child { node[simple] {$3\,1\,5\,2\,4$}
		child {node[linear] {$\oplus$}
			child { node { ~ }}
			child{node[linear] {$\ominus$}
				child { node { ~ }}	
				child{node[linear] {$\ominus$}
				child { node { ~ }}	
				child { node { ~ }}
				}
			}
		}
		child { node { ~ }}
		child { node { ~ }}
		child{node[linear] {$\ominus$}
			child { node { ~ }}	
			child { node { ~ }}	
		}
		child { node { ~ }}
	}
	child {node { ~ }}
	child {node[linear] {$\oplus$}
		child { node { ~ }}
		child { node { ~ }}
	}
	child {node[linear] {$\ominus$}
	   child {node[simple] {$2\,4\,1\,3$}
			child { node { ~ }}
			child { node { ~ }}
			child { node { ~ }}
			child { node { ~ }}
		}
		child {node { ~ }}
};
\end{tikzpicture}
\caption{Decomposition tree of $\pi = 6\ 9\ 8\ 7\ 3\ 11\ 5\ 4\ 10\ 17\ 1\ 2\ 14\ 16\ 13\ 15\ 12$.}\label{fig:tree}
\end{center}
\end{figure}

\begin{defi}
The substitution closure $\hat{\C}$ of a permutation class $\C$ is defined as
$\cup_{k\geq 1} \C^k$ where $\C^1 = \C$ and 
$\C^{k+1} = \{\sigma[\pi^{1}, \ldots , \pi^{n}] \mid \sigma \in \C \textrm{ and } \pi^{i} \in \C^k \textrm{ for any } i \textrm{ from } 1 \textrm{ to } n = |\sigma| \}$.
\end{defi}

Because simple permutations contain no proper intervals, we have:
\begin{rem}
For any class $\C$, the simple permutations in $\wc$ are exactly the simple permutations in $\C$.
\label{rem:simple_in_WC}
\end{rem}

Consequently, for any permutation class $\C$,
this allows to describe $\wc$ as the class of all permutations whose decomposition trees can be built on the set of nodes $\{\oplus,\ominus\}\cup \SC$,
where $\SC$ denotes the set of simple permutations in $\C$
(if $12$ and $21$ belong to $\C$; otherwise we have to remove $\oplus$ or $\ominus$ from the set of nodes).

\begin{defi}
A permutation class $\C$ is \emph{substitution-closed} if $\C = \hat{\C}$,
or equivalently if for every permutation $\sigma$ of $\C$,
and every permutations $\pi^{1}, \pi^{2}, \ldots  ,\pi^{n}$ of $\C$ (with $n= |\sigma|$),
the permutation $\sigma[\pi^{1}, \pi^{2} ,\ldots, \pi^{n}]$ also belongs to $\C$.
\end{defi}

Like before, a substitution-closed permutation class can therefore be seen as the set of decomposition trees built on the set of nodes
$\{\oplus,\ominus\}\cup \SC$ (if $12$ and $21$ belong to $\C$; but otherwise $\C$ is trivial and has at most one permutation of each size).

\begin{rem}
In~\cite{AA05}, it is proven that substitution-closed permutation classes can be characterized as the permutation classes $Av(B)$
whose basis $B$ contains only simple permutations (or maybe $12$ or $21$ for trivial classes).
\label{rem:characterization_substitution_closed}
\end{rem}

\subsection{From combinatorial specifications to generating functions and random samplers}
\label{ssec:ConstructibleClasses}

Let us leave aside permutations for now, and review some basics of the symbolic method about constructible structures and their description by combinatorial specifications. 
We will see in Theorem~\ref{thm:systemeNonAmbigu} (p.\pageref{thm:systemeNonAmbigu}) that the classes of permutations 
we are interested in fit in this general framework.

A class $\C$ of combinatorial structures is a set of discrete objects equipped with a notion of \emph{size}: 
the size is a function of $\C \rightarrow \mathbb{N}$ denoted $|\cdot|$ such that 
for any $n$ the number of objects of size $n$ in $\C$ is finite. 

Among the combinatorial structures, we focus on {\em constructible} ones, from the framework introduced in~\cite{FlSe09}. 
Basically, a constructible combinatorial class is a set of structures that can be defined from atomic structures of size $1$ (denoted by $\cal{Z}$), 
possibly structures of size $0$ (denoted by $\cal{E}$), 
and assembled by means of \emph{admissible} constructors. 
While a complete list of these combinatorial constructors is given in~\cite{FlSe09}, we only use a (small) subset of them: 
the disjoint union, denoted by~$\uplus$ or~$+$ (we may also use the notation~$\sum$), to choose between structures;  
and the Cartesian product, denoted by~$\times$, to form pairs of structures. 
More formally, a constructible combinatorial class is one that admits a combinatorial specification.

\begin{defi}
A {\em combinatorial specification} for a combinatorial class $\cal{C}_1$ is an equation or a system of equations of the form 
\[
\begin{cases}
	  \cal{C}_1&\!\!=~\cal{H}_1(\cal E, \cal Z,\cal{C}_1,\cal{C}_2,\dots,\cal{C}_m),	\\
   	  \cal{C}_2&\!\!=~\cal{H}_2(\cal E,\cal Z,\cal{C}_1,\cal{C}_2,\dots,\cal{C}_m),	\\ 		
	 &\phantom{=H_k}\vdots                 	\\
   	  \cal{C}_m&\!\!=~\cal{H}_m(\cal E,\cal Z,\cal{C}_1,\cal{C}_2,\dots,\cal{C}_m),
\end{cases}
\]
where each $\cal{H}_i$ denotes a term built from $\cal{C}_1, \dots,
\cal{C}_m, \cal Z$ and $\cal E$ using admissible
constructors.
\end{defi}

For example, the equation~$\cal{I}=\cal{E}+\cal{Z}\times\cal{I}$
describes a class $\cal{I}$ whose elements are finite sequences of
atoms.

In this framework, the {\em size} of a combinatorial structure is its
number of atoms ($\Z$) and from there, combinatorial structures can be
counted according to their size.  The size information for a whole
combinatorial class, say $\cal{C}$, is encoded by its {\em ordinary
generating function}\footnote{We do not use \emph{exponential} but \emph{ordinary} generating functions to count pattern-avoiding permutations. 
First, note that the corresponding exponential generating functions would have infinite radii of convergence, 
pattern-avoiding permutations of size $n$ being always less than $c^n$ for some constant $c$~\cite{MaTa04}. 
The use of ordinary generating functions is moreover very natural since our work is based on an encoding of permutations by trees built on a finite set of nodes.}, 
which is the formal power series~$C(z)=\sum_{n\geq 0}c_nz^n$ where the coefficient $c_n$
is the number of structures of size~$n$ in~$\cal{C}$.
Note that we also have $C(z)=\sum_{\pi \in \cal{C}}z^{|\pi|}$. 

Combinatorial specifications of combinatorial
classes may be automatically translated into systems defining their
generating function (possibly implicitly).  This system is obtained by
means of a dictionary that associates an operator on generating
functions to each admissible constructor.  The complete dictionary is
given in~\cite{FlSe09}, together with the proof that this translation
from constructors of combinatorial classes to operators on their
generating functions is correct.  Here, we only use the constructors
disjoint union and Cartesian product, which are respectively
translated to sum and product of generating functions.

A lot of information can be extracted from such functional systems; in
particular, one can compute as many coefficients of the series as
required,  and~\cite{FlSe09} provides many tools to get asymptotic
equivalents for these coefficients.

Combinatorial specifications may also be automatically translated into
uniform random samplers of objects in the class described by the
specification.  Indeed, a specification can be seen as a (recursive) procedure to
produce combinatorial objects, and randomizing the choices made during
this procedure transforms the specification into a random sampler.  To
ensure that such random samplers are uniform (\emph{i.e.} that for any
$n$, two objects \emph{of the same size} $n$ have the same probability
of being produced), two methods have been developed: the recursive
method~\cite{FlZiVC94} and the Boltzmann method~\cite{DuFlLoSc04}.  In
the first one, the coefficients of the generating functions are used
for the probabilistic choices to ensure uniformity, making this method
well-adapted for generating a large sample of objects of relatively
small size: this requires to compute only once a relatively small number of coefficients. 
The focus of the second one is to achieve efficiently  the generation of very
large objects, with a small tolerance (of a few percents) allowed on
their size.  Coefficients of the generating functions are not needed in
Boltzmann samplers, but rather the generating functions themselves.
More precisely, the value of the generating function at a given point
needs to be computed, and this is solved in~\cite{PiSaSo2010}. 

\section{Our results in existing context}
\label{sec:detailed_intro}

\subsection{Our contributions, and comparison with \cite{AA05,BHV08a}}
\label{subsec:QCS}

The goal of the present work is to solve algorithmically a combinatorial problem on permutation classes: 
computing in an automatic way a combinatorial specification for any given class, under some conditions specified below. 
We first have to determine how to describe the permutation class~$\C$ in input.
In all what follows, we will suppose that~$\C$ is given by its basis $B$ of excluded patterns and the set~$\SC$ of simple permutations in~$\C$, assuming that both these sets are finite. 
Note that from~\cite[Theorem 9]{AA05} the basis of $\C$ is necessarily finite when $\C$ contains finitely many simple permutations. 
On the other hand, from~\cite{BRV08,BBPR14}, it is enough to know~$B$ to decide whether~$\SC$ is finite and (in the affirmative) to compute~$\SC$~\cite{PR11}. 

Our work is a continuation of the main result (Theorem 10) of~\cite{AA05}: 
every permutation class containing finitely many simples has an algebraic generating function. 
The main step in the proof of this result is to construct a system of combinatorial equations describing~$\C$ using the substitution decomposition, 
and more precisely by propagation of pattern avoidance constraints in the decomposition trees. 
Although the proof is in essence constructive, there is still some work to be done to fully automatize this process. 
In Section~\ref{sec:ambiguous}, we review their method, going deeper in the details of the construction. 
This allows us to bring their methodology to a full algorithm
-- see Algorithm~\textsc{AmbiguousSystem} (p.\pageref{alg:sys-ambigu})
and Theorem~\ref{thm:systemeAmbigu} (p.\pageref{thm:systemeAmbigu}).

It is important to note that the description of~$\C$ obtained in this way is not a combinatorial specification, since it is {\em a priori} ambiguous 
(that is to say, unions are not necessarily disjoint). 
Nevertheless, as it is done in~\cite{AA05}, this ambiguous system can be used 
for proving the algebraicity of the generating function of~$\C$, and even allows its computation (or implicit determination, in less favorable cases): 
it is enough to apply the inclusion-exclusion principle. 
The advantage of a specification over an ambiguous system will be discussed in Section~\ref{subsec:chain} (fifth step, p.\pageref{fifth_step}). 

\medskip

The algebraicity result of~\cite{AA05} was re-proved in~\cite{BHV08a} and extended
to generating functions of some subsets of classes~$\C$ with finitely many simples: 
the alternating permutations in~$\C$, the even ones, the involutions in~$\C$,\,... 
For our purpose, those extensions are less important than the alternative proof of the main result of~\cite{AA05}: 
indeed, this second proof describes a method to build a combinatorial specification for~$\C$. 
Essential to this proof are \emph{query-complete sets}, whose definition we recall. 

\begin{defi}
A \emph{property} is any set $P$ of permutations. 
A permutation $\pi$ is said to satisfy $P$ when $\pi \in P$. 
A set $\cP$ of properties is {\em query-complete} if,
for every simple permutation \si (and also for $\si =\oplus$ or $\ominus$) 
and for every property $P \in \cP$, 
it can be decided whether $\si[\alpha_1 , \dots , \alpha_{|\sigma|}]$ satisfies $P$
knowing only which properties of $\cP$ are satisfied by each $\alpha_i$. 
\end{defi}

The proof of the main result (Theorem 1.1) of~\cite{BHV08a} shows that a combinatorial specification for a permutation class~$\C$ with finitely many simples
can be obtained from any finite query-complete set $\cP$ such that $\C \in \cP$. 
More precisely, this combinatorial specification consists of three types of equations, described below
(this is reduced to two types by plugging the third one into the second one).
Recall that $\SC$ denotes the set of simple permutations in $\C$. 
Note that all unions below are finite, since $\cP$ and $\SC$ are finite by assumption. 
\begin{itemize}
 \item First, $\C$ is written as the disjoint union 
\begin{equation*}
\C = \uplus \, \C_{\X}, 
\end{equation*}
where the union runs over all subsets $\X$ of $\cP$ containing $\C$ 
with $\C_{\X}$ denoting the set of permutations that satisfy every property in $\X$ 
and do not satisfy any property in $\cP \setminus \X$.
 \item Second, for any such set $\X$, the substitution decomposition allows to write 
\begin{equation*}
\C_\X = 1_\X \uplus \C_\X^{\oplus} \uplus \C_\X^{\ominus} \uplus \biguplus_{\sigma \in \s_\C} \C_\X^{\sigma} 
\end{equation*}
where $1_\X$ is either the set $\{1\}$ if the permutation $1$ belongs to $\C_\X$, or the empty set otherwise, 
and where $\C_\X^{\sigma}$ is the subset of $\C_\X$ of permutations whose decomposition tree has root $\sigma$. 
 \item And third, for $\sigma \in \{\oplus,\ominus\} \cup \s_\C$ and $\X$ as above, the fact that $\cP$ is query-complete allows to express $\C_\X^{\sigma}$ as 
\begin{equation*}
\C_\X^{\sigma} = \biguplus \si[\C_{\X_1}, \dots, \C_{\X_m}] 
\end{equation*}
where the union is over the set $E_{\X,\si}$ of all $m$-uples $(\X_1, \dots, \X_m)$ of subsets of $\cP$ such that 
if, for every $i \in [1..m]$, it holds that $\alpha_i \in \C_{\X_i}$ 
then $\si[\alpha_1, \dots, \alpha_m] \in \C_\X$. 
In the case where $\sigma = \oplus$ (resp.\ $\ominus$), 
to ensure uniqueness of the decomposition, we further need to enforce that $\X_1$ contains the property of being $\oplus$-indecomposable (resp.\ $\ominus$-indecomposable). 
W.l.o.g., we can assume that these properties are in $\cP$. 
\end{itemize}
Note that the number of equations in the specification obtained depends exponentially on the size of $\cP$,
since there is at least one equation for each subset of $\cP$ containing $\C$.
Similarly,
the number of terms of the union defining some $\C_\X^{\sigma}$ may be exponential in the size of $\cP$,
since the union is over $m$-uples of subsets of $\cP$ (with $m = |\sigma|$).

We point out that the above specification is not fully explicit (even assuming that $\cP$ is given),
since there is no explicit description of the sets $E_{\X,\si}$. 
As explained in the proof of Lemma 2.1 of~\cite{BHV08a}, the sets $E_{\X,\si}$ 
may be described using so-called \emph{lenient inflations}, which are intimately linked with the \emph{embeddings by blocks} of~\cite{AA05}. 
But neither \cite{BHV08a} nor \cite{AA05} discuss their effective computation. 
We will return to this problem later in this section. 

For any permutation class~$\C = Av(B)$ with finitely many simples,
the authors of~\cite{BHV08a} provide a finite query-complete set that contains $\C$,
and conclude that there is a combinatorial specification for any such $\C$.
More precisely, the class $\C$ being described as $\cap_{\beta \in B} Av(\beta)$, 
they rather consider separately every principal class $Av(\beta)$ for all $\beta \in B$, 
and define a finite query-complete set  $\cP_\beta$ containing it. This is essentially their Lemma 2.1. 
The query complete set associated with $\C$, denoted $\cP_\C$, is then obtained taking the union of all $\cP_\beta$. 
It consists of the following properties: 
the set of $\oplus$-indecomposable permutations, 
the set of $\ominus$-indecomposable permutations, 
and the set $Av(\rho)$ for every permutation $\rho$ which is a pattern of some $\beta \in B$. 
Thus $\cP_\C$ is often a big set.

It should be noticed that the query-complete sets that are used in the examples of~\cite[Section 4]{BHV08a} 
are however strictly included in the set $\cP_\C$. 
These smaller query-complete sets are better, since they result in specifications with fewer equations and unions having fewer terms
than the ones that would be obtained applying to the letter the specializations of the proofs. 
But there are no indications in~\cite{BHV08a} on how these smaller query-complete sets were computed, 
nor on how this could generalize to other examples. 
It should be noticed that in the examples of~\cite{BHV08a},
the class is either substitution-closed, or contains no simple permutations.

\smallskip

To summarize, the proof of the main result of~\cite{BHV08a}
gives a general method to compute a specification for a permutation class having finitely many simple permutations.
But there is still some work to be done to fully automatize this process,
and the specification obtained would be very big in general.
Moreover, an algorithm using this method would have a lot of computations to do (the computation of all the sets $E_{\X,\si}$).
On the other hand, the \emph{ad hoc} constructions of the examples of~\cite{BHV08a} show that using the specificity of a particular permutation class,
it is possible to obtain shorter specifications with fewer computations.

\medskip

Our main contribution is to give an algorithm to compute a specification for any permutation class having finitely many simples,
using a different approach --
see Algorithm \textsc{Specification} (p.\pageref{alg:final})
and Theorem~\ref{thm:systemeNonAmbigu} (p.\pageref{thm:systemeNonAmbigu}).
Our method is general
(unlike the \emph{ad hoc} methods used in the examples of~\cite{BHV08a})
but nevertheless uses the specificity of the permutation class given in input (unlike the method described in the proofs of~\cite{BHV08a})
in order to do less computations and to have fewer equations in the specification.

Even if the method we use and the specification we produce are not exactly the ones presented by~\cite{BHV08a}
and reviewed above, they have some similarities. When considering the partition of $\C$ into $\uplus \, \C_{\X}$ shown above, 
\cite{BHV08a} is looking at a very fine level of details, where a coarser level could be enough. 
With our method, we consider the partition of $\C$ which is the coarsest possible to allow the derivation of a specification. 
In practice, if two sets $\X$ and $\X'$ are such that $\C_{\X}$ and $\C_{\X'}$ appear as $\C_{\X} \uplus \C_{\X'}$ everywhere 
in the specification resulting from~\cite{BHV08a}, 
our specification will have only one term instead of these two, representing $\C_{\X} \uplus \C_{\X'}$. 
Of course, this holds for unions with more terms as well. 
Considering fewer sets $\X$ results in fewer equations in the specification, fewer terms in the unions and fewer sets $E_{\X,\si}$. Since the computation of the $E_{\X,\si}$ amounts to computing the specification from the query-complete set, 
the algorithmic complexity for computing the specification with our method is hereby reduced, 
compared to what a  formalized algorithm of the approach of~\cite{BHV08a} would give. 

Concretely, our algorithm computes a query-complete set, the sets $E_{\X,\si}$ associated, and the specification in parallel, whereas the method of~\cite{BHV08a} is to first compute a query-complete set and then deduce a specification.
Note that in our presentation, the result of our algorithm is only the specification, but
it contains also implicitly the description of the query-complete set and of the $E_{\X,\si}$.

To obtain the announced coarsest partition of $\C$, 
and the subsequent specification with as few equations as possible that it yields, 
we proceed as follows. 
We use the same guideline as in~\cite{AA05} for computing a possibly ambiguous combinatorial system describing \C 
(however making this approach effective): the essential idea is to use
the substitution decomposition and to propagate pattern avoidance constraints in the decomposition trees.
We get rid of the ambiguity by introducing complement sets, but only when they are needed
(the method in the proofs of~\cite{BHV08a} can be seen as somehow introducing all complement sets at once). 
In practice, it means that we are not only propagating avoidance constraints in decomposition trees, but also containment constraints.
It will be clear in Sections~\ref{sec:ambiguous} and~\ref{sec:disambiguation} that 
even though the purposes of those two types of constraint are opposite, the ways to propagate them are very similar, 
an essential step being the effective computation of the embeddings/lenient inflations mentioned in~\cite{AA05,BHV08a}. 
The method used to explicitly determine which avoidance/containment constraints are necessary 
and to effectively propagate them in the trees 
is completely new with respect to~\cite{AA05,BHV08a}. 

\smallskip

To conclude on our contributions compared with those of~\cite{AA05,BHV08a}, 
our work describes how to obtain a combinatorial specification for any class having finitely many simple permutations.
Contrary to~\cite{AA05,BHV08a}, our work is fully algorithmic.
Moreover, we develop a method allowing to have fewer equations in the specification and to have a better efficiency
compared to what a formalized algorithm of the approach of~\cite{BHV08a} would give.

\subsection{An algorithmic chain from $B$ to random permutations in $Av(B)$}
\label{subsec:chain}

Our main result (that is, the algorithmic computation of specifications for permutation classes with finitely many simples) 
can and should be viewed in the context of other recent algorithms from the literature. 
Together, they provide a full algorithmic chain starting with the
finite basis $B$ of a permutation class $\C$, and computing a 
specification for $\C$, from which it is possible to sample permutations in \C uniformly at random. 
Figure~\ref{fig:schema2} shows an overview of this algorithmic chain, 
and we present its main steps below. 
Note that this procedure may fail to compute its final result, 
namely when \C contains an infinite number of simple permutations, this condition being tested algorithmically. 

\begin{figure}[htbp]
\includegraphics[width=\textwidth]{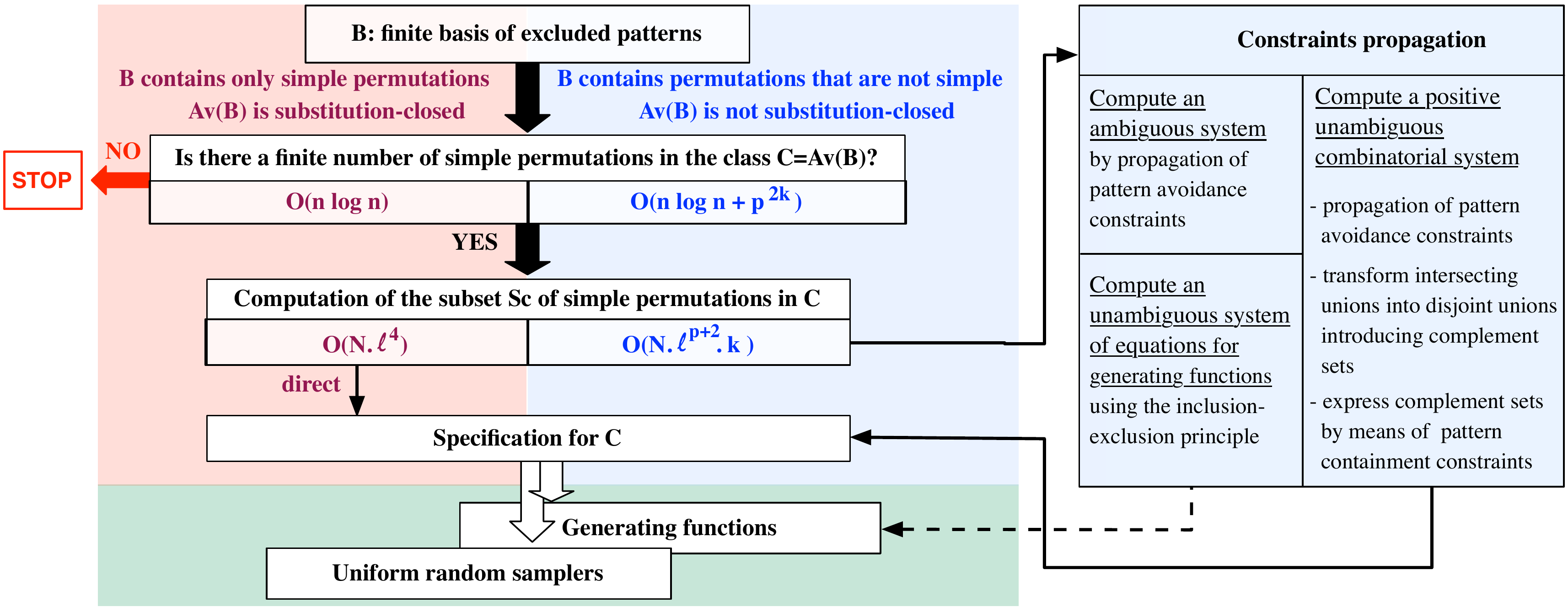}
\vspace{-3mm}
\caption
  [The full algorithmic chain starting from the basis $B$ of a permutation class~$\C$]
  {The full algorithmic chain starting from the basis $B$ of a permutation class~$\C$,
    with complexities given w.r.t.~$n = \sum_{\beta \in B} |\beta|$, $k = |B|$, 
    $p = \max \{|\beta| : \beta \in B\}$, $N = |\SC|$ and $\ell = \max \{|\pi| : \pi \in \SC\}$ 
    where $\SC$ is the set of simple permutations of~$\C$.}
\label{fig:schema2}
\end{figure}

We have chosen that the permutation class in input of our procedure should be given by its basis~$B$, that we require to be finite. 
This does not cover the whole range of permutation classes, but it is one of the possible ways to give a finite input to our algorithm.
There are of course other finite descriptions of permutation classes, even of some with infinite basis (by a recognition procedure for example). 
The assumption of the description by a finite basis has been preferred for two reasons: 
first, it encompasses most of the permutation classes that have been studied; 
and second, it is a necessary condition for classes to contain finitely many simple permutations (see~\cite[Theorem 9]{AA05})
and hence for our algorithm to succeed. 

\smallskip

\noindent \textbf{First step: Finite number of simple permutations. \\}
First, we check whether $\C = Av(B)$ contains only a finite number of
simple permutations.  
This is achieved using algorithms of
\cite{BBPR09} when the class is substitution-closed and of
\cite{BBPR14} otherwise. The complexity of these algorithms is
respectively $\mathcal{O}(n \log n)$ and $\mathcal{O}(n \log n + p^{2k})$, where
$n = \sum_{\beta \in B} |\beta|$, $p = \max \{|\beta| : \beta \in B\}$ and $k = |B|$.

\medskip

\noindent \textbf{Second step: Computing simple permutations. \\}
The second step of the algorithm is the computation of the set of
simple permutations $\SC$ contained in
$\C = Av(B)$, when we know it is finite. Again, when
$\C$ is substitution-closed, $\SC$ can
be computed by an algorithm that is more efficient than in the
general case. The two algorithms are described in \cite{PR11}, and
their complexity depends on the output: $\mathcal{O}(N \cdot
\ell^{p+2}\cdot k)$ in general and $\mathcal{O}(N \cdot \ell^{4})$ for
substitution-closed classes, with $N = |\SC|$,
$p = \max \{|\beta| : \beta \in B\}$, 
$\ell = \max \{|\pi| : \pi \in \SC\}$ 
and $k=|B|$.

\medskip

In the case of substitution-closed classes, the set of simple permutations in \C gives an immediate access to a specification for $\C$
-- see \cite{AA05} or Theorem~\ref{thm:wc} (p.\pageref{thm:wc}). 
In the general case, finding such a specification is the algorithmic problem that we address in this article. 

\medskip

\noindent \textbf{Third step: Computing a combinatorial specification. \\}
This corresponds to the computation of a combinatorial specification for \C
by propagation of pattern constraints and disambiguation of the equations, 
as briefly presented in Section~\ref{subsec:QCS}
and described in details in Sections~\ref{sec:ambiguous} and \ref{sec:disambiguation}. 

\medskip

From a combinatorial specification for \C, that we may obtain algorithmically as described above, there are two natural algorithmic continuations (which we have reviewed in Section~\ref{ssec:ConstructibleClasses}):

\medskip

\noindent \textbf{Fourth step: Computing the generating function $C(z)$ of \C. \\}
With the \emph{dictionary} of~\cite{FlSe09}, a system of equations defining $C(z)=\sum_{n\geq 0}c_nz^n$ is immediately deduced from the specification.
Because our specification involves only disjoint unions and Cartesian products, the resulting system is positive and algebraic.
In some favorable cases, this system may be solved for $C(z)$ explicitly. 
Even if it is not the case, many information may still be derived from the system, in particular about the coefficients $c_n$ or the growth rate of the class. 

\medskip

\noindent \textbf{An alternative for the computation of $C(z)$. \\}
As explained in~\cite{AA05} and reviewed earlier in this paper, 
it is also possible to obtain such a system of equations for $C(z)$ from an \emph{ambiguous} system describing $\C$, applying the inclusion-exclusion principle. 
In this case, the obtained system is algebraic but with negative terms in general.

\medskip

\noindent \textbf{Fifth step: Random sampling of permutations in \C.\\}\label{fifth_step}
In Section~\ref{ssec:ConstructibleClasses}, we have reviewed the principles that allow, in the same fashion as the dictionary of~\cite{FlSe09}, 
to translate the combinatorial specification for \C into uniform random samplers of permutations in \C. 
Remark that this translation is possible only with a specification, \emph{i.e.} a \emph{positive unambiguous} system describing \C. 
Indeed, whereas adapted when considering generating functions (where subtraction is easily handled), 
the inclusion-exclusion principle cannot be applied for random generation (since ``subtracting combinatorial objects'' is not an option in a procedure to produce them). 

\smallskip

To illustrate that this algorithmic chain is effective, we present in Section~\ref{sec:ex} how our algorithms run on examples. 
We also show some observations that are produced through it in Figures~\ref{fig:exemple900} and~\ref{fig:exemple_random} below. 
These figures have been obtained with a prototype implementing our algorithms, that we hope to make available for use by others in the future\footnote{A Boltzmann sampler for substitution-closed classes is already available here: \url{http://igm.univ-mlv.fr/~pivoteau/Permutations/}. The implementation is in Maple, an example of use is given in the worksheet.}.

\begin{figure}[htbp]
\begin{center}
\scalebox{0.35}{\input{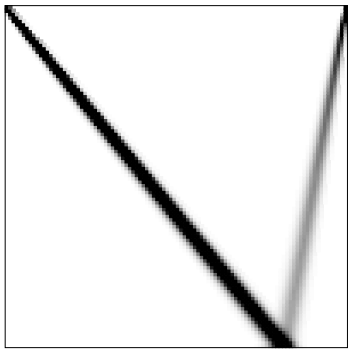}}
\end{center}
\caption{The shape of permutations in the (not substitution-closed) class $\C$ of Section~\ref{ssec:grand_ex}. }\label{fig:exemple900}
\end{figure}

Figure~\ref{fig:exemple900} shows the ``average diagram'' of a permutation in the class \C (not substitution-closed) studied in Section~\ref{ssec:grand_ex} (p.\pageref{ssec:grand_ex}). 
The \emph{diagram} of a permutation $\sigma$ is the set of points in the plane at coordinates $(i,\sigma_i)$, and 
the picture in Figure~\ref{fig:exemple900} is obtained by drawing uniformly at random $30\,000$ permutations of size $500$ in \C, 
and by overlapping their diagrams -- the darker a point $(x,y)$ is, the more of these permutations have $\sigma(x)=y$.

\begin{figure}[htbp]
\begin{center}
\includegraphics[scale=0.4]{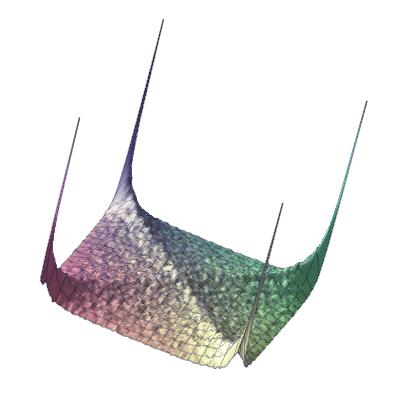}
\includegraphics[scale=0.25]{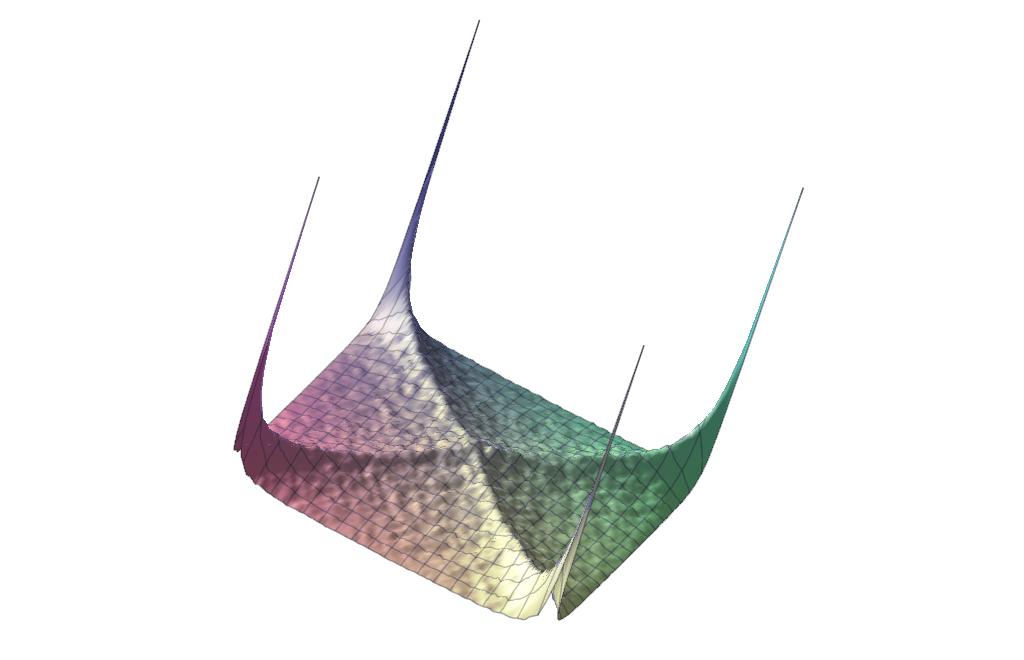}
\end{center}
\caption{The shape of separable permutations (left), and of permutations taken in the substitution-closed class
whose set of simple permutations is $\{2413,3142,24153\}$ (right).}\label{fig:exemple_random}
\end{figure}

Figure~\ref{fig:exemple_random} shows average
diagrams of permutations in the (substitution-closed) class $Av(2413,3142)$ of separable
permutations, and in another substitution-closed class.  These diagrams are
obtained overlapping the diagrams of $10\,000$ permutations of size
$100$ (resp.\ $500$).  The representation is however a little different from
Figure~\ref{fig:exemple900}: in these 3D representations, for a point
at coordinates $(x,y,z)$, $z$ is the number of permutations such that
$\sigma(x)=y$. 
Leaving aside the difference in the representation, these figures suggest a very different limit behavior in substitution-closed and not substitution-closed classes. 

Looking at these diagrams, a natural question is then to describe the average shape of permutations in classes. 
This is a question which has received quite a lot of attention lately, especially for classes $Av(\tau)$ for $\tau$ of size $3$, see~\cite{AtMa14,HoRiSl14,MaPe14,MiPa14}. 
Inspired by Figure~\ref{fig:exemple_random}, some of us (in collaboration with V. Féray and L. Gerin) 
have described the limit shape of separable permutations in~\cite{brownian}, thus explaining the first diagram of Figure~\ref{fig:exemple_random}. 
As we discuss in~\cite{brownian}, we are working on generalizing this result to substitution-closed classes, 
which would also explain the second diagram of Figure~\ref{fig:exemple_random}. 

\subsection{Perspectives}

As described in Section~\ref{subsec:chain}, our main result 
combines with previous works to yield 
an algorithm that produces, for any class~$Av(B)$ containing finitely many simple permutations, 
a recursive (resp.~Boltzmann) uniform random sampler. 
When generating permutations with such samplers, 
complexity is measured w.r.t.~the size of the permutation
produced and is quasilinear (resp.\ quadratic but can be made linear
using classical tricks and allowing a small variation on the size of the
output permutation~\cite{DuFlLoSc04}).  However, the complexity is not
at all measured w.r.t.~the number of equations in the specification
nor w.r.t.~the number of terms in each equation.  
In our context, where the specifications are produced automatically, 
and potentially contain a large number of equations/terms, this
dependency is of course relevant, and opens a new direction in the
study of random samplers.

\medskip

In addition to providing inspiration for the study of random permutations, 
 our algorithmic chain has other applications. 
Indeed, the specifications obtained could also be used to
compute or estimate growth rates of permutation classes.
Moreover,  the computed specifications could possibly be used to provide more efficient algorithms to test
membership of a permutation to a class.

\medskip

We should also mention that our procedure fails to be completely
general.  Although the method is generic and algorithmic, the classes
that are fully handled by the algorithmic process are those containing
a finite number of simple permutations. From~\cite{AA05},
such classes are finitely based.
And since there are countably many such permutation classes, 
only a very small subset of the (uncountably many) permutation classes 
is covered by our method.

But note that even if a class \C contains an infinite number of simple
permutations, we can (at least in theory) use our approach to perform random generation of
permutations of the class \C.
More precisely, fixing the maximum size~$n$ of permutations we want to generate, 
we can apply our algorithm to a class~$\C'\subseteq\C$ containing finitely many simple permutations and that coincides with~$\C$ up to size~$n$.
It is enough to choose~$\C'$ which is the subclass of~$\C$ whose set of simple permutations consists in all simple permutations of~$\C$ of size at most~$n$, 
whose computation is explained in~\cite{PR11}. 
If the maximal size~$n$ is large and if \C has many simple permutations of each size, 
it is likely that the complexity of our algorithm will be too large for it to be of any use.
But this approach may be relevant when the class has an infinite number of simple permutations, but a small number of each size. 

To enlarge the framework of application of our algorithm computing specifications,
we could explore the possibility of extending it to permutation classes that contain an infinite number of simple permutations, but that are finitely described.
A family of such classes is considered in~\cite{ARV15}, 
where the finite basis and algebraicity results of~\cite{AA05} are extended from 
classes with finitely many simple permutations 
to subclasses of substitution closures of geometrically griddable classes (we refer the reader to~\cite{ARV15} for definitions). 
The proofs in~\cite{ARV15} involve similar techniques as in~\cite{BHV08a} (including query-complete sets), but not only. 
In particular, they heavily rely on the results of~\cite{AABRV13}, which are non-constructive. 
Making all the process in~\cite{AABRV13} and~\cite{ARV15} constructive, and then turning it into an effective procedure, 
it may be possible to extend our algorithms to all classes considered in~\cite{ARV15}. 
But this is far from straightforward, and beyond the scope of the present work. 
Note that, with such an improvement, more classes would enter our framework, but it would be hard to leave the algebraic case. 

\section{Combinatorial specification of substitution-closed classes}
\label{sec:substitution-closed}

In this section, we recall how to obtain a combinatorial specification for substitution-closed classes having finitely many simple permutations. 

Recall that we denote by $\Wc$ the substitution closure of the permutation class $\C$, 
and that $\C$ is substitution-closed when {$\C=\Wc$}, 
or equivalently when the permutations in \C
are exactly the ones whose decomposition trees have
internal nodes labeled by $\oplus, \ominus$ or any simple permutation of {$\C$}.

For the purpose of this article, we additionally introduce the following notation:
\begin{defi}\label{def:C+}
For any set $\cA$ of permutations, $\cA^+$ (resp.\ $\cA^-$) denotes the set of permutations of $\cA$ that are $\oplus$-indecomposable 
(resp.\ $\ominus$-indecomposable) and $\mathcal{S_\cA}$ denotes the set of simple permutations of \cA.
\end{defi}

Theorem~\ref{thm:decomp_perm_AA05} (p.\pageref{thm:decomp_perm_AA05}) directly yields the following proposition:

\begin{prop}[Lemma 11 of~\cite{AA05}]\label{prop:sys_wc}
Let $\C=\wc$ be a substitution-closed class\footnote{that contains $12$ and $21$; 
it will be the case until the end of the article and will not be recalled again.}.
Then \wc satisfies the following system of equations, denoted $\Sys_{\wc}$:
	\begin{eqnarray}
        \wc  &=& 1\ \uplus \ \oplus[\wc^+, \wc]\  \uplus \ \ominus[\wc^-, \wc] \ \textstyle\uplus \biguplus_{\pi \in \s_\C} \pi[\wc, \dots, \wc]\label{eqn:Wc1}  \\
        \wc^+ &=& 1\  \uplus \ \ominus[\wc^-, \wc]\  \uplus\  \textstyle\biguplus_{\pi \in \s_\C} \pi[\wc, \dots, \wc] \label{eqn:Wc2}\\
        \wc^- &=& 1\  \uplus\  \oplus[\wc^+, \wc]\  \uplus\  \textstyle\biguplus_{\pi \in \s_\C} \pi[\wc, \dots, \wc]. \label{eqn:Wc3}
        \end{eqnarray}
\end{prop}

Note that by Remark~\ref{rem:simple_in_WC}, $\s_\C= \s_\wc$. 
By uniqueness of the substitution decomposition, unions are disjoint and so Equations~\eqref{eqn:Wc1} to \eqref{eqn:Wc3} describe
unambiguously the substitution-closed class $\Wc$. 
Hence, Proposition~\ref{prop:sys_wc} can be transposed in the framework of constructible structures as follows:

\begin{theo}\label{thm:wc}
Let \C be a substitution-closed class.
Then \C can be described as a constructible combinatorial class in the sense of Section~\ref{ssec:ConstructibleClasses}
with the following combinatorial specification, where the $\cal{E}_\pi$ for $\pi$ in $\mathcal{S_C}$ are distinct objects of size $0$:\footnote{The $\cal{E}_\pi$ are introduced only to distinguish between substitutions in distinct $\pi$ of the same size. 
The term $\cal{E}_\pi \times \C \times \dotsb \times \C$ then corresponds to the classical substitution operation: $\pi[ \, \C,  \dotsb, \C \, ]$}
$$\begin{cases}
\C = \Z + \ \cal{E}_\oplus \times \C^+\times\C \ + \ \cal{E}_\ominus \times \C^-\times\C \ + \ \sum_{\pi \in \mathcal{S_C}} \cal{E}_\pi \times \underbrace{\C \times \dotsb \times \C}_{|\pi|} \\
\C^+ = \Z + \ \cal{E}_\ominus \times \C^-\times\C \ + \ \sum_{\pi \in \mathcal{S_C}} \cal{E}_\pi \times \underbrace{\C \times \dotsb \times \C}_{|\pi|} \\ 
\C^- = \Z + \ \cal{E}_\oplus \times \C^+\times\C \ + \ \sum_{\pi \in \mathcal{S_C}} \cal{E}_\pi \times \underbrace{\C \times \dotsb \times \C}_{|\pi|}. 
\end{cases}$$
\end{theo}

Moreover this system can be translated into an equation for the generating function $C(z)$: 

\begin{prop}[Theorem 12 of~\cite{AA05}]\label{prop:eqn_serie_wc}
Let \C be a substitution-closed class, with generating function $C(z)$. Then
$$C(z)^2 + (S_{\C}(C(z))-1+z)C(z) +S_{\C}(C(z)) +z = 0 $$
with $S_{\C}(z)$ denoting the generating function that enumerate simple permutations in $\C$, i.e. $S_{\C}(z) = \sum_{\pi \in \mathcal{S_C}} z^{|\pi|}$.
\end{prop}

Hence, in the case of a substitution-closed class $\C$ (and for the substitution closure $\wc$ of any class), 
the system $\Sys_{\wc}$ that recursively describes the permutations in $\wc$ 
can be immediately deduced from the set $\mathcal{S_C}$ of simple permutations in $\C$. 
As soon as $\mathcal{S_C}$ is finite and known, this system is explicit and gives a combinatorial specification.

Our next goal is to describe an algorithm that computes a combinatorial system of equations for a general permutation class $\C$ 
from the simple permutations in $\C$, like for the case of substitution-closed classes. 
However, when the class is not substitution-closed, this is not as straightforward as what we have seen in Proposition~\ref{prop:sys_wc}, 
and we provide details on how to solve this general case in the following sections.

\section{A possibly ambiguous combinatorial system for permutation classes}
\label{sec:ambiguous}
\label{sec:addConstraint}

In this section, we explain how to derive a system of equations for a class $\C$ with finitely many simple permutations
from the combinatorial specification of its substitution closure. 
Our method follows the guideline of the constructive proof of~\cite[Theorem 10]{AA05}. 
However, unlike~\cite{AA05}, we make the whole process fully algorithmic. 

The key idea of the method is to describe recursively the permutations in $\C$, replacing the constraint of avoiding the elements of the basis
by constraints in the subtrees of the decomposition tree of permutations in $\C$.
This is done by computing the {\em embeddings} of non-simple permutations $\gamma$ of the basis\footnote{Theorem 9 of~\cite[Theorem 9]{AA05} ensures that, as soon as $\C$ contains finitely many simple permutations,
then the basis of $\C$ is finite.}
$B$ of $\C$ into simple permutations $\pi$ belonging to the class $\C$ (and into $\oplus$ and $\ominus$). 
These embeddings are block decompositions of the permutations $\gamma$, 
each (normalized) block being translated into a new avoidance constraint pushed downwards in the decomposition tree.
We then need to add new equations in the specification for $\C$ to take into account these new constraints.

The main algorithm of this section is \textsc{AmbiguousSystem} (Algo.~\ref{alg:sys-ambigu} below),
which uses auxiliary procedures described later on.
We prove in Section~\ref{ssec:proofThAmbiguous} that the result it produces has the following properties:

\begin{theo}\label{thm:systemeAmbigu}
Let $\C$ be a permutation class with a finite number of simple permutations. 
Denote by $B$ the (finite) basis of $\C$, by $\Bstar$ the subset of non-simple permutations in $B$, by $\s_\C$ the (finite) set of simple permutations in $\C$, and  by $\cA \langle E \rangle$ the set of permutations
of a set $\cA$ that avoid every pattern in a set $E$. 
The result of \textsc{AmbiguousSystem}$(B,\s_\C)$ 
is a finite system of combinatorial equations describing $\C$. 
The equations of this system are all of the form $\D_0 = 1\ \cup\ \bigcup \pi[\D_1, \dots, \D_n]$, where 
$\D_i = \wc^{\delta} \langle \Bstar \cup B_i \rangle$ with $\delta \in \{~, +, -\}$\footnote{For any set $\cA$ of permutations, 
when writing $\cA^{\delta}$ for $\delta \in \{~, +, -\}$, we mean that $\cA^{\delta}$ is either $\cA$ or $\cA^+$ or $\cA^-$.} and
$B_i$ contains only permutations corresponding to normalized blocks of elements of $\Bstar$. 
This system contains an equation whose left part is $\C$ and is complete, that is: every $\D_i$ that appears in the system is the left part of one equation of the system.
\end{theo}

\begin{algorithm}[htbp]
\DontPrintSemicolon 
\KwData{A finite basis of forbidden patterns defining ${\mathcal C}=Av(B)$ and the finite set $\mathcal{S}_\C$ of simple permutations in $\C$.}  
\KwResult{A system of equations of the form $\D_0 = 1 \cup \bigcup \pi[\D_1, \dots, \D_n]$ defining $\C$.}
\Begin{
	$\mathcal{E}\leftarrow$ \textsc{EqnForClass}$(\wc,\Bstar)$ \tcc*{See Algo.\,\ref{alg:comp-eqn} (p.\pageref{alg:comp-eqn})}
    \While{there is a right-only $\wc^{\delta}\langle E \rangle$ in some equation of $\Sys$} { $\mathcal{E}
      \leftarrow \mathcal{E}~\cup$ \textsc{EqnForClass}($\wc^{\delta}$, $E$)
    } }
\caption{\textsc{AmbiguousSystem}($B,\s_\C$)}\label{alg:sys-ambigu}
\end{algorithm}

An essential remark is that the obtained combinatorial system may be ambiguous, since it may involve unions of sets that are not always disjoint. 
We will tackle the problem of computing a non-ambiguous system in Section~\ref{sec:disambiguation}. 

It is also to note that the result of \textsc{AmbiguousSystem} provides a finite query-complete set containing $\C$:

\begin{cor}
\label{cor:QCS}
Let $\E$ be the system of equations output by \textsc{AmbiguousSystem}$(B,\s_\C)$.
Then the set $\{\D : \D \text{ is the left part of some equation of } \E \}$ is a query-complete set containing \C.
\end{cor}

\begin{proof}
This a direct consequence of the form of \E described in Theorem~\ref{thm:systemeAmbigu}.
\end{proof}

\subsection{A first system of equations}
Consider a permutation class ${\mathcal C}$, whose basis is
$B$ and which is not substitution-closed.  We compute a system describing
\C by adding constraints to the system obtained for \wc, as
in~\cite{AA05}.  
We denote by $\Bstar$ the subset of non-simple permutations of~$B$ \label{page:Bstar}
and  by $\cA \langle E \rangle$ the set of permutations
of $\cA$ that avoid every pattern in $E$, for any set $\cA$ of permutations and any set $E$ of
patterns.  Note that we have $(\cA \langle E
\rangle)^+ = \cA^+ \langle E \rangle$: the corresponding set is the one of 
permutations of $\cA$ that avoid $E$ and that are
$\oplus$-indecomposable. The same goes for $\cA^-$.

\begin{prop}\label{prop:system}
Let $\C$ be a permutation class, that contains $12$ and $21$. We have that
$\C^{\delta} = \wc^{\delta} \langle \Bstar \rangle$ for $\delta \in \{~, +, -\}$.
Moreover, 
        \begin{eqnarray}
          \wc \langle \Bstar \rangle &=& 1\ \uplus \ \oplus[\wc^+, \wc]\langle \Bstar \rangle\  \uplus \ \ominus[\wc^-, \wc]\langle \Bstar \rangle \ \uplus \textstyle\biguplus_{\pi \in \s_\C} \pi[\wc, \dots, \wc]\langle \Bstar \rangle \label{eqn:const1}\\
          \wc^+ \langle \Bstar \rangle &=& 1\  \uplus \ \ominus[\wc^-, \wc]\langle \Bstar \rangle\  \uplus\  \textstyle\biguplus_{\pi \in \s_\C} \pi[\wc, \dots, \wc]\langle \Bstar \rangle \label{eqn:const2}\\
          \wc^- \langle \Bstar \rangle &=& 1\  \uplus\  \oplus[\wc^+, \wc]\langle \Bstar \rangle\  \uplus\  \textstyle\biguplus_{\pi \in \s_\C} \pi[\wc, \dots, \wc]\langle \Bstar \rangle, \label{eqn:const3}
        \end{eqnarray}
all these unions being disjoint.
\end{prop}

\begin{proof}
Let $\sigma \in \C$, then $\sigma \in \wc$ and $\sigma$ avoids $\Bstar$, 
thus $\sigma \in \wc \langle \Bstar \rangle$.
Conversely, let  $\sigma \in \wc \langle \Bstar \rangle$ and let $\pi \in B$. 
If  $\pi \in \Bstar$ then $\sigma$ avoids $\pi$. 
Otherwise, $\pi$ is simple and $\pi \notin \C$. 
Because $\s_\wc = \s_\C$ by Remark~\ref{rem:simple_in_WC}, this implies that $\pi \notin \wc$.
Since $\sigma \in \wc$, $\sigma$ avoids $\pi$. Hence $\sigma \in \C$. 
Finally $\wc \langle \Bstar \rangle = \C$, thus $\wc^{\delta} \langle \Bstar \rangle = \C^{\delta}$.
Then the result follows from Proposition~\ref{prop:sys_wc}. 
\end{proof}

Equations~\eqref{eqn:const1} to~\eqref{eqn:const3} do not provide a combinatorial specification 
because the terms $\pi[\wc, \dots, \wc]\langle \Bstar \rangle$ are not simply expressed from the terms appearing 
on the left-hand side of these equations. 
To solve this problem, 
instead of decorating all terms on the right-hand side of Equations~\eqref{eqn:Wc1}--\eqref{eqn:Wc3} with constraint $\langle \Bstar \rangle$ 
(like in Equations~\eqref{eqn:const1}--\eqref{eqn:const3}), 
we propagate the constraint $\langle \Bstar \rangle$ into the subtrees. 
More precisely, by Lemma~18 of~\cite{AA05}, sets $ \pi[\wc, \dots, \wc]\langle \Bstar \rangle$
can be expressed as union of smaller sets: 
\begin{equation}
\pi[\wc, \dots, \wc]\langle \Bstar \rangle = \textstyle\bigcup_{i=1}^{\ell} \pi[\wc\langle E_{i,1} \rangle,\wc\langle E_{i,2} \rangle,\ldots,\wc\langle E_{i,k} \rangle] 
\label{eq:propagation}
\end{equation}
where the $E_{i,j}$ are sets of permutations which are patterns of some permutations of $\Bstar$. 
For instance, with $\C = Av(231)$, we have $\Bstar = \{231\}$, and 
$\ominus[\wc^-, \wc]\langle 231 \rangle = \ominus[\wc^-\langle 12 \rangle, \wc\langle 231 \rangle]$. 
Note that the set $\wc^-\langle 12 \rangle$, which is not a part of the initial system,
has appeared on the right-hand side of this equation, 
and we now need a new equation to describe it. 
In the general case, applying Equation~\eqref{eq:propagation} in the system of Proposition~\ref{prop:system}
introduces sets of the form $\wc^{\delta}\langle E_{i,j} \rangle$ on the right-hand side of an equation of the system that do not appear on the left-hand side of any equation. 
The reason is that $\{\C, \C^+,\C^-\}$ is not query-complete in general. 
Propagating the constraints in the subtrees as described below allows to determine a query-complete set which contains $\C$ 
(see Corollary~\ref{cor:QCS}).

We call \emph{right-only sets} the sets of the form $\wc^{\delta}\langle E_{i,j} \rangle$ which appear only on the right-hand side of an equation of the system. 
For each such set, we need to add a new equation to the system, 
starting from Equation~\eqref{eqn:Wc1},~\eqref{eqn:Wc2} or~\eqref{eqn:Wc3} depending on $\delta$, 
and propagating constraint $E_{i,j}$ instead of $\Bstar$. 
This may create new right-only terms in these new equations, and they are treated recursively in the same way. 
This process terminates, since the $E_{i,j}$ are sets of patterns of elements of $\Bstar$, 
and there is only a finite number of such sets (as $B$ is finite). 

The key to the precise description of the sets $E_{i,j}$, and to their effective computation, 
is given by the \emph{embeddings} of permutations $\gamma$ that are patterns of some $\beta \in \Bstar$ 
into the simple permutations $\pi$ of \C. 

\subsection{Embeddings: definition and computation}
\label{ssec:embeddings}

Recall that $\varepsilon$ denotes the empty permutation, \emph{i.e.} the permutation of size $0$. 
We take the convention $\cA \langle \varepsilon \rangle = \emptyset$.

\medskip

A \emph{generalized substitution} (also called \emph{lenient inflation} in~\cite{BHV08a}) $\sigma\{\pi^{1}, \pi^{2} ,\ldots, \pi^{n}\}$ is defined like a substitution 
with the particularity that any $\pi^i$ may be the empty permutation. 
Note that $\sigma[\pi^{1}, \pi^{2} ,\ldots, \pi^{n}]$ necessarily contains $\sigma$
whereas $\sigma\{\pi^{1}, \pi^{2} ,\ldots, \pi^{n}\}$ may avoid $\sigma$.
For instance, the generalized substitution $ 1\, 3\, 2 \{2\, 1, \varepsilon, 1\}$ gives the permutation $ 2\, 1\, 3$ which avoids $ 1\, 3\, 2$.

Thanks to generalized substitutions, we define the notion of embedding, 
which expresses how a pattern $\gamma$ can be involved in a permutation whose decomposition tree has a root $\pi$:
\begin{defi}\label{defi:assignation}
Let $\pi=\pi_1\dots\pi_n$ and $\gamma$ be two permutations of size $n$ and $p$ respectively
and $P_\gamma$ the set of intervals\footnote{Recall that in this article,
an interval of a permutation is a set of \emph{indices} corresponding to a block of the permutation
(see Definition \ref{defn:interval} p.\pageref{defn:interval}).}
of $\gamma$, including the trivial ones. 
An {\em embedding of $\gamma$ in~${\pi=\pi_1\dots\pi_n}$} is a map $\alpha$ from $\{1, \ldots, n\}$ to $P_\gamma$ such that:
\begin{itemize}\setlength{\itemsep}{-3pt}
\item if the intervals $\alpha(i)$ and $\alpha(j)$ are not empty, and $i<j$, then $\alpha(i)$ consists of smaller indices than $\alpha(j)$;
\item as a word, $\alpha(1) \ldots \alpha(n)$ is a factorization of the word $1 \ldots |\gamma|$ (which may include empty factors).
\item it holds that $\pi\{\gamma_{\alpha(1)},\dots,\gamma_{\alpha(n)}\}=\gamma$ 
(see Definition~\ref{def:subsequence} (p.\pageref{def:subsequence}) for the definition of $\gamma_{\alpha(i)}$).
\end{itemize}
\end{defi}

\begin{example}
For any permutations $\gamma$ and $\pi$, $\alpha : 
\begin{cases}
	  1   & \mapsto [1..|\gamma|]\\
   	  k>1 & \mapsto \emptyset
\end{cases}
$ is an embedding of $\gamma$ in~$\pi$. Indeed $\gamma_{[1..|\gamma|]} = \gamma$ and $\pi\{\gamma,\varepsilon,\dots,\varepsilon\}=\gamma$.
\end{example}

Note that if we denote the non-empty images of $\alpha$ by $\alpha^1,\dots,\alpha^N$ and if we remove from $\pi$ the $\pi_i$ such that $\alpha(i)=\varepsilon$,
we obtain a pattern $\sigma$ of $\pi$ such that $\gamma = \sigma [\gamma_{\alpha^1},\dots,\gamma_{\alpha^N} ]$. 
But this pattern $\sigma$ may occur at several places in $\pi$ so a block decomposition
$\gamma = \sigma [\gamma_{\alpha^1},\dots,\gamma_{\alpha^N} ]$ may correspond to several embeddings of $\gamma$ in~$\pi$.

\begin{example}\label{ex:block-e}
There are $12$ embeddings of $\gamma = 5\,4\,6\,3\,1\,2$ into $\pi = 3\,1\,4\,2$, shown in Table~\ref{table:embeddings}. 
For instance, when writing $\gamma$ as the substitution ${\ominus[3241,12]}$, 
they are derived from the generalized substitutions $\pi\{3241,12,\varepsilon,\varepsilon\}$, $\pi\{3241,\varepsilon,\varepsilon,12\}$ 
and $\pi\{\varepsilon,\varepsilon,3241,12\}$, corresponding to the three occurrences of $21$ in $\pi$. 
But when writing $\gamma$ as $312[3241,1,1]$, 
since $312$ has only one occurrence in $\pi$, 
only one embedding is derived, which comes from $\pi\{3241,1,\varepsilon,1\}$.
\end{example}

\begin{table}[htbp]
\begin{center}
\begin{tabular}{rl||c|c|c|c|}
\cline{3-6}
&& $\gamma_{\alpha(1)}$ & $\gamma_{\alpha(2)}$ & $\gamma_{\alpha(3)}$ & $\gamma_{\alpha(4)}$ \\
\hline
\multicolumn{1}{ |c } {$\gamma = 1[546312]$} & $ = \pi\{546312,\varepsilon,\varepsilon,\varepsilon\}$ & $546312$ & $\varepsilon$ & $\varepsilon$ & $\varepsilon$\\
	\cline{3-6}
\multicolumn{1}{ |c }{} & $ = \pi\{\varepsilon,546312,\varepsilon,\varepsilon\}$ & $\varepsilon$ & $546312$ & $\varepsilon$ & $\varepsilon$\\
	\cline{3-6}
\multicolumn{1}{ |c }{} & $ = \pi\{\varepsilon,\varepsilon,546312,\varepsilon\}$ & $\varepsilon$ & $\varepsilon$ & $546312$ & $\varepsilon$\\
	\cline{3-6}
\multicolumn{1}{ |c }{} & $ = \pi\{\varepsilon,\varepsilon,\varepsilon,546312\}$ & $\varepsilon$ & $\varepsilon$ & $\varepsilon$ & $546312$\\
\hline
\multicolumn{1}{ |c }{$\gamma = \ominus[213,312]$} & $ = \pi\{213,312,\varepsilon,\varepsilon\}$ & $213$ & $312$ & $\varepsilon$ & $\varepsilon$ \\
	\cline{3-6}
\multicolumn{1}{ |c }{} & $ = \pi\{213,\varepsilon,\varepsilon,312\}$ & $213$ & $\varepsilon$ & $\varepsilon$ & $312$ \\
	\cline{3-6}
\multicolumn{1}{ |c }{} & $ = \pi\{\varepsilon,\varepsilon,213,312\}$ & $\varepsilon$ & $\varepsilon$ & $213$ & $312$ \\
\hline
\multicolumn{1}{ |c }{$\gamma = \ominus[3241,12]$} & $ = \pi\{3241,12,\varepsilon,\varepsilon\}$ & $3241$ & $12$ & $\varepsilon$ & $\varepsilon$ \\
	\cline{3-6}
\multicolumn{1}{ |c }{} & $ = \pi\{3241,\varepsilon,\varepsilon,12\}$ & $3241$ & $\varepsilon$ & $\varepsilon$ & $12$ \\
	\cline{3-6}
\multicolumn{1}{ |c }{} & $ = \pi\{\varepsilon,\varepsilon,3241,12\}$ & $\varepsilon$ & $\varepsilon$ & $3241$ & $12$ \\
\hline
\multicolumn{1}{ |c }{$\gamma = 231[21,1,312]$} & $ = \pi\{21,\varepsilon,1,312\}$ & $21$ & $\varepsilon$ & $1$ & $312$ \\
\hline
\multicolumn{1}{ |c }{$\gamma = 312[3241,1,1]$} & $ = \pi\{3241,1,\varepsilon,1\}$ & $3241$ & $1$ & $\varepsilon$ & $1$ \\
\hline
\end{tabular}
\end{center}
\caption{The embeddings of $\gamma = 5\,4\,6\,3\,1\,2$ into $\pi = 3\,1\,4\,2$.}
\label{table:embeddings}
\end{table}

Note that this definition of embeddings conveys the
same notion as in~\cite{AA05}, but it is formally different and it will turn out to be more
adapted to the definition of the sets $E_{i,j}$ in Section~\ref{ssec:propagation}.

For now, we present how to compute the embeddings of some permutation $\gamma$ into a permutation $\pi$. 
This is done with \textsc{AllEmbeddings} (Algo.~\ref{alg:all_embeddings}) in two main steps, 
as suggested by the two parts of Table~\ref{table:embeddings}.
First we compute all block decompositions of $\gamma$ 
(the left part of the table shows only some block decompositions of $\gamma=5\,4\,6\,3\,1\,2$, 
namely those that can be expressed as a generalized substitution in $\pi = 3\,1\,4\,2$).
Then for each block decomposition of $\gamma$, we compute all the embeddings of $\gamma$ into $\pi$ which correspond to this block decomposition 
(this is the right part of the table). 

\SetKwBlock{CBDfunc}{\textsc{BlockDecompositions}}{end}
\SetKwBlock{CEfunc}{\textsc{Embeddings}}{end}
\SetKwBlock{Intfunc}{\textsc{Intervals}}{end}
\begin{algorithm}[htbp]
\DontPrintSemicolon 
\KwData{Two permutations $\gamma$ and $\pi$}
\KwResult{The set of embeddings of $\gamma$ into $\pi$}
\Begin{
$\mathcal{E} \leftarrow \emptyset$ \;
$\mathcal{D} \leftarrow$ \textsc{BlockDecompositions}$(\gamma)$ \tcc*{See Algo.~\ref{alg:blockdecompositions}}
\ForEach{$d \in \mathcal{D}$}
{$\mathcal{E} \leftarrow \mathcal{E} \, \cup$ \textsc{Embeddings}$(d,\pi)$ \tcc*{See Algo.~\ref{alg:embeddings}}
}
\Return $\mathcal{E}$\;
}
\caption{\textsc{AllEmbeddings}$(\gamma,\pi)$}\label{alg:all_embeddings}
\end{algorithm}

\medskip

\begin{algorithm}[htbp]
\DontPrintSemicolon 
\KwData{A permutation $\gamma$}  
\KwResult{The set $\mathcal{D}$ of block decompositions of $\gamma$}
\Begin{ 
$\mathcal{D} \leftarrow \emptyset$; \quad 
$\mathcal{P} \leftarrow$ \textsc{Intervals}$(\gamma,1)$\;
\ForEach{$u=\Big( (i_1,j_1),\dots,(i_m,j_m) \Big) \in \mathcal{P}$}
{$\mathcal{P} \leftarrow \mathcal{P} \setminus \{u\}$ \;
\lIf{$j_m = |\gamma|$}{$\mathcal{D} \leftarrow \mathcal{D} \cup \{u\}$}
\Else{$\mathcal{S} \leftarrow$ \textsc{Intervals}$(\gamma,j_m+1)$\;
$\mathcal{P} \leftarrow \mathcal{P} \cup \{u \cdot s \mid s \in \mathcal{S} \}$
}
}
\Return  $\mathcal{D}$\;
}

\bigskip

\tcc{Returns the set of intervals of $\gamma$ of the form $(i,j)$}
\SetKwSty{textsc}
\Intfunc(\params{$\gamma,i$}){\SetKwSty{textbf}
$\mathcal{I} \leftarrow \emptyset$\;
$max \leftarrow \gamma_i$;\quad $min \leftarrow \gamma_i$\;
\For{$j$ from $i$ to $|\gamma|$}
{
 $max\leftarrow \max(max,\gamma_j)$;\quad $min\leftarrow \min(min,\gamma_j) $\;
\lIf{$max-min = j-i$}{$\mathcal{I} \leftarrow \mathcal{I} \cup \{(i,j)\}$}
}
\Return $\mathcal{I}$
}

\caption{\textsc{BlockDecompositions}$(\gamma)$}\label{alg:blockdecompositions}
\end{algorithm}

The procedure \textsc{BlockDecompositions} (Algo.~\ref{alg:blockdecompositions}) finds all the
block decompositions of $\gamma$.  It needs to compute all the
intervals of $\gamma$. Such intervals can be described as pairs of
indices $(i,j)$ with $i \leq j$ such that $\max_{i \leq k \leq
  j}(\gamma_k) - \min_{i \leq k \leq j}(\gamma_k)= j-i$. Denoting by
$p$ the size of $\gamma$, this remark allows to compute easily all the
intervals starting at $i$, for each $i \in [1..p]$ -- see \textsc{Intervals}$(\gamma,i)$.  
Then, \textsc{BlockDecompositions} builds the set $\mathcal{D}$
of all sequences of intervals of the form
$\big((i_1,j_1),\dots,(i_m,j_m) \big)$ with $i_1=1$, $j_m=p$, and
$i_{k+1}=j_k + 1$ for all $k<m$. These sequences correspond exactly to
the block decompositions of $\gamma$. They are computed iteratively, starting from all
the intervals of the form $(1,i)$ and examining how they can be
extended to a sequence of intervals in $\mathcal{D}$.
To this end, note that any sequence
$\big((i_1,j_1),\dots,(i_m,j_m) \big)$ as above but such that $j_m \neq p$
is the prefix of at least one sequence of $\mathcal{D}$, since for
every $i \in [1..p]$, at least $(i,i)$ is an interval of $\gamma$.

For each $i \in [1..p]$, computing \textsc{Intervals}$(\gamma,i)$ is
done in $\mathcal{O}(p)$, so $\mathcal{O}(p^2)$ is enough to compute
once and for all the results of \textsc{Intervals}$(\gamma,i)$ for
all $i \in [1..p]$ (provided we store them).  The computation of all block decompositions of
$\gamma$ with \textsc{BlockDecompositions}$(\gamma)$ then
costs $\mathcal{O}(p\, 2^p)$. 
Indeed there are at most $2^{p-1}$ such sequences of $\mathcal{D}$ and each sequence of $\mathcal{D}$ contains at most $p$ intervals.
At each step, the algorithm extends a sequence already built.
Thus the (amortized) overall complexity is $\mathcal{O}(p\, 2^p)$ and this bound is tight.

\begin{algorithm}[htbp]
\DontPrintSemicolon 
\KwData{A permutation $\pi$ of size $n$ and a block decomposition $d \in \mathcal{D}$ of $\gamma$, $d = \Big( (i_1,j_1),\dots,(i_m,j_m) \Big)$} 
\KwResult{The set of embeddings of $\gamma$ into $\pi$ which correspond to $d$}

\Begin{
\If(\tcc*[f]{No embedding of $\gamma$ into $\pi$ corresponding to $d$}){$m > n$}
	{\Return $\emptyset$}
\Else{
\tcc{$\sigma$ is the skeleton such that $d=\sigma[\gamma^{(1)}, \ldots , \gamma^{(m)}]$}
$\sigma \leftarrow$ $\gamma_{\{i_1,i_2,  \ldots, i_m\}}$ \;
$\mathcal{E} \leftarrow \emptyset$\;
\ForEach{subset $S = \{s_1, \ldots, s_m\}$ of $\{1,2,\ldots, n\}$ with $s_1 < \ldots < s_m$}
{\If{$\pi_{\{s_1,s_2,  \ldots, s_m\}} = \sigma$}{
$\alpha \leftarrow $ embedding of $\gamma$ in $\pi$ such that
$\begin{cases}
\alpha(s_k) = [i_k..j_k] \text{ for } 1\leq k \leq m \\
\alpha(i) = \varepsilon \text{ for } i \notin S
\end{cases}$\;
$\mathcal{E} \leftarrow \mathcal{E} \cup \{\alpha\}$ \;
}
}
\Return $\mathcal{E}$}
}

\caption{\textsc{Embeddings}$(d,\pi)$}\label{alg:embeddings}
\end{algorithm}

The procedure \textsc{Embeddings} (Algo.~\ref{alg:embeddings}) finds the embeddings of $\gamma$ in $\pi$ which correspond to a given block-de\-com\-po\-si\-tion of $\gamma$. 
The output $\mathcal{D}$ of \textsc{BlockDecompositions} corresponds to all the substitutions $\sigma[\gamma^{(1)}, \ldots , \gamma^{(m)}]$ which are equal to $\gamma$. 
For each $d \in \mathcal{D}$ we first determine the corresponding {\em skeleton} $\sigma$ defined as the normalization of a sequence $s_1 \ldots s_m$ of $m$ integers, 
where each $s_i$ is an element of $\gamma$ falling into the $i$-th block of this decomposition. 
Then, we compute the set of occurrences of $\sigma$ in $\pi$, 
since they are in one-to-one correspondence with embeddings of $\gamma$ into $\pi$ which corresponds to the given substitution $\gamma = \sigma[\gamma^{(1)}, \ldots , \gamma^{(m)}]$. 
These occurrences are naively computed by testing, for every subsequence of $m$ elements of~$\pi$, whether it is an occurrence of $\sigma$.

The preprocessing part of \textsc{Embeddings} (Algo.~\ref{alg:embeddings}), which
consists in the computation of the skeleton~$\sigma$, costs $\mathcal{O}(m)$.  
The examination of all possible occurrences of $\sigma$ in $\pi$ is then
performed in $\mathcal{O}(m \, {\binom{n}{m}})$, where $n$ is the size of $\pi$.

In total, computing the set $\mathcal{D}$ of block decompositions of
$\gamma$ is performed in $\mathcal{O}(p\, 2^p)$. The set $\mathcal{D}$
contains at most ${\binom{p-1}{m-1}}$ block decompositions in $m$
blocks, so that the complexity of the main loop of 
\textsc{AllEmbeddings} 
is at most $\sum_{m=1}^{p} m
{\binom{n}{m}}{\binom{p-1}{m-1}} \leq 
p \, 2^{n+p-1}$.
Note that it can be reduced to $\min(p,n) \, 2^{n + \min(p,n)-1}$, by discarding all block decompositions of $\gamma$ in $m>n$ blocks, 
since these will never be used in an embedding of $\gamma$ into $\pi$. 
The upper bound $\sum_{m=1}^{p} m
{\binom{n}{m}}{\binom{p-1}{m-1}}$ is tight, as can be seen with $p \leq n$, $\gamma = 12 \ldots p$ and $\pi = 12\ldots n$. 

\medskip

In the next section, we explain how to use embeddings to propagate the pattern avoidance constraints in the subtrees.

\subsection{Propagating constraints}
\label{ssec:propagation}
 
To compute our combinatorial system,
we compute equations for sets $\wc^{\delta}\langle E\rangle$ (initially for $\wc\langle \Bstar \rangle$, that is \C)
starting from Equation~\eqref{eqn:Wc1},~\eqref{eqn:Wc2} or~\eqref{eqn:Wc3} (p.\pageref{eqn:Wc1}). 
For every set $\pi[\wc_1, \dots, \wc_n]$ that appears on the right-hand side of the equation, 
we push the pattern avoidance constraints of $E$ in the subtrees. 
This is achieved using embeddings of excluded patterns in the root $\pi$.  For
instance, assume that $\gamma= 5\,4\,6\,3\,1\,2 \in B^\star$ and $\cal
S_{\cal C}=\{3142\}$, and consider $3142[\wc, \wc,\wc, \wc]\langle \gamma \rangle$.
The embeddings of~$\gamma$ in~$3142$ indicates how the pattern $\gamma$ can be
found in the subtrees in $3142[\wc, \wc,\wc,\wc]$. 
The first embedding of Example~\ref{ex:block-e} indicates that the full pattern $\gamma$ can appear 
all included in the first subtree.
On the other hand, the last embedding of the same example
tells us that $\gamma$ can spread over all the subtrees of~$3142$ except
the third one. In order to avoid this particular embedding of~$\gamma$, it
is enough to avoid one of the induced pattern $\gamma_I$ in one of the
subtrees. However, in order to ensure that~$\gamma$ is avoided, the
constraints resulting from all the embeddings must be considered and
merged. This is formalized in Proposition~\ref{prop:restrictions=blocs}.

\begin{prop}\label{prop:restrictions=blocs}
Let $\pi$ be a simple permutation of size $n$ and $\C_1, \dots, \C_n$ be sets of permutations. 
For any permutation $\gamma$, the set
$\pi[\C_1, \dots, \C_n] \langle \gamma \rangle$ rewrites as a union of sets of the form $\pi[\D_1, \dots, \D_n]$ where, for all $i$, 
$\D_i = \C_i \langle \gamma, \dots \rangle$ and the restrictions appearing after $\gamma$ (if there are any) are patterns of $\gamma$ 
corresponding to normalized blocks of $\gamma$. 

More precisely, we have
\begin{equation} \label{eq:propagate}
        \pi[\C_1, \dots ,\C_n] \langle \gamma \rangle =
        \bigcup_{(k_1, \dots, k_\ell) \in  K^\pi_\gamma} \pi[\C_1 \langle E_{1,k_1 \dots k_\ell} \rangle, \dots ,\C_n \langle E_{n,k_1 \dots k_\ell} \rangle]
\end{equation}
where $K^\pi_\gamma =\{(k_1, \dots, k_\ell) \in [1..n]^\ell\ |\ \forall i,\ \gamma_{\alpha_i(k_i)} \neq \varepsilon \text{ and } \gamma_{\alpha_i(k_i)} \neq 1\}$
and $E_{m,k_1 \dots k_\ell }= \{ \gamma_{\alpha_i(k_i)}\ |\ i \in [1..\ell] \text{ and } k_i=m \}$,
$\{\alpha_1, \dots, \alpha_\ell\}$ being the set of embeddings of $\gamma$ in $\pi$.

Similarly, $\oplus[\C^+_1, \C_2] \langle \gamma \rangle$ (resp.\ $\ominus[\C^-_1, \C_2] \langle \gamma \rangle$) rewrites as
the union over $K^{12}_\gamma$ (resp. $K^{21}_\gamma$) of sets $\oplus[\D_1^+, \D_2]$ (resp.~$\ominus[\D_1^-, \D_2]$)
with $\D_i = \C_i \langle E_{i,k_1 \dots k_\ell} \rangle$.
\end{prop}

Proposition~\ref{prop:restrictions=blocs} and its proof heavily borrow from the proof of Lemma 18 of~\cite{AA05}. 
There are however several differences.  

In the first part of our statement, while
we stop at the condition $\D_i = \C_i \langle \gamma, \dots \rangle$ (which appears in the proof of Lemma 18 in~\cite{AA05}), 
\cite{AA05} rephrases it as: $\D_i$ is $\C_i \langle \gamma \rangle$ or a strong subclass of this class 
(that is, a proper subclass of $\C_i$ which has the property that every basis element of $\D_i$ is involved in some basis element of $\C_i$). 
Note that this rephrasing is not always correct (but this does not affect the correctness of the result of~\cite{AA05} which uses their Lemma 18).
A counter example is obtained taking  $\pi =2413$, $\C_i = Av(12)$ for all $i$ and $\gamma=2143$. 
Then, because of the embedding of $\gamma$ in $\pi$ that maps $21$ in $\pi_1$ and $43$ in $\pi_2$,  
we get a term with $\D_1=\C_1 \langle 2143, 21 \rangle = Av(12, 21)$, which is not a strong subclass of~$\C_1 \langle \gamma \rangle = Av(12)$.

The second part of Proposition~\ref{prop:restrictions=blocs} is not present in~\cite{AA05}. 
That is, Proposition~\ref{prop:restrictions=blocs} goes further than the proof of Lemma 18 of~\cite{AA05}:
we provide a statement which is nonetheless constructive as the proof of~\cite{AA05}, 
but also explicit, so that it can be directly used for algorithmic purpose.

\begin{proof}
We first consider the case of a simple root $\pi$. 
Let $\gamma$ be  a permutation  and  $\{\alpha_1, \dots, \alpha_\ell\}$ be the set of embeddings of $\gamma$ in $\pi$, 
each $\alpha_i$ being associated to the generalized substitution $\gamma = \pi\{\gamma_{\alpha_i(1)},\dots, \gamma_{\alpha_i(n)}\}$. 

Let $\sigma = \pi[\sigma^{(1)}, \dots, \sigma^{(n)}]$ where each $\sigma^{(k)}\in\C_k$.
Then 
$\sigma$ avoids $\gamma$ \ssi for every embedding~$\alpha_i$ (with $1\leq i \leq \ell$), 
there exists $k_i \in [1..n]$ such that $\gamma_{\alpha_i(k_i)}$ is not a pattern of $\sigma^{(k_i)}$, 
i.e. such that $\sigma^{(k_i)}$ avoids $\gamma_{\alpha_i(k_i)}$. 
Equivalently, $\sigma$ avoids $\gamma$ \ssi there exists a tuple $(k_1, \dots, k_\ell) \in [1..n]^\ell$ such that 
for every embedding~$\alpha_i$ (with $1\leq i \leq \ell$), $\sigma^{(k_i)}$ avoids $\gamma_{\alpha_i(k_i)}$. 
Thus,
\[ \pi[\C_1, \dots ,\C_n] \langle \gamma \rangle =
\bigcup_{(k_1, \dots, k_\ell) \in [1..n]^\ell}\ \bigcap_{i=1}^\ell \pi[\C_1, \dots ,\C_{k_i} \langle \gamma_{\alpha_i(k_i)} \rangle, \dots ,\C_n].
\]       
But, for any set $\calA$ of permutations we have $\calA\langle \varepsilon \rangle = \emptyset$.
So if for some $i \in [1..\ell]$, $\gamma_{\alpha_i(k_i)}=\varepsilon$, 
then $\bigcap_{i=1}^\ell \pi[\C_1, \dots ,\C_{k_i} \langle \gamma_{\alpha_i(k_i)} \rangle, \dots ,\C_n]=\emptyset$.
The same goes for the trivial permutation $1$ since every permutation contains $1$.\\
Therefore, we have:
\[ \pi[\C_1, \dots ,\C_n] \langle \gamma \rangle =
\bigcup_{(k_1, \dots, k_\ell) \in K^\pi_\gamma}\ \bigcap_{i=1}^\ell \pi[\C_1, \dots ,\C_{k_i} \langle \gamma_{\alpha_i(k_i)} \rangle, \dots ,\C_n],
\]
with $K^\pi_\gamma =\{(k_1, \dots, k_\ell) \in [1..n]^\ell\ |\ \forall i,\ \gamma_{\alpha_i(k_i)} \neq \varepsilon \text{ and } \gamma_{\alpha_i(k_i)} \neq 1\}$.
Following Example~\ref{ex:block-e}, and denoting $\alpha_1$ to $\alpha_{12}$ the embeddings of Table~\ref{table:embeddings} from top to bottom, 
we have for example $\mathbf{k} = (1,2,3,4,1,4,4,2,1,4,4,1) \in K^\pi_\gamma$.

Moreover, as $\pi$ is simple, by uniqueness of the substitution decomposition we have for any sets of permutations $E_1, \dots, E_n,$ $ F_1, \dots, F_n$:
\[
\pi[\C_1 \langle E_1 \rangle, \dots ,\C_n \langle E_n \rangle] \cap
\pi[\C_1 \langle F_1 \rangle, \dots ,\C_n \langle F_n \rangle] =
\pi[\C_1 \langle E_1 \cup F_1 \rangle, \dots ,\C_n \langle E_n \cup F_n \rangle].
\]
Thus,
 \[       \pi[\C_1, \dots ,\C_n] \langle \gamma \rangle =
        \bigcup_{(k_1, \dots, k_\ell) \in  K^\pi_\gamma} \pi[\C_1 \langle E_{1,k_1 \dots k_\ell} \rangle, \dots ,\C_n \langle E_{n,k_1 \dots k_\ell} \rangle].
\]
where $E_{m,k_1 \dots k_\ell }= \{ \gamma_{\alpha_i(k_i)}\ |\ i \in [1..\ell] \text{ and } k_i=m \}$.
Following the same example again, $E_{1,\mathbf{k}} = \{\gamma_{\alpha_1(1)},\gamma_{\alpha_5(1)}, \gamma_{\alpha_9(1)}, \gamma_{\alpha_{12}(1)}\} = \{546312,213,3241\}$, 
$E_{2,\mathbf{k}} = \{546312,12\}$, 
$E_{3,\mathbf{k}} = \{546312\}$ and
$E_{4,\mathbf{k}} = \{546312,312,12\}$.

Proposition~\ref{prop:restrictions=blocs} follows as soon as we ensure that each set $E_{m,k_1 \dots k_\ell}$
contains $\gamma$ and possibly other restrictions that are patterns of $\gamma$ corresponding to normalized blocks of $\gamma$. 
For a given $m \in [1..n]$, there always exists an embedding of $\gamma$ in $\pi$ that maps the whole permutation $\gamma$ to $\pi_m$. 
Denoting this embedding $\alpha_{j_m}$, we have $\gamma_{\alpha_{j_m}(m)}=\gamma$ and $\gamma_{\alpha_{j_m}(q)}=\varepsilon$ for $q \neq m$. 
Consequently, if $k_{j_m} \neq m$ then $\gamma_{\alpha_{j_m}(k_{j_m})}=\varepsilon$ and $(k_1, \dots, k_\ell) \notin K^\pi_\gamma$. 
Therefore when $(k_1, \dots, k_\ell) \in K^\pi_\gamma$ then $k_{j_m} = m$ and the set $E_{m,k_1 \dots k_\ell }$ contains at least $\gamma_{\alpha_{j_m}(k_{j_m})} = \gamma$. 
Moreover, by definition its other elements are patterns of $\gamma$ corresponding to normalized blocks of $\gamma$.

Notice now that the proof immediately extends to the case of roots $\oplus$ and $\ominus$. 
Indeed, in the above proof, we need $\pi$ to be simple 
only because we use the uniqueness of the substitution decomposition.
In the case $\pi =\oplus$ (resp.~$\ominus$), the uniqueness is ensured by taking the set $\C^+_1$ of $\oplus$-indecomposable permutations of $\C_1$
(resp.\ the set $\C^-_1$ of $\ominus$-indecomposable permutations).
\end{proof}

\begin{example}
For $\pi = \ominus$ and $\gamma = 3412$, there are three embeddings of $\gamma$ in $\pi$: 
$\alpha_1$ which follows from the generalized substitution $3412 = \ominus\{3412,\varepsilon\}$, 
$\alpha_2$ which follows from $3412 = \ominus\{\varepsilon,3412\}$, and 
$\alpha_3$ which follows from $3412 = \ominus \{12,12\}$. 
So  $K^\pi_\gamma =\{(1,2,1),(1,2,2)\}$ and the application of Equation~\eqref{eq:propagate} gives
\[
\ominus[\C_1^-,\C_2]\langle 3412 \rangle = \ominus[\C_1^-\langle 3412,12 \rangle,\C_2 \langle 3412 \rangle] \cup 
\ominus [\C_1^-\langle 3412 \rangle,\C_2 \langle 3412,12 \rangle],
\]
which simplifies to 
$
\ominus[\C_1^-,\C_2]\langle 3412 \rangle = \ominus[\C_1^-\langle 12 \rangle,\C_2 \langle 3412 \rangle] \cup 
\ominus [\C_1^-\langle 3412 \rangle,\C_2 \langle 12 \rangle]
$.
\end{example}

By induction on the size of~$P$, 
Proposition~\ref{prop:restrictions=blocs} extends to the case of a set $P$ of excluded patterns, instead of a single permutation $\gamma$:

\begin{prop}\label{prop:strongsubclass}
For any simple permutation $\pi$ of size $n$ and for any set of permutations $P$, 
the set $\pi[\C_1, \dots, \C_n] \langle P \rangle$ rewrites as a union of sets~$\pi[\D_1, \dots, \D_n]$ where for all $i$,
$\D_i = \C_i \langle P \cup P_i \rangle$ with $P_i$ containing only permutations corresponding to normalized blocks of elements of~$P$. 

Similarly, $\oplus[\C^+_1, \C_2] \langle P \rangle$ (resp.\ $\ominus[\C^-_1, \C_2] \langle P \rangle$) rewrites as a union of sets 
$\oplus[\D_1^+, \D_2]$ (resp.\ $\ominus[\D_1^-, \D_2]$) 
where for $i=1$ or $2$, $\D_i = \C_i \langle P \cup P_i \rangle$ 
with $P_i$ containing only permutations corresponding to normalized blocks of elements of $P$.
\end{prop}

Now we have all the results we need to describe an algorithm computing a (possibly ambiguous) combinatorial system describing $\C$.

\subsection{An algorithm computing a combinatorial system describing $\C$}
\label{ssec:proofThAmbiguous}

We describe the algorithm \textsc{AmbiguousSystem} (Algo.~\ref{alg:sys-ambigu} p.\pageref{alg:sys-ambigu})
that takes as input the basis $B$ of a class $\C$ and the set $\SC$ of simple permutations in $\C$ (both finite),
and that produces in output a (possibly ambiguous) system of combinatorial equations describing the
permutations of $\C$ through their decomposition trees. 
Recall (from p.\pageref{page:Bstar}) that $\Bstar$ denotes the subset of non-simple permutations of $B$. 

The main step is performed by~\textsc{EqnForClass} (Algo.~\ref{alg:comp-eqn} below),
taking as input $\s_\C$, $\delta \in \{~ , +, -\}$ and a set $E$ of patterns (initially $\Bstar$),
and producing an equation describing $\wc^{\delta}\langle E\rangle$.
The algorithm takes an equation of the form~\eqref{eqn:Wc1},~\eqref{eqn:Wc2} or~\eqref{eqn:Wc3} (p.\pageref{eqn:Wc1}) 
describing $\wc^{\delta}$ and  adds one by one the constraints of $E$ using Equation~\eqref{eq:propagate}
as described in procedure \textsc{AddConstraints}.

\SetKwBlock{PGfunc}{\textsc{AddConstraints}}{end}
\begin{algorithm}[htbp]
\DontPrintSemicolon 
\KwData{$E$ is a set of permutations, $\wc^{\delta}$ is given by $\s_\C$ and $\delta \in \{~ , +, -\}$ }
\SetKw{In}{in} 
\KwResult{An equation defining $\wc^{\delta}\langle E\rangle$ as a union of $\pi[\C_1, \dots , \C_n]$ }
\Begin{
	$\mathcal{E} \leftarrow$ Equation~\eqref{eqn:Wc1} or \eqref{eqn:Wc2} or \eqref{eqn:Wc3} p.\pageref{eqn:Wc1} (depending on $\delta$), replacing the left part $\wc^{\delta}$ by $\wc^{\delta}\langle E\rangle$\;
	\ForEach{constraint $\gamma$ \In $E$}{
		\ForEach{$t=\pi[\C_1, \dots , \C_n]$ that appears \In $\mathcal{E}$}{
			$t\leftarrow$ \textsc{AddConstraints}$(\pi[\C_1, \dots , \C_n], \gamma )$ 
			\tcc*{this step modifies $\mathcal{E}$}
 		}
	}
	\Return $\mathcal{E}$\;
}
\bigskip
\tcc{Returns a rewriting of $\pi[\C_1 \dots \C_n] \langle \gamma \rangle$ as a union $\bigcup \pi[\D_1, \dots \D_n]$ }
\SetKwSty{textsc} 
\PGfunc(\params{$\pi[\C_1 \ldots \C_n] , \gamma$}){ \SetKwSty{textbf} 
	compute all the embeddings of $\gamma$ in $\pi$ with \textsc{AllEmbeddings}
        \tcc*{See Algo.\,\ref{alg:all_embeddings} (p.\pageref{alg:all_embeddings})}
	compute $K^\pi_\gamma$ and sets $E_{m,k_1 \dots k_\ell }$ defined in Equation~\eqref{eq:propagate} \; 
	\Return $\bigcup_{(k_1, \dots, k_\ell) \in K^\pi_\gamma} \pi[\C_1 \langle E_{1,k_1 \dots k_\ell} \rangle, \dots ,\C_n \langle E_{n,k_1 \dots k_\ell} \rangle]$\;
}
\caption{\textsc{EqnForClass}($\wc^\delta,E$)}\label{alg:comp-eqn}
\end{algorithm}

The procedure~\textsc{AmbiguousSystem}
keeps adding new equations to the system 
which consists originally of the equation describing $\wc\langle \Bstar \rangle$, that is \C.
This algorithm repeatedly calls  \textsc{EqnForClass}
until every $\wc^{\delta}\langle E_i \rangle$ appearing in the system is defined by an equation. 
All the sets $E$ are sets of normalized blocks (therefore of patterns) of permutations in $\Bstar$. 
Since $B$ is finite, there is only a finite number of patterns of elements of $\Bstar$, hence a finite number of possible $E$, and~\textsc{AmbiguousSystem} terminates. 
As for its complexity, it depends on the number of equations in the output system for which we give bounds in Section~\ref{ssec:bounds}.

The correctness of the algorithm is a consequence of
Propositions~\ref{prop:system}, \ref{prop:restrictions=blocs} and~\ref{prop:strongsubclass}, 
which therefore completes the proof of Theorem~\ref{thm:systemeAmbigu}.

\begin{example}
\label{ex:algo1}
Consider the class $\C =Av(B)$ for $B=\{1243,2413,531642,41352\}$: $\C$ contains only one simple permutation (namely $3142$), 
and $\Bstar = \{1243\}$. Applying the procedure~\textsc{AmbiguousSystem}  to this class $\C$ gives the following system of equations:
\begin{small}
\begin{eqnarray}
\wc\langle1 2 4 3 \rangle &=& 1 \ \cup \ \oplus[\wc^{+}\langle 1 2 \rangle , \wc\langle1 3 2 \rangle] \ \cup \ \oplus[\wc^{+}\langle 1 2 4 3 \rangle , \wc\langle2 1 \rangle] \ \cup \ \ominus[\wc^{-}\langle 1 2 4 3 \rangle , \wc\langle1 2 4 3 \rangle] \nonumber\\
&\ \cup \ & 3 1 4 2[\wc\langle1 2 4 3 \rangle , \wc\langle1 2 \rangle , \wc\langle2 1 \rangle , \wc\langle1 3 2 \rangle] \ \cup \ 3 1 4 2[\wc\langle1 2 \rangle , \wc\langle1 2 \rangle , \wc\langle1 3 2 \rangle , \wc\langle1 3 2 \rangle] \label{eqn:ambigu1}\\
\wc^{+}\langle1 2 \rangle &=& 1 \ \cup \ \ominus[\wc^{-}\langle 1 2 \rangle , \wc\langle1 2 \rangle]\label{eqn:ambigu2}\\
\wc\langle1 3 2 \rangle &=& 1 \ \cup \ \oplus[\wc^{+}\langle 1 3 2 \rangle , \wc\langle2 1 \rangle] \ \cup \ \ominus[\wc^{-}\langle 1 3 2 \rangle , \wc\langle1 3 2 \rangle]\label{eqn:ambigu3}\\
\wc\langle2 1 \rangle &=& 1 \ \cup \ \oplus [\wc^{+}\langle 2 1 \rangle , \wc\langle2 1 \rangle].\label{eqn:ambigu4}\\
\wc^{+}\langle 1 2 4 3 \rangle &=& \dots .\nonumber
\end{eqnarray}
\end{small}
\end{example}

This is a {\em simplified} version of the actual output of~\textsc{AmbiguousSystem}. \label{page:mention_des_simplifications}
For instance, with a literal application of the algorithm, instead of Equation~\eqref{eqn:ambigu1} we would get:
\begin{scriptsize}
\begin{align}
 &\wc\langle1 2 4 3 \rangle = 1 \ \cup \ \oplus[\wc^{+}\langle 1 2 4 3, 1 2 \rangle , \wc\langle 1 2 4 3, 1 3 2 \rangle] \ \cup \ \oplus[\wc^{+}\langle 1 2 4 3 \rangle , \wc\langle 1 2 4 3, 1 3 2, 2 1 \rangle] \ \cup \ \ominus[\wc^{-}\langle 1 2 4 3 \rangle , \wc\langle1 2 4 3 \rangle] \nonumber\\
 &\cup  \ 3 1 4 2[\wc \langle 1 2 4 3, 1 2 \rangle , 
				\wc \langle 1 2 4 3, 1 2 \rangle , 
				\wc \langle 1 2 4 3, 1 3 2 \rangle , 
				\wc \langle 1 2 4 3, 1 3 2 \rangle] 
\ \cup\ 3 1 4 2[\wc \langle 1 2 4 3, 1 2 \rangle , 
				\wc \langle 1 2 4 3, 1 2 \rangle , 
				\wc \langle 1 2 4 3, 1 3 2 \rangle , 
				\wc \langle 1 2 4 3, 1 3 2, 2 1 \rangle] \nonumber\\
& \cup\ 3 1 4 2[\wc \langle 1 2 4 3, 1 2 \rangle , 
				\wc \langle 1 2 4 3, 1 2 \rangle , 
				\wc \langle 1 2 4 3, 1 3 2, 2 1 \rangle , 
				\wc \langle 1 2 4 3, 1 3 2 \rangle] 
\ \cup\ 3 1 4 2[\wc \langle 1 2 4 3, 1 2 \rangle , 
				\wc \langle 1 2 4 3, 1 2 \rangle , 
				\wc \langle 1 2 4 3, 1 3 2, 2 1 \rangle , 
				\wc \langle 1 2 4 3, 1 3 2, 2 1 \rangle]\nonumber \\
& \cup\ 3 1 4 2[\wc \langle 1 2 4 3 \rangle , 
				\wc \langle 1 2 4 3, 1 2 \rangle , 
				\wc \langle 1 2 4 3, 1 3 2, 2 1 \rangle , 
				\wc \langle 1 2 4 3, 1 3 2 \rangle] 
\ \cup\ 3 1 4 2[\wc \langle 1 2 4 3 \rangle , 
				\wc \langle 1 2 4 3, 1 2 \rangle , 
				\wc \langle 1 2 4 3, 1 3 2, 2 1 \rangle , 
				\wc \langle 1 2 4 3, 1 3 2, 2 1 \rangle]
\nonumber
\end{align}
\end{scriptsize}
\noindent Nevertheless, this union can be simplified. 
The simplification process will be described more thoroughly in Section~\ref{ssec:simplifications}. 
We illustrate it by two examples. 
First, since a permutation that avoids $1 2$ or $132$ will necessarily avoid $1243$, 
the term $\oplus[\wc^{+}\langle 1 2 4 3, 1 2 \rangle , \wc\langle 1 2 4 3, 1 3 2 \rangle]$ rewrites as
$\oplus[\wc^{+}\langle 1 2 \rangle , \wc\langle1 3 2 \rangle]$ 
(see Proposition~\ref{prop:simplification_restriction} p.~\pageref{prop:simplification_restriction}). 
We can also remove some terms of the union, 
such as $$3 1 4 2[\wc \langle 1 2 4 3, 1 2 \rangle , \wc \langle 1 2 4 3, 1 2 \rangle , \wc \langle 1 2 4 3, 1 3 2 \rangle , \wc \langle 1 2 4 3, 1 3 2, 2 1 \rangle]$$ 
which is included in $3 1 4 2[\wc\langle1 2 \rangle , \wc\langle1 2 \rangle , \wc\langle1 3 2 \rangle , \wc\langle1 3 2 \rangle]$ 
(see Proposition~\ref{prop:simplification_union} p.~\pageref{prop:simplification_union}). 
Such simplifications can be performed on the fly, each time a new equation is computed. 

\medskip

We observe on Example~\ref{ex:algo1} that the system produced by~\textsc{AmbiguousSystem} (Algo.~\ref{alg:sys-ambigu}) is ambiguous in general. 
In this case, Equation~\eqref{eqn:ambigu1} gives an
ambiguous description of the class $\wc\langle1 2 4 3 \rangle$: the two terms with root $\oplus$ have non-empty
intersection, and similarly for root $3142$. 
Following the route of~\cite{AA05}, we could use inclusion-exclusion on this system. 
On Equation~\eqref{eqn:ambigu1} of Example~\ref{ex:algo1}, this would give: 
{\small
\begin{align*}
    \wc\langle1 2 4 3 \rangle = \ 1 \ & \cup \ \oplus [\wc^{+}\langle 1 2 \rangle , \wc\langle1 3 2 \rangle] \ \cup \ \oplus[\wc^{+}\langle 1 2 4 3 \rangle ,
\wc\langle2 1 \rangle] \ \setminus \ \oplus [\wc^{+}\langle 1 2 \rangle , \wc\langle2 1 \rangle]  \\
 & \cup \ \ominus [\wc^{-}\langle 1 2 4 3 \rangle , \wc\langle1 2 4 3 \rangle] \ \cup 3 1 4 2 [\wc\langle1 2 \rangle , \wc\langle1 2 \rangle ,
    \wc\langle1 3 2 \rangle , \wc\langle1 3 2 \rangle] \\
 & \cup \ 3 1 4
    2 [\wc\langle1 2 4 3 \rangle , \wc\langle1 2 \rangle , \wc\langle2
    1 \rangle , \wc\langle1 3 2 \rangle] \ \setminus \ 3 1 4
    2[\wc\langle1 2 \rangle , \wc\langle1 2 \rangle , \wc\langle2 1
    \rangle , \wc\langle1 3 2 \rangle].
\end{align*}
    }
As explained earlier, this is not the route we follow. 
In the next section, we explain how to modify the algorithm \textsc{AmbiguousSystem}
to obtain a combinatorial specification by introducing pattern containment constraints.

\section{A non-ambiguous combinatorial system, \emph{i.e.}, a combinatorial specification}
\label{sec:disambiguation}

The goal of this section is to describe an algorithm computing a specification
for any permutation class having finitely many simple permutations
-- see algorithm~\textsc{Specification} (Algo.~\ref{alg:final}).
This algorithm proceeds as~\textsc{AmbiguousSystem} (Algo.~\ref{alg:sys-ambigu} p.\pageref{alg:sys-ambigu}),
but also transforms each equation produced into a  non-ambiguous one.
The disambiguation of equations is performed by~\textsc{Disambiguate} (Algo.~\ref{alg:disambiguise}). 
This algorithm replaces ambiguous unions appearing in an equation by disjoint unions using complement sets,
in the spirit of Figure~\ref{fig:union}. 

\begin{figure}[htbp]
\begin{tikzpicture}[scale=.5]
\draw (0,1) circle (50pt);
\draw (2,1) circle (50pt);
\draw (1,-0.5) circle (50pt);
\draw (-2,1.8) node {$A$};
\draw (4,1.8) node {$B$};
\draw (1,-2.6) node {$C$};
\draw (-0.6,1.6) node {$1$};
\draw (2.6,1.6) node {$2$};
\draw (1,-1.3) node {$3$};
\draw (1,1.8) node {$4$};
\draw (0,0) node {$5$};
\draw (2,0) node {$6$};
\draw (1,0.5) node {$7$};
\draw (10,1.5) node[left] {$A \cup B \cup C$};
\draw (10,1.5) node[right] {$ =\ 1\ \uplus\ 2\ \uplus\ 3\ \uplus\ 4\ \uplus\ 5\ \uplus\ 6\ \uplus\ 7$};
\draw (10,0.5) node[right] {$ =  (A \cap \overline{B} \cap \overline{C}) \uplus ( \overline{A} \cap B \cap \overline{C}) \uplus (\overline{A} \cap \overline{B} \cap C) $};
\draw (10,-0.5) node[right] {$ \uplus\ (A \cap B \cap \overline{C}) \uplus ( A \cap \overline{B} \cap C) \uplus (\overline{A} \cap B \cap C) $};
\draw (10,-1.5) node[right] {$ \uplus\ (A \cap B \cap C)$};
\end{tikzpicture}
\caption{Rewriting unions as disjoint unions.}\label{fig:union}
\end{figure}
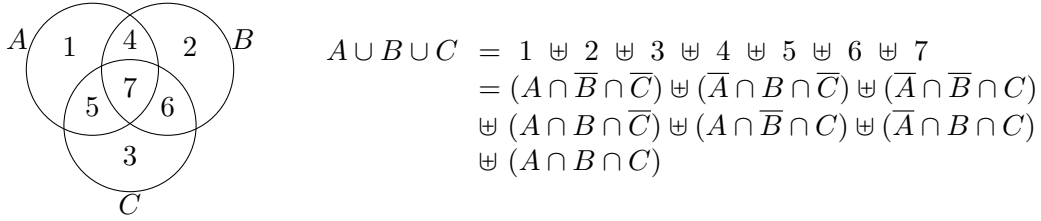

This may result in new terms appearing on the right side of the modified equation. 
Indeed, the terms of the system obtained from~\textsc{AmbiguousSystem} involve pattern avoidance constraints (which were denoted $\langle E \rangle$). 
Consequently, taking complements, we are left with new pattern \emph{containment} constraints as well (which we will denote $(A)$). 
These new terms need to be defined by an equation, to be added to the system. 
This is solved by the procedure~\textsc{EqnForRestriction} (Algo.~\ref{alg:ComputeEquationForRestriction}), 
whose working principle is similar to that of~\textsc{EqnForClass}.

Finally,~\textsc{EqnForRestriction} and~\textsc{Disambiguate} combine into~\textsc{Specification} (Algo.~\ref{alg:final}), 
which is our main algorithm.

\begin{algorithm}[htbp]
\DontPrintSemicolon 
\KwData{A finite basis of patterns defining ${\mathcal C}=Av(B)$ and the finite set $\mathcal{S}_\C$ of simple permutations in $\C$.}  
\KwResult{A combinatorial specification defining $\C$.}
\Begin{
	$\Sys \leftarrow $ \textsc{EqnForRestriction}($\wc$,$\Bstar$,$\emptyset$) \tcc*{See Algo.\,\ref{alg:ComputeEquationForRestriction}
\While{ there is an equation $F$ in $\Sys$ that has not been processed}{
	$F \leftarrow $ \textsc{Disambiguate}($F$) \tcc*{See Algo.\,\ref{alg:disambiguise}}
	\While{ there exists a right-only \rs $\wc^{\delta}\langle E\rangle(A)$ in equation $F$}{
		$\Sys \leftarrow \Sys \bigcup$ \textsc{EqnForRestriction}($\wc^{\delta}$,$E$,$A$). 
	}
}
\Return $\Sys$\;}
}
\caption{\textsc{Specification}($B,\mathcal{S}_\C$)} \label{alg:final}
\end{algorithm}

\subsection{Presentation of the method}

We start by introducing notation to deal with the pattern containment constraints. 

\begin{defi}
\label{dfn:restriction}
For any set $\cP$ of permutations, we define $\mathcal{P}\langle E \rangle(A)$  as the set of
permutations of $\mathcal{P}$  that avoid every pattern
of $E$ and contain every pattern of $A$: 
$$\cP\langle E \rangle(A)=\{\sigma \in \cP \mid \forall \pi \in E, \pi \not \preccurlyeq \sigma \text{ and } \forall \pi \in A, \pi \preccurlyeq \sigma\}.$$
When $\cP=\wc^{\delta}$ for $\delta \in\{~, +, -\}$ and ${\mathcal C}$ a permutation class, such a set is called a \emph{restriction}.

We also denote $\cP (E)= \cP\langle \emptyset \rangle (E)$.
\end{defi}

Restrictions are a generalization of permutations classes.
For $A=\emptyset$, ${\mathcal C}\langle E \rangle (\emptyset)$
is the standard permutation class ${\mathcal C}\langle E \rangle$.
Any permutation class can be written as $S\langle E \rangle$
with $S$ the set of all permutations and $E$ a set of patterns that may be infinite.
Likewise, any restriction can be written as $S\langle E \rangle (A)$ with $E$ a set of patterns that may be infinite,
but we can always choose a finite set for $A$: if $A$ is infinite,
then $S\langle E \rangle (A) = \emptyset$.

\begin{rem}\label{rem:conventions}
Our convention that no permutation class contains the empty permutation $\varepsilon$ 
implies that $\varepsilon \notin \wc^{\delta} \langle E \rangle(A)$, for any restriction $\wc^{\delta} \langle E \rangle(A)$. 
We can also make the assumption that $\varepsilon \notin E$ and $\varepsilon \notin A$, 
since $\wc^{\delta}\langle \{\varepsilon\} \cup E \rangle (A)=\emptyset$ and
$\wc^{\delta}\langle E \rangle(\{\varepsilon\} \cup A)=\wc^{\delta}\langle E \rangle (A)$ for any $E$ and $A$. 
Moreover we assume that $A\cap E$ is empty and that $1 \notin E$, 
otherwise $\wc^{\delta}\langle E \rangle(A)=\emptyset$. 
We finally assume that $1 \notin A$ since 
$\wc^{\delta}\langle E \rangle(A)=\wc^{\delta}\langle E \rangle(A \setminus \{1\})$.
\end{rem}

Restrictions are stable by intersection as
$\cP\langle E \rangle(A)\cap \cP\langle E' \rangle(A')=\cP\langle E
\cup E'\rangle(A\cup A')$.

\begin{defi}
A {\em \rt} is a set of permutations
$\pi[\D_{1},\D_{2},\ldots,\D_{n}]$ 
for $\pi$ a simple permutation, or $\oplus[\D^+_{1},\D_{2}]$ or $\ominus[\D^-_{1},\D_{2}]$,
where each $\D_{i}$ is a restriction of the form $\wc\langle E \rangle(A)$.
\end{defi}
By uniqueness of the substitution decomposition of a permutation (Theorem~\ref{thm:decomp_perm_AA05} p.\pageref{thm:decomp_perm_AA05}), 
restriction terms are stable by intersection and the intersection is performed component-wise for terms sharing the same root: 
$\pi[\D_{1},\ldots,\D_{n}] \cap \pi[{\mathcal T}_{1}, \ldots,{\mathcal T}_{n}] = \pi[\D_{1}\cap {\mathcal T}_{1},\ldots,\D_{n}\cap {\mathcal T}_{n}]$.

\bigskip

Now we have all the notions we need to present the general structure of our main algorithm: \textsc{Specification} (Algo.~\ref{alg:final} above).
Adapting to restrictions the ideas developed for classes in Section~\ref{sec:ambiguous},
we obtain a non-ambiguous equation for any restriction $\wc^{\delta} \langle E \rangle (A)$ for $\delta \in \{~, +, -\}$
within four steps (where $1_{(A)} = 1$ if $A = \emptyset$ and $\emptyset$ otherwise):
\begin{itemize}
\item \textbf{Step~1}, equation from the substitution decomposition of permutations:\\
$\wc^{\delta} \langle E \rangle (A) = 1_{(A)}\ 
\uplus\ \dots\ \uplus \textstyle\biguplus_{\pi \in \s_\C} \pi[\wc, \dots, \wc]\langle E \rangle (A)$

 \item \textbf{Step~2}, propagation of the avoidance constraints:\\
$\wc^{\delta} \langle E \rangle (A) = 1_{(A)}\ 
\uplus\ \dots\ \uplus \textstyle\biguplus_{\pi \in \s_\C} \cup_{\ell \in L_{\pi}} \pi[\wc\langle F^\ell_1 \rangle, \dots, \wc\langle F^\ell_n \rangle] (A)$

 \item \textbf{Step~3}, propagation of the containment constraints:\\
$\wc^{\delta} \langle E \rangle (A) = 1_{(A)}\ 
\uplus\ \dots\ \uplus \textstyle\biguplus_{\pi \in \s_\C} \cup_{i \in I_{\pi}} \pi[\wc\langle E^i_1 \rangle (A^i_1), \dots, \wc\langle E^i_n \rangle (A^i_n)]$

 \item \textbf{Step~4}, disambiguation of the equation:\\
$\wc^{\delta} \langle E \rangle (A) = 1_{(A)}\ 
\uplus\ \dots\ \uplus \textstyle\biguplus_{\pi \in \s_\C} \uplus_{k \in K_{\pi}} \pi[\wc\langle E'^k_1 \rangle (A'^k_1), \dots, \wc\langle E'^k_n \rangle (A'^k_n)]$
\end{itemize}

The algorithm \textsc{Specification} starts by computing a non-ambiguous equation for $\C = \wc \langle B^\star \rangle (\emptyset)$,
calling the procedures \textsc{EqnForRestriction} (performing Steps~1 to~3 above -- see details in Algo.~\ref{alg:ComputeEquationForRestriction} p.\pageref{alg:ComputeEquationForRestriction})
and \textsc{Disambiguate} (performing Step~4 -- see details in Algo.~\ref{alg:disambiguise} p.\pageref{alg:disambiguise}).
The specification for $\C$ is completed by using the same method of four steps to obtain an equation for each $\wc\langle E'^k_j \rangle (A'^k_j)$
appearing on the right side of the produced equation, and again recursively to obtain an equation for each restriction that appears in the system.

Even if we use it only to compute specifications for permutation classes (writing $\C = \wc \langle B^\star \rangle (\emptyset)$),
this algorithm allows more generally to obtain a specification for any restriction $S^{\delta} \langle E \rangle (A)$
such that $S\langle E \rangle$ has a finite number of simple permutations 
(implying that $E$ can be chosen finite).
Indeed we only have to replace \textsc{EqnForRestriction}($\wc$,$B$,$\emptyset$) with
\textsc{EqnForRestriction}($\wc^{\delta}$,$E$,$A$), where $\C=S\langle E \rangle$.

We prove in the next three subsections that the result of our main algorithm \textsc{Specification} has the following properties:

\begin{theo}\label{thm:systemeNonAmbigu}
Let $\C$ be a permutation class given by its finite basis $B$ and whose set $\s_\C$ of simple permutations is finite and given. Denote by $\Bstar$ the set of non-simple permutations of $B$. 

The result of \textsc{Specification}$(B,\s_\C)$ (Algo.~\ref{alg:final}) is a finite combinatorial system of equations of the form $\D_0 = {\mathit{1}}_{[1 \in \D_0]} \uplus \biguplus \pi[\D_1, \dots, \D_n]$, 
where ${\mathit{1}}_{[1 \in \D_0]}$ is the class consisting in a unique permutation of size $1$ if the permutation $1$ belongs to $\D_0$ and is empty otherwise, 
and where each $\D_i$ is a \rs $\wc^{\delta}\langle E \rangle(A)$ with $\delta \in \{~, +, -\}$, and $E$ and $A$ containing only normalized blocks\footnote{
Recall from Definition~\ref{def:block} (p.\pageref{def:block}) that the normalized blocks of $\gamma$ are 
the permutations $\gamma_I$ obtained when restricting $\gamma$ to any of its intervals $I$.
} of elements of $\Bstar$. Moreover $\C$ appears as the left part of an equation, and every $\D_i$ that appears
in the system is the left part of one equation of the system.

In particular this provides a combinatorial specification of $\C$.
\end{theo}

\begin{cor}
\label{cor:QCS2}
Let $\E$ be the system of equations output by \textsc{Specification}$(B,\s_\C)$.
Then the set $\{\D : \D \text{ is the left part of some equation of } \E \}$ is a query-complete set containing \C.
\end{cor}

\begin{proof}
This a direct consequence of the form of \E described in Theorem~\ref{thm:systemeNonAmbigu}.
\end{proof}

\subsection{Computing an equation for each \rs}\label{ssec:computeEqn}

Let $\wc^{\delta}\langle E \rangle(A)$ be a \rs. Our goal here is to find
an equation describing this \rs using smaller \rts
(smaller w.r.t.\ inclusion).

If $A = \emptyset$, this is exactly the problem addressed in Section~\ref{sec:addConstraint}
and solved by pushing down the pattern avoidance constraints with the procedure \textsc{AddConstraints} of~\textsc{EqnForClass} (Algo.~\ref{alg:comp-eqn}).
The procedure~\textsc{EqnForRestriction} (Algo.~\ref{alg:ComputeEquationForRestriction}) below shows how to propagate also the pattern \emph{containment} constraints induced by $A \neq \emptyset$.

\SetKwBlock{AMfunc}{\textsc{AddMandatory}}{end}
\begin{algorithm}[htbp]
\DontPrintSemicolon 
\KwData{$\wc^{\delta}, E,A$ with $E,A$ sets of permutations, $\wc^{\delta}$  given by $\s_\C$ and $\delta \in \{~, +, -\}$.}
\KwResult{An equation defining $\wc^{\delta}\langle E \rangle(A)$ as a union of \rts.}
\SetKw{In}{in} 
\Begin{
$\mathcal{E} \leftarrow$ Equation \eqref{eqn:Wc1} or \eqref{eqn:Wc2} or \eqref{eqn:Wc3} (depending on $\delta$), replacing the left part $\wc^{\delta}$ by $\wc^{\delta}\langle E\rangle(A)$\;

\ForEach{avoidance constraint $\gamma$ \In $E$}{
	\ForEach{$t=\pi[\C_1, \dots , \C_n]$ that appears \In $\mathcal{E}$}{
		$t\leftarrow$ \textsc{AddConstraints}$(\pi[\C_1, \dots , \C_n], \gamma )$ 
		\tcc*{See Algo.\,\ref{alg:comp-eqn} (p.\pageref{alg:comp-eqn})}
	}
}
\ForEach{containment constraint $\gamma$ \In $A$}{
	\ForEach{$t=\pi[\C_1, \dots , \C_n]$ that appears \In $\mathcal{E}$}{
		$t\leftarrow$ \textsc{AddMandatory}$(\pi[\C_1, \dots , \C_n], \gamma )$ 
	}
}
\Return $\mathcal{E}$ \;
}

\bigskip
\SetKwSty{textsc}
\AMfunc(\params{$\pi[\D_1, \dots, \D_n],\gamma$}){\SetKwSty{textbf}
\Return a rewriting of $\pi[\D_1, \dots, \D_n] (\gamma)$ as a union of \rts using Eq.~\eqref{eq:AddMandatory} below\;
}
\caption{\textsc{EqnForRestriction}$(\wc^{\delta},E,A)$ \label{alg:ComputeEquationForRestriction}}
\end{algorithm}

The pattern \emph{containment} constraints are propagated by \textsc{AddMandatory}, in a very similar fashion to the pattern \emph{avoidance} constraints propagated by \textsc{AddConstraints}.
To compute $t(\gamma)$ for $\gamma$ a permutation and $t = \pi[\D_1, \dots, \D_n]$ a restriction term, we first  compute all embeddings of $\gamma$ into $\pi$.
In this case, a permutation belongs to $t(\gamma)$ if and only if at least one embedding is satisfied.
Then $\pi[\D_1, \dots, \D_n] ( \gamma )$ rewrites as a union of sets of the form $\pi[\D_1(\gamma_1), \dots, \D_n(\gamma_n)]$ where, 
for all $i$, $\gamma_i$ is a normalized block of $\gamma$ which may be empty or $\gamma$ itself 
(recall that if $\gamma_i$ is empty, then  $\D_{j}(\gamma_i) = \D_j$). More precisely:
\begin{prop}\label{prop:containment=blocs}
Let $\pi$ be a permutation of size $n$ and $\D_1, \dots, \D_n$ be sets of permutations.
For any permutation $\gamma$, let $\{\alpha_1, \dots, \alpha_\ell\}$ be the set of embeddings of $\gamma$ in $\pi$, then 
\begin{equation}\label{eq:AddMandatory}
\pi[\D_1, \dots ,\D_n] ( \gamma )  = 
	\bigcup_{i=1}^{\ell} \pi[\D_{1}(\gamma_{\alpha_{i}(1)}),\D_{2}(\gamma_{\alpha_{i}(2)}),\ldots,\D_{n}(\gamma_{\alpha_{i}(n)})].
\end{equation}
\end{prop}

For instance, for $t = 3142[\D_{1},\D_{2},\D_{3},\D_{4}]$ and $\gamma = 546312$, there are $12$ embeddings of~$\gamma$ into $3142$ (see Table~\ref{table:embeddings} p.\pageref{table:embeddings}),
and the embedding $3142\{21,\varepsilon,1,312\}$ contributes to the above union with the term $3142[\D_{1}(21),\D_{2},\D_{3}(1),\D_{4}(312)]$.

\begin{proof}
Let $\sigma \in \pi[\D_1, \dots, \D_n]$; then $\sigma = \pi[\sigma_1, \dots, \sigma_n]$ where each $\sigma_k\in\D_k$. 
Then $\sigma$ contains~$\gamma$, if and only if there exists an embedding $\alpha_i$ of $\gamma  = \pi[\gamma_{\alpha_i(1)} \dots \gamma_{\alpha_i(n)}]$
such that $\gamma_{\alpha_i(j)}$ is a pattern of $\sigma_j$ for all $j$.
\end{proof}

Hence, any \rt $t = \pi[\D_1, \dots, \D_n](\gamma)$ rewrites as a (possibly ambiguous) union of \rts.

\subsection{Disambiguation procedure}\label{sec:disambiguate}

\medskip

As explained earlier,~\textsc{Disambiguate} (Algo.~\ref{alg:disambiguise}) disambiguates equations introducing complement sets. 

\begin{algorithm}[htb]
\DontPrintSemicolon 
\KwData{A potentially ambiguous equation $F$ defining a restriction
\KwResult{A non-ambiguous equation equivalent to $F$ }
\Begin{
	\ForEach{root $\pi$ that appears several times in $F$}{
	Replace the union of the restriction terms of $F$ whose root is $\pi$ by a disjoint union using 
	Equations~\eqref{eq:DisambiguateRoot}, \eqref{eq:ComplementTerm} and \eqref{eq:14} below.\;
}
\Return $F$\;}
}
\caption{\textsc{Disambiguate}($F$)} \label{alg:disambiguise}
\end{algorithm}

Every equation produced by~\textsc{EqnForRestriction} is written as $t =1 \cup t_{1} \cup t_{2} \cup t_{3} \ldots
\cup t_{k}$ where the sets $t_{i}$ are \rts (some $\pi[\D_{1},\D_{2},\ldots,\D_{n}]$) and $t$ is a
\rs (some $\wc^{\delta}\langle E \rangle(A)$).
By uniqueness of the substitution decomposition of a permutation, \rts
of this union which have different roots $\pi$ are disjoint. Thus for
an equation we only need to disambiguate unions of terms with same
root. 
For example in Equation~\eqref{eqn:ambigu1} (p.\pageref{eqn:ambigu1}), there are two pairs
of ambiguous terms which are terms with root $3 1 4 2$ and terms with root $\oplus$.
Every ambiguous union can be written in an unambiguous way:
\begin{prop}\label{prop:union}
Let $A_1, \dots, A_n$ be $n$ sets and for each of them denote
$\overline{A_i}$ the complement of $A_i$ in any set containing
$\bigcup_{i=1}^{n} A_i$.  The union $\bigcup_{i=1}^{n} A_i$ rewrites
as the disjoint union of the $2^n-1$ sets of the form
$\bigcap_{i=1}^{n} X_i$ with $X_i \in \{A_i, \overline{A_i}\}$ and at
least one $X_i$ is equal to $A_i$.
\end{prop}

This proposition is the starting point of the disambiguation. See Figure~\ref{fig:union} (p.\pageref{fig:union}) for an example. 
In order to use Proposition~\ref{prop:union}, we have to choose in which set we take complements.

\begin{defi}
For any restriction term $t$, we define its complement $\overline{t}$ as follows:
\begin{itemize}
 \item if $t=\pi [\D_1, \dots, \D_n]$ with $\pi$ simple and for all $i$, $\D_i \subset \wc$,
we set $\overline{t}=\pi [\wc, \dots, \wc] \setminus t$;
 \item if $t = \oplus [\D_1^+, \D_2]$ with $\D_1,\D_2 \subset \wc$, we set $\overline{t}=\oplus [\wc^+, \wc] \setminus t$;
 \item if $t = \ominus [\D_1^-, \D_2]$ with $\D_1,\D_2 \subset \wc$, we set $\overline{t}=\ominus [\wc^-, \wc] \setminus t$.
\end{itemize}
Moreover for any restriction $\D$ of $\wc^{\delta}$ with $\delta \in \{~, +, -\}$,
we set $\overline{\D}=\wc^{\delta} \setminus \D$.
\end{defi}

From Proposition~\ref{prop:union}, every ambiguous union of restriction terms sharing the same root in an equation of our system can be
written in the following unambiguous way:
\begin{equation}\label{eq:DisambiguateRoot}
	\textstyle\bigcup_{i=1}^{k} t_{i}=\textstyle\biguplus_{X
	  \subseteq [1\ldots k], X \not= \emptyset} \bigcap_{i \in X} t_{i}
	\cap \bigcap_{i \in \overline{X}} \overline{t_{i}}.
\end{equation}

For instance, consider terms with root $3142$ in Equation~\eqref{eqn:ambigu1}:
$t_{1} = 3 1 4 2 [\wc\langle1 2 \rangle , \wc\langle1 2 \rangle , $ $
\wc\langle1 3 2 \rangle , \wc\langle1 3 2 \rangle]$ and $t_{2} = 3 1
4 2 [\wc\langle1 2 4 3 \rangle , \wc\langle1 2 \rangle , \wc\langle2
1 \rangle , \wc\langle1 3 2
\rangle]$.  Equation~\eqref{eq:DisambiguateRoot} applied to
$t_1 \cup t_2$ in Equation~\eqref{eqn:ambigu1} gives an expression of the form
\[\wc\langle1243\rangle = 1 \cup \oplus[\ldots] \cup \oplus[\ldots] \cup \ominus[\ldots] 
\cup (t_{1} \cap t_{2}) \uplus (t_{1} \cap \overline{t_{2}}) \uplus (\overline{t_{1}} \cap t_{2})\textrm{.}
\]

We now explain how to compute the complement $\overline{t}$ of a \rt $t$.

\begin{prop}\label{prop:complementPi}
Let $t=\pi [\D_1, \dots, \D_n]$ be a \rt.
Then $\overline{t}$ is the disjoint union of the $2^n-1$ sets of the form $\pi [\X_1, \dots, \X_n]$ with $\X_i \in \{\D_i, \overline{\D_i}\}$, 
and not all $\X_i$ are equal to $\D_i$:
\begin{equation}\label{eq:ComplementTerm}
\overline{t}=\biguplus_{\X \subseteq \{1,\ldots,n\}, \X \not= \emptyset} \pi[\D'_{1},\ldots,\D'_{n}] \mbox{ where }\D'_{i} = \overline{\D_{i}}\mbox{  if }i \in \X\mbox{ and }\D'_{i} = \D_{i}\mbox{  otherwise},
\end{equation}
\end{prop}

For example,
$\overline{\ominus[\D_{1},\D_{2}]} =
\ominus[\D_{1},\overline{\D_{2}}] \uplus \ominus[\overline{\D_{1}},\D_{2}]
\uplus \ominus[\overline{\D_{1}},\overline{\D_{2}}]$.

\begin{proof}
Recall that $\overline{t}=\pi [\wc, \dots, \wc] \setminus t$. Let $\sigma = \pi[\sigma_1, \dots, \sigma_n] \in \overline{t}$. Assume that for each $i$ $\sigma_i \in \D_i$, by uniqueness of substitution
decomposition we get a contradiction. Therefore $\sigma  \in \pi [\X_1, \dots, \X_n]$ with $\X_i \in \{\D_i, \overline{\D_i}\}$, 
and not all $\X_i$ are equal to $\D_i$.

Conversely if $\sigma \in \pi [\X_1, \dots, \X_n]$ with $\X_i \in \{\D_i, \overline{\D_i}\}$ and at least one $\X_i$ is equal to $\overline{\D_i}$,
then $\sigma \in \pi [\wc, \dots, \wc]$ and $\sigma \notin t$.

Finally for such sets $\X_i$ that are distinct, the sets $\pi [\X_1, \dots, \X_n]$ are disjoints.
Indeed $\D_i \cap \overline{\D_i}$ is empty and the writing as $\pi [\sigma_1, \dots, \sigma_n]$ is unique. 
So the union describing $\overline{t}$ is disjoint.
The proof is similar when $\pi = \oplus$ or $\pi = \ominus$.
\end{proof}

Proposition~\ref{prop:complementPi} shows that $\overline{t_{i}}$ is not a \rt in general.
However it can be expressed as a disjoint union of some $\pi[\D'_{1},\ldots,\D'_{n}]$,
where the $\D'_i$ are either \rss or complements of \rss.
The complement operation being pushed from \rts down to
\rss, we now compute~$\overline{\D}$, for a given \rs $\D = \wc^{\delta}\langle
E\rangle(A)$, $\overline{\D}$ denoting the
set of permutations of $\wc^{\delta}$ that are not in~$\D$. Note
that, given a permutation $\sigma$ of $A$, then any permutation $\tau$ of
$\wc^{\delta}\langle \sigma \rangle$ is in $\overline{{\mathcal
    D}}$ because $\tau$ avoids $\sigma$ whereas permutations of
$\D$ must contain $\sigma$. Symmetrically, if a
permutation $\sigma$ is in  $E$ then permutations of
$\wc^{\delta}(\sigma)$ are in $\overline{{\mathcal
    D}}$. It is straightforward to check that $\textstyle
\overline{\wc^{\delta}\langle E \rangle (A)} = \big[
\bigcup_{\sigma \in E} \wc^{\delta}(\sigma)\big]
\bigcup \big[ \bigcup_{\sigma \in A}
\wc^{\delta}\langle\sigma\rangle\big]$.
Unfortunately this expression is ambiguous.
As before, we can rewrite it as an unambiguous union:

\begin{prop}\label{prop:complementRestriction}
Let $\D = \wc^{\delta}\langle E \rangle (A)$ be a restriction with $\delta \in \{~, +, -\}$, $k = |E|$ and $\ell = |A|$.
Then $\overline{\D}$ is the disjoint union of the $2^{k+\ell}-1$ restrictions $\wc^{\delta} \langle E' \rangle(A')$ with 
$(E',A')$ a partition of $E \uplus A$ such that $(E',A') \neq (E,A)$. In other words,
\begin{equation} \label{eq:14}
\overline{\wc^{\delta}\langle E \rangle (A)}
= \biguplus_{\underset{X \times Y \not=
    \emptyset\times\emptyset}{{X\subseteq A, Y \subseteq E}}}
\wc^{\delta}\langle X\cup\overline{Y} \rangle (Y \cup \overline{X}) \textrm{, where } \overline{X} = A \setminus X \textrm{ and }\overline{Y} = E \setminus Y \textrm{.}
\end{equation}
\end{prop}

\begin{proof}
Let $\tau \in \overline{\D} = \wc^{\delta} \setminus \D$.
Define $E' = \{\pi \in E \cup A \ |\ \pi \npreceq \tau \}$ and $A' = \{\pi \in E \cup A \ |\ \pi \preccurlyeq \tau \}$.
Observe that $\tau \in \wc^{\delta}\langle E' \rangle (A')$.
Moreover $E' \cap A' = \emptyset$, $E' \cup A' = E \cup A$ and $(E',A') \neq (E,A)$ otherwise $\tau$ would be in $\D$.

Conversely let $\tau \in \wc^{\delta}\langle E' \rangle (A')$ with $(E',A')$ a partition of $E \uplus A$ 
such that $(E',A') \neq (E,A)$, then $\tau \in \wc^{\delta}$.
As $(E',A') \neq (E,A)$, either there is some $\sigma$ in $E$ that $\tau$ contains, or there is some $\si$ in $A$ that $\tau$ avoids.
In both cases, $\tau \notin \D$ thus $\tau \in \wc^{\delta} \setminus \D = \overline{\D}$.

Finally for distinct partitions $(E',A')$ of $E \cup A$, the sets $\wc^{\delta}\langle E' \rangle (A')$ are disjoints.
Indeed a permutation in two sets of this form would have to both avoid and contain some permutation of $E \cup A$, which is impossible.
\end{proof}

Proposition~\ref{prop:complementRestriction} shows that $\overline{\D}$ is not a \rs in general but can be expressed as a disjoint
union of \rss. For instance,
\begin{small}
\begin{align*}
\overline{\wc\langle 231,123 \rangle (4321)} =
& \,\wc(123,\,231,\,4321) \,\uplus\,
\wc\langle 123 \rangle (231,4321) \,\uplus\,
\wc\langle 231 \rangle (123,4321) \,\uplus\,\\
& \,
\wc\langle 4321 \rangle (123,231) \,\uplus\,
\wc\langle 123,4321 \rangle (231) \,\uplus\,
\wc\langle 231,4321 \rangle (123) \,\uplus\,
\wc\langle 123,231,4321 \rangle.
\end{align*}
\end{small}
Moreover by uniqueness of the substitution decomposition, 
$$\pi[\D_1, \dots, \D_k \uplus \D'_k, \dots, \D_n] = \pi[\D_1, \dots, \D_k, \dots, \D_n] \uplus  \pi[\D_1, \dots, \D'_k, \dots, \D_n].$$
Therefore using Equations~\eqref{eq:ComplementTerm} and \eqref{eq:14} we have that for any \rt $t_i$, its complement $\overline{t_{i}}$ can be expressed as a disjoint
union of \rts:
\begin{prop}
For any \rt $t$, its complement $\overline{t}$ can be written as a disjoint union of \rts.
More precisely if $t=\pi [\D_1, \dots, \D_n]$ with $D_i = \wc^{\delta_i}\langle E_i \rangle (A_i)$ and $m = \sum_{i=1}^{n} |E_i| + |A_i|$,
then $\overline{t}$ is the disjoint union of the $2^m-1$ \rts
 $t=\pi [\D'_1, \dots, \D'_n]$ such that for all $i$, $D'_i = \wc^{\delta_i}\langle E'_i \rangle (A'_i)$ with 
$(E'_i,A'_i)$ a partition of $E_i \uplus A_i$, and there exists $i$ such that $(E'_i,A'_i) \neq (E_i,A_i)$.
\end{prop}

By distributivity of intersection over disjoint union, Equation~\eqref{eq:DisambiguateRoot} above can therefore be
rewritten as a disjoint union of intersection of \rts. Because \rts
are stable by intersection, the right-hand side of Equation~\eqref{eq:DisambiguateRoot}
is hereby written as a disjoint union of \rts. This leads to the following result:
\begin{prop}
Any union of \rts can be written as a disjoint union of \rts, and this can be done algorithmically using 
Equations~\eqref{eq:DisambiguateRoot}, \eqref{eq:ComplementTerm} and \eqref{eq:14}.
\end{prop}

Altogether, for any equation of our system, we are able to rewrite it
unambiguously with disjoint unions of \rts, using the algorithm \textsc{Disambiguate}.

\subsection{\textsc{Specification}: an algorithm computing a combinatorial specification describing $\C$}
\label{ssec:proof_main_Th}

The procedures described above finally combine into 
\textsc{Specification} (Algo.~\ref{alg:final} p.\pageref{alg:final}), the main algorithm of this article, 
which computes a combinatorial specification for \C. Equations of the specification are computed iteratively, 
starting from the one for \C: this is achieved
using~\textsc{EqnForRestriction} (Algo.~\ref{alg:ComputeEquationForRestriction}) described
in Section~\ref{ssec:computeEqn}, which produces equations that may be ambiguous.
As we do not know how to decide whether an equation is ambiguous or not,
we apply \textsc{Disambiguate} (Algo.~\ref{alg:disambiguise}) to every equation produced.
Since some new right-only \rss may appear during this process, 
to obtain a complete system we compute iteratively
equations defining these new \rss using again \textsc{EqnForRestriction}.

The termination of~\textsc{Specification} is easy to prove.
Indeed, for all the \rss $\wc^{\delta}\langle E\rangle(A)$ that are considered in the inner while loop of \textsc{Specification}, every permutation in the sets $E$ and $A$ is a pattern of some element of the basis $B$ of $\C$. And since $B$ is finite, there is a finite number of such restrictions.
Consequently, the algorithm produces an unambiguous system (\emph{i.e.} a combinatorial specification)
which is the result of a finite number of iterations of computing equations followed by their disambiguation.

As for~\textsc{AmbiguousSystem} (Algo.~\ref{alg:sys-ambigu}), the complexity of~\textsc{Specification} depends on the number of equations produced, that we discuss in Section~\ref{ssec:bounds}.

\subsection{Size of the specification obtained}
\label{ssec:bounds}

The complexity of~\textsc{Specification} (Algo.~\ref{alg:final}) depends on the number of equations in the computed specification, which may be quite large.\footnote{
The same goes for the complexity of \textsc{AmbiguousSystem}.
The bounds given in Proposition~\ref{prop:upper}, Corollary~\ref{cor:upper} and Proposition~\ref{prop:lower}
also hold for the number of equations of the system output by \textsc{AmbiguousSystem}.}
We were not able to determine exactly how big it can be, and could only provide in Proposition~\ref{prop:upper} and Corollary~\ref{cor:upper} upper bounds on its size (\emph{i.e.}, number of equations) which seems to be (very) overestimated.  
We leave open the question of improving the upper bound on the size of the specification produced by our method.
However, we point out that such an upper bound cannot be less than an exponential (in the sum of the sizes of the excluded patterns).
Indeed, we give in Proposition~\ref{prop:lower} a generic example where our method produces such an exponential number of equations in the specification.
Nevertheless, this convoluted example was created on purpose, and in many cases the number of equations obtained is not so high.

\begin{prop}\label{prop:upper}
Let $\C = Av(B)$ and $\Bstar$ be the set of non-simple permutations of $B$. 
Let~$P^\star$ be the set of normalized blocks of permutations of $\Bstar$.
The number of equations in the specification of \C computed by the procedure~\textsc{Specification} is at most $3^{|P^\star|}$. 
\end{prop}
\begin{proof}
In the specification we obtain, every equation is of the form
\[
\wc^{\delta}\langle E \rangle(A) = \dots \text{, where } \delta \in \{~, +, -\}, E \cup A \subseteq P^\star \text{.}
\]
As explained in Remark~\ref{rem:conventions} (p.\pageref{rem:conventions}), we can further assume that $E \cup A \subseteq P^\star \setminus \{1\}$ 
and that $E \cap A = \emptyset$. 
The number of equations is then bounded by the number of such triplets $(\delta, E, A)$ which is $3^{1+(|P^\star|-1)}$.
\end{proof}

Using the previous proposition and the fact that the number of blocks of a permutation of size $k$ is less than $k^2$, we have the following consequence:
\begin{cor}
\label{cor:upper}
Let $\C = Av(B)$ and $t = \sum_{\pi \in B} |\pi|$.
The number of equations in the specification of \C computed by~\textsc{Specification} is at most $3^{t^2}$.
\end{cor}

However, \textsc{Specification} is designed to compute only the equations we need,  and the number of equations produced is in practice much smaller. 
See for instance the example of Section~\ref{sec:ex}, where $\Bstar = \{1243, 2341\}$ and $P^\star = \{1, 12, 21, 123, 132, 1243, 2341\}$: 
the upper bound of Proposition~\ref{prop:upper} is $3^7 = 2187$, but only $16$ equations are effectively computed. 
But as shown by the following proposition, 
\textsc{Specification} produces in the worst case a specification with 
a number of equations that is exponential in $t$. 

\begin{prop}\label{prop:lower}
For each $n \geq 4$, there exists a class $\C_n = Av(B_n)$ whose specification computed by the procedure~\textsc{Specification} (Algo.~\ref{alg:final}) has at least $2^{s}$ equations, 
where the sum $t$ of the sizes of the elements of $B_n$ is approximately $s$, in the sense that $ 1 < t/s \ll s$.
\end{prop}

\begin{proof}
For any $n\geq 4$, denote by $\s_n$ the set of simple permutations of size $n$ and by $s_n$ its cardinality. 
Remember from~\cite{AAK03} that $s_n \sim \frac{n!}{e^2}$. 
Fix some $n \geq 4$, let $s=s_{2n-1}$, and define $\gamma = \oplus[\tau_{1}, \dots, \tau_{s}]$ with $\{ \tau_{i} : 1 \leq i \leq s\}=\s_{2n-1}$ 
and $B_n = \s_{2n} \cup \{\gamma \}$. 
Note that $B_n$ is an antichain, and consider the class $\C_n = Av(B_n)$. 
The sum of the sizes of the elements of $B_n$ is $t = 2n\cdot s_{2n} + (2n-1)\cdot s_{2n-1}$. 
Thus $s < t$ and $t/s \sim 4n^2$. 
Using Stirling formula, $t$ and $s$ are both of order $(\frac{2n}{e})^{2n+\text{ constant}}$.

It is not hard to see that $\C_n$ contains a finite number of simple permutations. 
Indeed, it follows from \cite{SchmerlTrotter1993} (see details in~\cite{PR11} for instance) 
that $\C_n$ contains no simple permutations of size $2n$ or more. 
Note also that the simple permutation $\pi = 246135$ is small enough that it belongs to $\C_n$. 
Moreover, $B_n^\star = \{\gamma\}$ and from Proposition~\ref{prop:system}, $\C_n = \hat{\C_n} \langle \gamma \rangle$. 

We claim that the computations performed by \textsc{AddConstraints}$(\pi[\wc_n, \ldots, \wc_n] , \gamma)$ 
in~\textsc{EqnForRestriction}$(B_n)$ (Algo.~\ref{alg:ComputeEquationForRestriction} p.\pageref{alg:ComputeEquationForRestriction}) 
will create at least $2^{s}$ right-only terms, 
thus giving rise to at least $2^{s}$ additional equations in the specification of $\C_n$. 
More precisely, with notation from Proposition~\ref{prop:restrictions=blocs} (p.\pageref{prop:restrictions=blocs}),
we prove that that for each subset $E$ of $\s_{2n-1}$, 
there exists a tuple $(k_1, \ldots, k_{\ell}) \in K^\pi_{\gamma}$ such that 
$E_{2,k_1, \ldots k_{\ell}} = E \cup \{\gamma\}$,
ensuring that $\hat{\C_n} \langle E \cup \{\gamma\} \rangle$ appear in the system of equations.

Let us start by classifying the embeddings of $\gamma$ in $\pi$ into three categories. 
For all $i \in [1..s]$, let us denote by $\alpha_{i}$ the embedding of $\gamma$ in $\pi$ that sends $\tau_{i}$ to $\pi_2$,
$\oplus[\tau_{1}, \dots, \tau_{i-1}]$ to $\pi_1$ and $\oplus[\tau_{i+1}, \dots, \tau_{s}]$ to $\pi_3$; 
and let $\alpha_{s+1}$ be the embedding of $\gamma$ in $\pi$ that sends $\gamma$ to $\pi_2$. 
The remaining embeddings of $\gamma$ in $\pi$ are denoted $\{\alpha_{i}\ |\ i \in [s\!+\!2\,..\,\ell] \}$.

Note that for any $i \neq s+1$, there exists some $k_{i} \in [1..6]$, with $k_{i} \neq 2$, 
such that $\gamma_{\alpha_{i}(k_{i})} \neq \varepsilon$ and $\gamma_{\alpha_{i}(k_{i})} \neq 1$ 
(since for each $j$, $\tau_j$ is simple and $|\pi| < |\tau_j|$, thus $\tau_j$ has to be entirely embedded in some $\pi_k$). 
This remark allows to consider, for each subset $E$ of $\s_{2n-1}$, a tuple $(k_1, \ldots k_{\ell})$ defined as follows. 
For $i \in [1..s]$, we set $k_i =2$ if $\tau_i \in E$, and otherwise we choose $k_i \neq 2$ such that 
$\gamma_{\alpha_{i}(k_{i})} \neq \varepsilon$ and $\gamma_{\alpha_{i}(k_{i})} \neq 1$. 
We set $k_{s+1} =2$. 
For $ i \in [s\!+\!2\,..\,\ell]$, we choose $k_i \neq 2$ such that 
$\gamma_{\alpha_{i}(k_{i})} \neq \varepsilon$ and $\gamma_{\alpha_{i}(k_{i})} \neq 1$. 
Consequently, the following properties hold: \\
\hspace*{0.8cm} - $(k_1, \ldots k_{\ell}) \in K^\pi_{\gamma}$ (defined in Equation~\eqref{eq:propagate} p.\pageref{eq:propagate}); \\
\hspace*{0.8cm} - $k_i = 2$ if and only if $\tau_i \in E$ or $i=s+1$; \\
\hspace*{0.8cm} - for $\tau_i \in E$, $\gamma_{\alpha_i(k_i)} = \tau_i$; \\
\hspace*{0.8cm} - and for $i=s+1$, $\gamma_{\alpha_i(k_i)} = \gamma$. \\
This ensures that $E_{2,k_1, \ldots k_{\ell}} = E \cup \{\gamma\}$ as claimed.
\end{proof}

\subsection{Simplifications on the fly}
\label{ssec:simplifications}

During the computation of the equations by~\textsc{Specification} (Algo.~\ref{alg:final}), 
many restriction terms appear, that may be redundant or a bit more intricate than necessary.
For instance, in the equations obtained when pushing down the constraints in the subtrees using the rewriting described in 
Propositions~\ref{prop:restrictions=blocs} (p.\pageref{prop:restrictions=blocs}) and~\ref{prop:containment=blocs} (p.\pageref{prop:containment=blocs}), 
some element of a given union may not be useful because it may be included in some other element of the union.
We simplify these unions by deleting useless elements, using Proposition~\ref{prop:simplification_union} below.
Moreover, 
when a restriction $\D$ of the form $\wc^{\delta} \langle E \rangle (A)$ arise, 
it can often be written as $\D = \wc^{\delta} \langle E' \rangle (A')$ 
with $E'$ (resp.\ $A'$) having fewer elements than $E$ (resp.\ $A$).
We use the description having as few elements as possible, thanks to Proposition~\ref{prop:simplification_restriction}.
Proposition~\ref{prop:simplification_trivial2} further allows some trivial simplifications. 

\medskip

For any sets $E$ and $A$ of patterns, 
$\min E$ (resp.\ $\max A$) denote the subset of $E$ (resp.\ $A$) 
containing all minimal (resp.\ maximal) elements of $E$ (resp.\ $A$) for the pattern order $\preccurlyeq$. 

\begin{prop}
In any equation, every restriction $\wc^{\delta}\langle E\rangle(A)$ may be replaced by \\
$\wc^{\delta} \langle \min E\rangle (\max A)$ 
without modifying the set of permutations described. 
\label{prop:simplification_restriction} 
\end{prop}

\begin{proof}
The proposition is an immediate consequence of $\wc^{\delta}\langle E\rangle(A) = \wc^{\delta}\langle \min E\rangle(\max A)$. 
This identity follows easily from two simple facts: 
if a permutation $\sigma$ avoids $\pi$ then $\sigma$ avoids all patterns containing $\pi$; 
and if $\sigma$ contains $\pi$ then $\sigma$ contains all patterns contained in $\pi$. 
\end{proof}

\begin{prop}
The restriction term $\pi[\C_1, \ldots, \C_n]$ is the empty set if and only if 
there exists $i$ such that $\C_i = \emptyset$. 
In this case, $\pi[\C_1, \ldots, \C_n]$ may be removed from any union of restriction terms 
without modifying the set of permutations described by this union. 
\label{prop:simplification_trivial2}
\end{prop}

\begin{prop}
Consider a union of restriction terms containing two terms with the same root 
$\pi[\C_1, \ldots, \C_n]$ and $\pi[\D_1, $ $ \ldots, \D_n]$. 
If for all $i$, $\C_i \subseteq \D_i$, then we may remove $\pi[\C_1, \ldots, \C_n]$ from this union 
without modifying the set of permutations it describes. 
\label{prop:simplification_union} 
\end{prop}

\begin{proof}
If for all $i$, $\C_i \subseteq \D_i$, then $\pi[\C_1, \ldots, \C_n] \subseteq \pi[\D_1, \ldots, \D_n]$, 
giving the result immediately. 
\end{proof}

Performing the simplifications of Proposition~\ref{prop:simplification_restriction} requires 
that we can compute $\min E$ and $\max A$ effectively. 
This can be done naively 
by simply checking all pattern relations between all pairs of elements of $E$ (resp.\ $A$). 
Similarly, to perform all simplifications induced by Propositions~\ref{prop:simplification_trivial2} and \ref{prop:simplification_union}, 
we would need to be able to decide whether $\C_i$ is empty or $\C_i \subseteq \D_i$, for any restrictions $\C_i$ and $\D_i$. 
Lemmas~\ref{lem:simplification_trivial1} and \ref{lem:simplification} below give sufficient conditions for the emptiness and 
the inclusion of a restriction into another, hence allowing to perform some simplifications. 

\begin{lem}
A restriction $\wc^{\delta}\langle E\rangle(A)$ is the empty set as soon as 
there exist $\pi \in E$ and $\sigma \in A$ such that $\pi \preccurlyeq \sigma$. 
\label{lem:simplification_trivial1}
\end{lem}

\begin{lem}\label{lem:simplification}
Let $\delta \in \{~, +, -\}$  and $E_1$, $E_2$, $A_1$ and $A_2$ be any sets of permutations.
Then $\wc^{\delta}\langle E_1\rangle(A_1) \subseteq \wc^{\delta}\langle E_2\rangle(A_2)$ as soon as:
\begin{itemize}
\item for each $\pi \in E_2$, there exists $\tau \in E_1$ such that $\tau \preccurlyeq \pi$, and

\item for each $\pi \in A_2$, there exists $\tau \in A_1$ such that $\pi \preccurlyeq \tau$.
\end{itemize}
\end{lem}

\begin{proof}
Assume the two conditions of the statement are satisfied.
Consider $\sigma \in \wc^{\delta}\langle E_1\rangle(A_1)$, and let us 
prove that $\sigma \in \wc^{\delta}\langle E_2\rangle(A_2)$. 

Let $\pi \in E_2$. By assumption, there exists $ \tau \in E_1$ such that $\tau \preccurlyeq \pi$.
Since $\sigma \in \wc^{\delta}\langle E_1\rangle$, 
$\sigma$ avoids $\tau$. 
From $\tau \preccurlyeq \pi$, we conclude that $\sigma$ also avoids $\pi$. 
Therefore $\sigma \in \wc^{\delta}\langle E_2\rangle$.

Similarly, let $\pi \in A_2$. There exists $\tau \in A_1$ such that $\pi \preccurlyeq \tau$.
Since $\sigma \in \wc^{\delta}\langle E_1\rangle(A_1)$, 
we know that $\sigma$ contains $\tau$. Consequently, $\sigma$ also contains $\pi$. 
Hence $\sigma \in \wc^{\delta}\langle E_2\rangle(A_2)$.
\end{proof}

Note that the condition given by Lemma~\ref{lem:simplification} is sufficient, but it is not necessary. 
Indeed for $E_1 = \{34152\}$, $A_1=\{364152\}$,
we have $\wc^{\delta}\langle E_1\rangle(A_1) = \emptyset$.
In particular
$\wc^{\delta}\langle E_1\rangle(A_1) \subseteq \wc^{\delta}\langle 123\rangle(132)$
even though the first condition of Lemma~\ref{lem:simplification} is not satisfied.
We may however wonder if the condition of Lemma~\ref{lem:simplification} could be necessary 
under the assumption that $E_1$, $E_2$, $A_1$ and $A_2$ satisfy additional conditions, 
like $\wc^{\delta}\langle E_1\rangle(A_1)$ being non-empty and $E_2$ and $A_2$ being antichains. 

Similarly, the condition given by Lemma~\ref{lem:simplification_trivial1} is also sufficient, but it is not necessary. 
Indeed $\wc^{\delta}\langle 132, 213, 231, 312\rangle(12,21)$ is empty
even though the condition of Lemma~\ref{lem:simplification_trivial1} is not satisfied.

\smallskip

Any of the simplifications explained above can be performed on the fly, each time a new restriction term appears while running the procedure~\textsc{Specification}. 
This allows to compute systems of equations significantly more compact than the raw output of our algorithm. 
We do not claim that these simplifications constitute an exhaustive list of all possible simplifications. 
But Propositions~\ref{prop:simplification_restriction} to~\ref{prop:simplification_union} and Lemmas~\ref{lem:simplification_trivial1} and \ref{lem:simplification} 
cover all simplifications that are performed in our implementation of~\textsc{Specification}, and in the examples of Section~\ref{sec:ex}. 

\section{Examples}
\label{sec:ex}

We apply the method described in this article to find a combinatorial specification for three permutation classes. 
The first two are very small and easy examples,
and are the examples of non-substitution-closed permutation classes dealt with in~\cite{BHV08a} (Examples 4.3 and 4.4 of~\cite{BHV08a}). 
They use only a fraction of the machinery we have developed, but 
have been chosen to illustrate some of the differences between our method and the one of~\cite{BHV08a}. 
The third example is more involved, is meant to illustrate all of the auxiliary procedures defined,
and shows better what our methodology brings.

\subsection{Two examples from~\cite{BHV08a}}
\label{ssec:petits_ex}

\subsubsection{The class $\C = Av(132)$}

First, we note that \C is not substitution-closed (indeed, the excluded pattern $132$ is not simple). 
Next, we remark that \C is a subclass of the separable permutations, 
that is to say it contains no simple permutations. (Equivalently, only nodes $\oplus$ and $\ominus$ appear in the decomposition trees.) 
Alternatively, this could be determined from steps 1 and 2 of the algorithmic chain presented in Section~\ref{subsec:chain}.  

Because \C is not substitution-closed, a specification for \C cannot be simply derived as in Section~\ref{sec:substitution-closed}, 
but can be obtained as the result of our algorithm~\textsc{Specification}. 
More precisely, the first step is to call \textsc{EqnForRestriction}$(\wc,\{132\},\emptyset)$
(which works exactly like \textsc{EqnForClass} on an input where the last part is $\emptyset$).
It starts with Equation \eqref{eqn:Wc1} of Proposition~\ref{prop:sys_wc} (p.\pageref{prop:sys_wc}), 
describing the substitution closure $\wc$ of \C, which rewrites in our case as:
\[
\wc =  {1}\ \uplus\ \oplus[\wc^{+},\wc]\ \uplus\  \ominus[\wc^{-},\wc]. 
\]
Then, we propagate into each term of this equation the constraint of avoiding the pattern $132$,
using the embeddings of $132$ into $\oplus$ and $\ominus$.
There is only one non-trivial embedding in this case, corresponding to $132 = \oplus[1,21]$.
This results in the equation 
\[
 \C = \wc\langle 132 \rangle = 1  \cup \oplus[ \wc^+\langle 132 \rangle,  \wc\langle 21 \rangle] \cup \ominus[ \wc^-\langle 132 \rangle,  \wc\langle 132 \rangle].
\]
Because this equation is already non-ambiguous, calling the disambiguation procedure~\textsc{Disambiguate} does not modify it 
(except for changing the union symbols $\cup$ into disjoint union symbols $\uplus$). 
This is the first equation of our specification for $\C$. 

This equation contains right-only restrictions, like $\wc^+\langle 132 \rangle$, $\wc\langle 21 \rangle$, \ldots 
To complete the specification, \textsc{Specification} next computes an equation for each of them, in the same fashion as above, 
that is to say running \textsc{EqnForRestriction} followed by \textsc{Disambiguate}. 
All right-only restrictions that appear are processed iteratively in the same way by~\textsc{Specification}, until none is left. 
Termination is guaranteed since the number of possible restrictions considered is finite (see Section~\ref{ssec:proof_main_Th}). 

The final result of calling~\textsc{Specification} on $\C = Av(132)$ is the following specification: 
\begin{align*}
\C = \wc\langle 132 \rangle & = 1 \uplus \oplus[ \wc^+\langle 132 \rangle,  \wc\langle 21 \rangle] \uplus \ominus[ \wc^-\langle 132 \rangle,  \wc\langle 132 \rangle] \\
\wc^+\langle 132 \rangle & = 1 \uplus \ominus[ \wc^-\langle 132 \rangle,  \wc\langle 132 \rangle] \\
\wc \langle 21 \rangle & = 1 \uplus \oplus[ \wc^+\langle 21 \rangle,  \wc\langle 21 \rangle] \\
\wc^-\langle 132 \rangle & = 1 \uplus \oplus[ \wc^+\langle 132 \rangle,  \wc\langle 21 \rangle] \\
\wc^+ \langle 21 \rangle & = 1.
\end{align*}
An important remark is that this is exactly the result produced by our general method, 
without applying any \emph{ad hoc} argument about this specific $\C$. 
Note that this specification is actually also the result of applying~\textsc{AmbiguousSystem} to $\C$, 
which happens to be a non-ambiguous system already in this case.

\medskip

We now turn to the specifications for $\C = Av(132)$ given in~\cite{BHV08a}. 
A specification for $\C$ may be derived following the proofs of~\cite{BHV08a}. 
As reminded in Section~\ref{subsec:QCS}, the proof of their Lemma 2.1 shows that the
set $\cP_\C = \{Av(12), Av(21), Av(132), \mathcal{I}_\oplus, \mathcal{I}_\ominus\}$ is query-complete,
where $\mathcal{I}_\oplus$ (resp.~$\mathcal{I}_\ominus$) is the set of $\oplus$-indecomposable (resp.~$\ominus$-indecomposable) permutations.
Starting from this query-complete set, the proof of their Theorem 1.1 shows that there exists a specification for \C having $1+2^4=17$ equations.
To obtain this specification, we would need to compute the sets $E_{\X,\si}$ defined in Section~\ref{subsec:QCS},
which is not addressed in general in~\cite{BHV08a}.

This specification is however not the one presented in Example 4.3 of~\cite{BHV08a}.
Using \emph{ad hoc} arguments and constructions, this example shows that 
the set $\mathcal{P} = \{Av(21), Av(132), \mathcal{I}_\oplus, \mathcal{I}_\ominus\}$ is query-complete. 
This allows to derive a specification for $\C$ with significantly fewer equations, namely,  
\begin{align*}
\C & = \C_{Av(21), \mathcal{I}_\oplus, \mathcal{I}_\ominus} \uplus \C_{Av(21), \mathcal{I}_\ominus} \uplus 
\C_{\mathcal{I}_\oplus} \uplus \C_{\mathcal{I}_\ominus} \\
\C_{Av(21), \mathcal{I}_\oplus, \mathcal{I}_\ominus} & = 1 \\
\C_{Av(21), \mathcal{I}_\ominus}& = \oplus[\C_{Av(21), \mathcal{I}_\oplus, \mathcal{I}_\ominus} \,,\, \C_{Av(21), \mathcal{I}_\oplus, \mathcal{I}_\ominus} \uplus \C_{Av(21), \mathcal{I}_\ominus}] \\
\C_{\mathcal{I}_\oplus} & =\ominus[ \C_{Av(21), \mathcal{I}_\oplus, \mathcal{I}_\ominus} \uplus \C_{Av(21), \mathcal{I}_\ominus} \uplus \C_{\mathcal{I}_\ominus} \,,\,
\C_{Av(21), \mathcal{I}_\oplus, \mathcal{I}_\ominus} \uplus \C_{Av(21), \mathcal{I}_\ominus} \uplus \C_{\mathcal{I}_\oplus} \uplus \C_{\mathcal{I}_\ominus} ]\\
\C_{\mathcal{I}_\ominus} & = \oplus[\C_{\mathcal{I}_\oplus} \,,\, \C_{Av(21), \mathcal{I}_\oplus, \mathcal{I}_\ominus} \uplus \C_{Av(21), \mathcal{I}_\ominus}].
\end{align*}
In the system above, for $\X \subset \mathcal{P}' = \{Av(21), \mathcal{I}_\oplus, \mathcal{I}_\ominus\}$, 
the set $\C_\X$ represents permutations of \C satisfying all properties in $\X$ and none in $\mathcal{P}' \setminus \X$. 
Note that the specification above is not exactly the one that would be obtained from $\mathcal{P}$ applying to the letter the proof of~\cite[Theorem 1.1]{BHV08a}, 
which would have $1+2^3 = 9$ equations,
since the empty terms like $\C_{\mathcal{I}_\oplus, \mathcal{I}_\ominus}$ or $\C_{Av(21), \mathcal{I}_\oplus}$ have been removed. 

This specification has the same number of equations as the one we obtained. 
However, ours was derived as the result of a general and algorithmic method, 
whereas that of~\cite{BHV08a} made use of \emph{ad hoc} arguments specific to $\C$.

\medskip

Finally, note that in both our construction and the specification given by~\cite[Example 4.3]{BHV08a}, 
the constraint of avoiding the pattern $21$ appears because of the embedding of $132$ into $\oplus$ (corresponding to $132 = \oplus[1,21]$), 
which is the only non-trivial embedding in this example. 
However, our specification and the one of~\cite{BHV08a} are not the same.
For example, \cite{BHV08a} has an equation for $\C_{\mathcal{I}_\ominus} = \C^-(21)$ 
whereas we do not have an equation for this set, but for $\wc^-\langle 132 \rangle = \C^-$.
Note that beyond the difference of notation, the equations of~\cite{BHV08a} and ours have a different form.
For example, our equation for $\wc^+\langle 132 \rangle = \C^+$ correspond to the equation of~\cite{BHV08a} for 
$\C_{\mathcal{I}_\oplus} = \C^+(21) = \C^+ \setminus \{1\}$,
but since we look at a coarser level, we express this later set as
$\ominus[ \wc^-\langle 132 \rangle,  \wc\langle 132 \rangle]$
while it is expressed in~\cite{BHV08a} as
$\ominus[ \C_{Av(21), \mathcal{I}_\oplus, \mathcal{I}_\ominus} \uplus \C_{Av(21), \mathcal{I}_\ominus} \uplus \C_{\mathcal{I}_\ominus} \,,\,
\C_{Av(21), \mathcal{I}_\oplus, \mathcal{I}_\ominus} \uplus \C_{Av(21), \mathcal{I}_\ominus} \uplus \C_{\mathcal{I}_\oplus} \uplus \C_{\mathcal{I}_\ominus} ]$.

\subsubsection{The class $\C = Av(2413,3142,2143)$} 

This second example is another non-substitution-closed subclass of the set of separable permutations. 
A specification for $\C$ can again be obtained using our algorithm~\textsc{Specification}. 
The starting point is Equation \eqref{eqn:Wc1} of Proposition~\ref{prop:sys_wc} (p.\pageref{prop:sys_wc}), 
in which the $2143$-avoidance constraint has been propagated by \textsc{EqnForRestriction} (or equivalently \textsc{EqnForClass} in this simple case). 
This equation is
\[
\C = \wc\langle 2143 \rangle =  1 \cup \oplus[ \wc^+\langle 2143 \rangle,  \wc\langle 21 \rangle] \cup \oplus[ \wc^+\langle 21 \rangle,  \wc\langle 2143 \rangle] 
\cup \ominus[ \wc^-\langle 2143 \rangle,  \wc\langle 2143 \rangle], 
\]
which is ambiguous. 
Indeed, the two terms $\oplus[ \wc^+\langle 21 \rangle,  \wc\langle 2143 \rangle]$ and $\oplus[ \wc^+\langle 2143 \rangle,  \wc\langle 21 \rangle]$ 
have a non-empty intersection (for instance, $12$ belongs to both terms). 
The reason for this ambiguity is the embedding corresponding to $2143=\oplus[21,21]$, 
which has two components different from $\varepsilon$ and $1$. 
(This is the only non-trivial embedding for the class \C considered.) 

\medskip

Calling the procedure~\textsc{Disambiguate} on the above equation (and performing the simplifications of Section~\ref{ssec:simplifications}), \textsc{Specification} 
obtains the first equation of a specification for $\C$: 
\begin{align*}
\C = & \ \wc\langle 2143 \rangle = 1  \uplus \oplus[ \wc^+\langle 2143 \rangle (21),  \wc\langle 21 \rangle] \uplus \oplus[ \wc^+\langle 21 \rangle,  \wc\langle 2143 \rangle(21)] \uplus \oplus[ \wc^+\langle 21 \rangle,  \wc\langle 21 \rangle] \\
& \hspace*{2cm} \ \uplus \ominus[ \wc^-\langle 2143 \rangle,  \wc\langle 2143 \rangle].
\end{align*}
Note that restrictions with pattern containment constraints, like $\wc^+\langle 2143 \rangle (21)$, have appeared when running~\textsc{Disambiguate}. 
The full specification is then obtained calling iteratively the procedures \textsc{EqnForRestriction} and \textsc{Disambiguate} (with simplifications) 
on all right-only restrictions. 
The final result produced by~\textsc{Specification} for $\C= Av(2413,3142,2143)$ is then 
{\small 
\begin{eqnarray*}
\C = \wc\langle 2143 \rangle & = & 1  \uplus \oplus[ \wc^+\langle 2143 \rangle (21),  \wc\langle 21 \rangle] \uplus \oplus[ \wc^+\langle 21 \rangle,  \wc\langle 2143 \rangle(21)] \uplus \oplus[ \wc^+\langle 21 \rangle,  \wc\langle 21 \rangle]  \label{eqn_ex_21} \\
& ~ & \ \ \uplus \ominus[ \wc^-\langle 2143 \rangle,  \wc\langle 2143 \rangle] \\ 
\wc^+\langle 2143 \rangle (21) & = & \ominus[ \wc^-\langle 2143 \rangle,  \wc\langle 2143 \rangle] \\ 
\wc \langle 21 \rangle & = & 1 \uplus \oplus[ \wc^+\langle 21 \rangle,  \wc\langle 21 \rangle] \label{eqn_ex_24} \\
\wc^+ \langle 21 \rangle & = & 1  \label{eqn_ex_25} \\
\wc\langle 2143 \rangle (21) & = & \oplus[ \wc^+\langle 2143 \rangle (21),  \wc\langle 21 \rangle] \uplus \oplus[ \wc^+\langle 21 \rangle,  \wc\langle 2143 \rangle(21)] 
\uplus \ominus[ \wc^-\langle 2143 \rangle,  \wc\langle 2143 \rangle] \label{eqn_ex_26} \\
\wc^-\langle 2143 \rangle & = &  1 \uplus 
\oplus[ \wc^+\langle 21 \rangle,  \wc\langle 21 \rangle] \uplus \oplus[ \wc^+\langle 2143 \rangle (21),  \wc\langle 21 \rangle] \uplus \oplus[ \wc^+\langle 21 \rangle,  \wc\langle 2143 \rangle(21)]. 
\end{eqnarray*}}

This specification can again be compared with those obtained from~\cite{BHV08a}.
The proofs therein show that the set
$$\cP_\C = \{Av(12), Av(21), Av(132), Av(213), Av(231), Av(312), Av(2413), Av(3142), Av(2143), \mathcal{I}_\oplus, \mathcal{I}_\ominus\}$$
is query-complete, so that there exists a specification for \C having $1+2^8=257$ equations.

Example 4.4 of~\cite{BHV08a} however provides a specification with fewer equations. 
It is derived from the query-complete set $\mathcal{P} = \{Av(21), Av(2143), \mathcal{I}_\oplus, \mathcal{I}_\ominus\}$.
There is however no hint of how this set was found, and especially not of a general method which would have resulted in this query-complete set. 
The specification for $\C$ that follows from this set through the proof of~\cite[Theorem 1.1]{BHV08a} would then have $1+2^3 = 9$ equations. 
Getting rid of empty terms, the specification for \C obtained from this set which is given in~\cite[Example 4.4]{BHV08a} is:
\begin{align*}
\C & = \C_{Av(21), \mathcal{I}_\oplus, \mathcal{I}_\ominus} \uplus \C_{Av(21), \mathcal{I}_\ominus} \uplus 
\C_{\mathcal{I}_\oplus} \uplus \C_{\mathcal{I}_\ominus} \\
\C_{Av(21), \mathcal{I}_\oplus, \mathcal{I}_\ominus} & = 1 \\
\C_{Av(21), \mathcal{I}_\ominus}& = \oplus[\C_{Av(21), \mathcal{I}_\oplus, \mathcal{I}_\ominus} \,,\,  \C_{Av(21),\mathcal{I}_\oplus, \mathcal{I}_\ominus} \uplus \C_{Av(21), \mathcal{I}_\ominus}] \\
\C_{\mathcal{I}_\oplus} & = \ominus[\C_{Av(21), \mathcal{I}_\oplus, \mathcal{I}_\ominus} \uplus \C_{Av(21), \mathcal{I}_\ominus} \uplus \C_{\mathcal{I}_\ominus}  \,,\,  
\C_{Av(21), \mathcal{I}_\oplus, \mathcal{I}_\ominus} \uplus \C_{Av(21), \mathcal{I}_\ominus} \uplus \C_{\mathcal{I}_\oplus} \uplus \C_{\mathcal{I}_\ominus}]\\
\C_{\mathcal{I}_\ominus} & = \oplus[\C_{Av(21), \mathcal{I}_\oplus, \mathcal{I}_\ominus} \,,\,   \C_{\mathcal{I}_\oplus} \uplus \C_{\mathcal{I}_\ominus}]
\uplus \oplus[\C_{\mathcal{I}_\oplus}  \,,\, \C_{Av(21), \mathcal{I}_\oplus, \mathcal{I}_\ominus} \uplus \C_{Av(21), \mathcal{I}_\ominus}].
\end{align*}
In the system above, for $\X \subset \mathcal{P}' = \{Av(21), \mathcal{I}_\oplus, \mathcal{I}_\ominus\}$, 
the set $\C_\X$ represents permutations of \C satisfying all properties in $\X$ and none in $\mathcal{P}' \setminus \X$. 
\smallskip

In this case, our specification obtained as the result of a general method has one
more equation than the one of~\cite{BHV08a} obtained using \emph{ad hoc} constructions. 

\medskip

The non-substitution closed classes addressed in the examples of~\cite{BHV08a} have no simple permutations.
We now turn to the study of a more involved example.

\subsection{The class $\C = Av(1243, 2341, 2413, 41352, 531642)$}
\label{ssec:grand_ex}

With this third example, we apply the method described in this article in a more complicated case. 
Namely, we derive a combinatorial specification for the permutation class $\C = Av(B)$ 
where $B=\{1243, 2341, 2413, 41352, 531642\}$.\footnote{The reader interested in more details about this example can find them in~\cite[Section 3.5]{theseAdeline}.}
From this description we derive its generating function and furthermore generate at random large permutations of the class. 
We follow the different steps described in the diagram of Figure~\ref{fig:schema2} (p.\pageref{fig:schema2}). 

First, note that $Av(B)$ is not substitution-closed as $1243$ and $2341$ are non-simple permutations and belong to the basis of the class. 
Then we test whether the class contains a finite number of simple permutations,
and if it is the case, we compute the set $\s_{\C}$ of simple permutations in \C.
In our case, there is only one simple permutation in \C: $\s_{\C} =\{3142\}$.

Now we have all we need to run our algorithm~\textsc{Specification},
or our algorithm~\textsc{AmbiguousSystem} if we want to compute the generating function using inclusion-exclusion.
Their first step is to compute an equation for \C.

\medskip

\subsubsection{An equation for \C} ~ 

\noindent {\it An equation for the substitution-closure $\wc$ of \C.} 
An equation for the substitution closure $\wc$ of \C immediately follows from $\s_{\C}$. 
Specifically, Equation \eqref{eqn:Wc1} of Proposition~\ref{prop:sys_wc} (p.\pageref{prop:sys_wc}) is in our case:
\begin{equation}
\wc =  {1}\ \uplus\ \oplus[\wc^{+},\wc]\ \uplus\  \ominus[\wc^{-},\wc]\ \uplus\ 3142[\wc,\wc,\wc,\wc]. \label{eq:C} \\
\end{equation}

\noindent {\it From $\wc$ to $\C = Av(B)$.}
Since $\C = \wc \langle \Bstar \rangle$ where $\Bstar$ denotes the non-simple permutations of $B$,
the algorithm~\textsc{AmbiguousSystem} calls \textsc{EqnForClass}$( \wc,\Bstar)$.
Similarly, the algorithm~\textsc{Specification} calls \textsc{EqnForRestriction}$( \wc,\Bstar,\emptyset)$,
which works exactly like \textsc{EqnForClass}$( \wc,\Bstar)$.
Starting from Equation~\eqref{eq:C}, it consists in adding the non-simple pattern avoidance constraints imposed by the avoidance of $B$.
In our case, $1243$ and $2341$ are the only two non-simple patterns in $B$, 
and we compute an equation for $\C = \wc\langle 1243,2341\rangle$ by adding these two constraints one after the other.

\smallskip

To compute $\wc\langle1243\rangle$, we propagate the constraint of avoiding $1243$ into each term of Equation~\eqref{eq:C} using the embeddings of $\gamma = 1243$ into $\oplus$, $\ominus$ and $3142$. 
This gives the following equation: 
\noindent
\begin{align}
\wc\langle 1243 \rangle     
=  {1}\ &\cup\ \oplus[\wc^{+}\langle 12 \rangle,\wc\langle 132 \rangle]\ \cup\ \oplus[\wc^{+}\langle 1243 \rangle,\wc\langle 21 \rangle]\ \cup\ \ominus[\wc^{-}\langle 1 2 4 3 \rangle , \wc\langle1 2 4 3 \rangle] \label{eq:12431}\\
&\cup  3 1 4 2 [\wc\langle1 2 4 3 \rangle , \wc\langle1 2 \rangle , \wc\langle2 1 \rangle , \wc\langle1 3 2 \rangle] \cup 3 1 4 2[\wc\langle1 2 \rangle , \wc\langle1 2 \rangle , \wc\langle1 3 2 \rangle , \wc\langle1 3 2 \rangle]. \nonumber
\end{align}

As explained in Section~\ref{ssec:simplifications}, to obtain the above equation, two types of simplifications 
have been performed. 
First, when a union contains two terms such that one is strictly included in the other, then the smaller term has been removed (see Proposition~\ref{prop:simplification_union}). 
Second, when two excluded patterns are such that the avoidance of one implies the avoidance of the other, then the larger one has been removed (see Proposition~\ref{prop:simplification_restriction}) -- for instance we have simplified $\wc\langle1243,12\rangle$ into $ \wc\langle12\rangle$. 

\smallskip

Next we add the constraint of avoiding $2341$ in the above equation as $\wc\langle 1243 \rangle\langle 2341 \rangle $ $ = \wc\langle1243, 2341\rangle$.
We propagate the constraint  $2341$ into the $5$ different terms that appear in Equation~\eqref{eq:12431}, and perform simplifications as described above.
The output of \textsc{EqnForClass}$( \wc,\Bstar)$ or \textsc{EqnForRestriction}$( \wc,\Bstar,\emptyset)$ is then:
\begin{align}
\wc\langle1 2 4 3 ,2 3 4 1 \rangle\ =\ 1\ 
	&\cup\ \oplus [\wc^{+}\langle1 2 4 3 ,2 3 4 1 \rangle , \wc\langle2 1 \rangle] 
	\ \cup\ \oplus [\wc^{+}\langle1 2 \rangle , \wc\langle1 3 2 ,2 3 4 1 \rangle] \label{eq:firstline}\\
	&\cup\ \ominus [\wc^{-}\langle1 2 3 \rangle , \wc\langle1 2 4 3 ,2 3 4 1 \rangle] 
	\ \cup\ 3 1 4 2 [\wc\langle1 2 \rangle , \wc\langle1 2 \rangle , \wc\langle1 2 \rangle , 	\wc\langle1 3 2 ,2 3 4 1 \rangle]. \nonumber
\end{align}

If we run \textsc{AmbiguousSystem}, this is exactly the first equation of the system.
This equation is ambiguous, since the intersection of 
$t_1 = \oplus [\wc^{+}\langle1 2 \rangle , \wc\langle1 3 2 ,2 3 4 1 \rangle] $ and 
$t_2 = \oplus [\wc^{+}\langle1 2 4 3 ,2 3 4 1 \rangle, \wc\langle2 1 \rangle] $ is 
$\oplus[\wc^{+}\langle1 2 \rangle , \wc\langle2 1 \rangle]$ which is not empty.

\medskip

\noindent {\it Disambiguation.}
If we run \textsc{Specification}, the equation is disambiguated using the procedure~\textsc{Disambiguate}.
We can write $(t_1 \cup t_2)$ as $(t_{1} \cap \bar{t_{2}}) \uplus (\bar{t_{1}} \cap t_{2}) \uplus (t_{1} \cap t_{2})$. 
The computation of the complement terms $\bar{t_{1}}$ and $\bar{t_{2}}$ (with Equation~\eqref{eq:ComplementTerm} p.\pageref{eq:ComplementTerm}) 
increases the number of terms in the disjoint union that describes $t_1 \cup t_2$, 
producing in this case a disjoint union of 15 terms. 
After eliminating terms that are empty or strictly included in another one, Equation~\eqref{eq:firstline} is replaced by:
\begin{small}
\begin{align}
\wc\langle1 2 4 3 ,2 3 4 1 \rangle = 1 
&\uplus \oplus [\wc^{+}\langle1 2 4 3 ,2 3 4 1 \rangle(1 2 ) , \wc\langle2 1 \rangle]  
 \uplus \oplus [\wc^{+}\langle1 2 \rangle , \wc\langle1 3 2 ,2 3 4 1 \rangle(2 1 )] 
 \uplus \oplus [\wc^{+}\langle1 2 \rangle , \wc\langle2 1 \rangle] \label{eq:C_non_ambigu}\\
&\uplus \ominus [\wc^{-}\langle1 2 3 \rangle , \wc\langle1 2 4 3 ,2 3 4 1 \rangle] 
	\uplus 3 1 4 2 [\wc\langle1 2 \rangle , \wc\langle1 2 \rangle , \wc\langle1 2 \rangle , \wc\langle1 3 2 ,2 3 4 1 \rangle]. \nonumber
\end{align}	
\end{small}
This is the final version of the first equation of the specification.

\medskip

\subsubsection{The whole system} ~ 

\noindent {\it AmbiguousSystem.}
To obtain a complete system, \textsc{AmbiguousSystem} 
starts with Equation~\eqref{eq:firstline} computed above. 
Using \textsc{EqnForClass}, it iterates the process of computing equations for new classes $\wc^{\delta}\langle E \rangle$ appearing on the right side of this equation, 
and it does it subsequently for any class appearing in one of these new equations.
As already mentioned, this process terminates since there is a finite number of subsets $E$ of patterns of permutations of $B$.

The result of \textsc{AmbiguousSystem} on \C is the following:
\begin{small}
\begin{align*}
\wc\langle1 2 4 3 ,2 3 4 1 \rangle\ =\ 1\ 
	&\cup\ \oplus [\wc^{+}\langle1 2 4 3 ,2 3 4 1 \rangle , \wc\langle2 1 \rangle] 
	\ \cup\ \oplus [\wc^{+}\langle1 2 \rangle , \wc\langle1 3 2 ,2 3 4 1 \rangle]\\
	&\cup\ \ominus [\wc^{-}\langle1 2 3 \rangle , \wc\langle1 2 4 3 ,2 3 4 1 \rangle] 
	\ \cup\ 3 1 4 2 [\wc\langle1 2 \rangle , \wc\langle1 2 \rangle , \wc\langle1 2 \rangle , 	\wc\langle1 3 2 ,2 3 4 1 \rangle] \nonumber\\		
\wc^{+}\langle1 2 4 3 ,2 3 4 1 \rangle\ = \ 1\ 
	&\cup\ \ominus [\wc^{-}\langle1 2 3 \rangle , \wc\langle1 2 4 3 ,2 3 4 1 \rangle] 
	\ \cup\ 3 1 4 2 [\wc\langle1 2 \rangle , \wc\langle1 2 \rangle , \wc\langle1 2 \rangle , 	\wc\langle1 3 2 ,2 3 4 1 \rangle] \nonumber \\		
\wc\langle2 1 \rangle\ =\ 1\ 
	&\cup\ \oplus [\wc^{+}\langle2 1 \rangle , \wc\langle2 1 \rangle]\nonumber \\		
\wc^{+}\langle1 2 \rangle\ = \ 1 \ & \nonumber \\		
\wc\langle1 3 2 ,2 3 4 1 \rangle\ =\ 1\ 
	&\cup\ \oplus [\wc^{+}\langle1 3 2 ,2 3 4 1 \rangle , \wc\langle2 1 \rangle] \ \cup\ \ominus [\wc^{-}\langle1 3 2 ,1 2 3 \rangle , \wc\langle1 3 2 ,2 3 4 1 \rangle] \nonumber \\		
\wc^{-}\langle1 2 3 \rangle\ = \ 1\ &
	\cup\ \oplus [\wc^{+}\langle1 2 \rangle , \wc\langle1 2 \rangle]
	\ \cup\ 3 1 4 2 [\wc\langle1 2 \rangle , \wc\langle1 2 \rangle , \wc\langle1 2 \rangle , \wc\langle1 2 \rangle]  \nonumber \\		
\wc\langle1 2 \rangle\ =\ 1\ 
	&\cup\ \ominus [\wc^{-}\langle1 2 \rangle , \wc\langle1 2 \rangle]\nonumber \\	
\wc^{+}\langle2 1 \rangle\ = \ 1\ & \nonumber \\	
\wc^{+}\langle1 3 2 ,2 3 4 1 \rangle\ = \ 1\ &
	\cup\ \ominus [\wc^{-}\langle1 3 2 ,1 2 3 \rangle , \wc\langle1 3 2 ,2 3 4 1 \rangle] \nonumber \\
\wc^{-}\langle1 3 2 ,1 2 3  \rangle\ = \ 1\ &
	\cup\ \oplus [\wc^{+}\langle1 2 \rangle , \wc\langle2 1 ,1 2 \rangle] \nonumber \\		
\wc^{-}\langle1 2 \rangle\ = \ 1\ & \nonumber \\		
\wc\langle21,12 \rangle\ = \ 1.\ & \nonumber 
\end{align*}
\end{small}
As noticed earlier, the first equation of the system is indeed ambiguous.
This system can be used to compute the generating function of \C using inclusion-exclusion.
Since we want not only to compute the generating function of \C, but also to generate at random uniform permutations of \C,
we rather compute a specification for \C.

\medskip

\noindent {\it Specification.}
The unambiguous equation for \C computed by \textsc{Specification} is Equation~\eqref{eq:C_non_ambigu}.
As noticed in Section~\ref{sec:disambiguate}, right-only terms, possibly involving pattern containment constraints like $\wc\langle1 3 2 ,2 3 4 1 \rangle(2 1)$, appear in this equation. 
These terms are not defined by our system, so we have to compute an equation for each of them (and iteratively so for right-only terms appearing in these new equations).
This is done by iterating~\textsc{EqnForRestriction}.
This procedure works in a fashion similar to \textsc{EqnForClass}, 
except that, in addition, it propagates pattern containment constraints in the subtrees.
For the term $\wc\langle1 3 2 ,2 3 4 1 \rangle(2 1)$, we describe how Algorithm \textsc{EqnForRestriction} computes an additional equation. 
We do not give details for the other right-only terms. 

First, starting from Equation~\eqref{eq:C}, we add the pattern avoidance constraints, namely $132$ and $2341$, and compute an equation for $\wc\langle1 3 2 ,2 3 4 1 \rangle$ like before.
We obtain the following equation:
{\small \begin{equation*}
\wc\langle1 3 2 ,2 3 4 1 \rangle = 1
\ \cup\ \oplus [\wc^{+}\langle1 3 2 ,2 3 4 1 \rangle , \wc\langle2 1 \rangle]
\ \cup\ \ominus [\wc^{-}\langle1 3 2 ,1 2 3 \rangle , \wc\langle1 3 2 ,2 3 4 1 \rangle]. 
\end{equation*}}

Then, we add the pattern containment constraints (here $21$) to this equation, considering embeddings of $21$ in $\oplus$ and $\ominus$.
The equation obtained after simplification is 
\[
\wc\langle1 3 2 ,2 3 4 1 \rangle(21) = 
\oplus [\wc^{+}\langle1 3 2 ,2 3 4 1 \rangle(2 1 ) , \wc\langle2 1 \rangle]
\ \cup\ \ominus [\wc^{-}\langle1 3 2 ,1 2 3 \rangle , \wc\langle1 3 2 ,2 3 4 1 \rangle].
\]
This equation is already unambiguous, so this is its final version.

\smallskip

The process is iterated  until each term that appear on the right hand side of an equation of the system is defined by an unambiguous equation, 
which finally leads to the following specification: 

\begin{footnotesize}
\begin{eqnarray*}
\wc\langle1 2 4 3 ,2 3 4 1 \rangle 
	&=& 1 
	\uplus\oplus [\wc^{+}\langle1 2 \rangle , \wc\langle1 3 2 ,2 3 4 1 \rangle(2 1)]   	
	\uplus \oplus [\wc^{+}\langle1 2 4 3 ,2 3 4 1 \rangle(12) , \wc\langle21\rangle]
	\uplus \oplus [\wc^{+}\langle1 2 \rangle , \wc\langle2 1 \rangle]\\
 	&\uplus & 	   
 	\ominus [\wc^{-}\langle1 2 3 \rangle , \wc\langle1 2 4 3 ,2 3 4 1 \rangle]   
 	\uplus 3 1 4 2 [\wc\langle1 2 \rangle , \wc\langle1 2 \rangle , \wc\langle1 2 \rangle , \wc\langle1 3 2 ,2 3 4 1 \rangle]\\
\wc^{+}\langle1 2 \rangle 
	&=& 1  \uplus\ominus [\wc^{-}\langle1 2 \rangle , \wc\langle1 2 \rangle]\\
\wc\langle1 3 2 ,2 3 4 1 \rangle(2 1 ) 
	&=&  
	\oplus [\wc^{+}\langle1 3 2 ,2 3 4 1 \rangle(2 1 ) , \wc\langle2 1 \rangle]
	\uplus \ominus [\wc^{-}\langle1 3 2 ,1 2 3 \rangle , \wc\langle1 3 2 ,2 3 4 1 \rangle]  \\
\wc^{+}\langle1 2 4 3 ,2 3 4 1 \rangle(1 2 ) 
	&=& 
	\ominus [\wc^{-}\langle1 2 3 \rangle(1 2 ) , \wc\langle1 2 4 3 ,2 3 4 1 \rangle(1 2 )] 
	\uplus \ominus [\wc^{-}\langle1 2 \rangle , \wc\langle1 2 4 3 ,2 3 4 1 \rangle(1 2 )] \\
 	&\uplus&  
 	\ominus [\wc^{-}\langle1 2 3 \rangle(1 2 ) , \wc\langle1 2 \rangle]   
 	\uplus 3 1 4 2 [\wc\langle1 2 \rangle , \wc\langle1 2 \rangle , \wc\langle1 2 \rangle , \wc\langle1 3 2 ,2 3 4 1 \rangle]\\
\wc\langle2 1 \rangle 
	&=& 1  \uplus\oplus [\wc^{+}\langle2 1 \rangle , \wc\langle2 1 \rangle]\\
\wc^{-}\langle1 2 3 \rangle 
	&=& 1   
	\uplus \oplus [\wc^{+}\langle1 2 \rangle , \wc\langle1 2 \rangle]
	\uplus 3 1 4 2 [\wc\langle1 2 \rangle , \wc\langle1 2 \rangle , \wc\langle1 2 \rangle , \wc\langle1 2 \rangle] \\
\wc\langle1 2 \rangle 
	&=& 1  
	\uplus\ominus [\wc^{-}\langle1 2 \rangle , \wc\langle1 2 \rangle]\\
\wc\langle1 3 2 ,2 3 4 1 \rangle 
	&=& 1     
	\uplus \oplus [\wc^{+}\langle1 3 2 ,2 3 4 1 \rangle , \wc\langle2 1 \rangle]
	\uplus\ominus [\wc^{-}\langle1 3 2 ,1 2 3 \rangle , \wc\langle1 3 2 ,2 3 4 1 \rangle]\\
\wc^{-}\langle1 2 \rangle
	 &=& 1\\
\wc^{+}\langle1 3 2 ,2 3 4 1 \rangle(2 1 ) 
	&=& 
	\ominus [\wc^{-}\langle1 3 2 ,1 2 3 \rangle , \wc\langle1 3 2 ,2 3 4 1 \rangle]\\
\wc^{-}\langle1 3 2 ,1 2 3 \rangle 
	&=& 1  	
	\uplus\oplus [\wc^{+}\langle1 2 \rangle , \wc\langle2 1 ,1 2 \rangle]\\
\wc^{-}\langle1 2 3 \rangle(1 2 ) 
	&=& 
	\oplus [\wc^{+}\langle1 2 \rangle , \wc\langle1 2 \rangle]
	\uplus 3 1 4 2 [\wc\langle1 2 \rangle , \wc\langle1 2 \rangle , \wc\langle1 2 \rangle , \wc\langle1 2 \rangle]  \\
\wc\langle1 2 4 3 ,2 3 4 1 \rangle(1 2 ) 
	&=&    
	\oplus [\wc^{+}\langle1 2 \rangle , \wc\langle1 3 2 ,2 3 4 1 \rangle(2 1 )] 
 	\uplus \oplus [\wc^{+}\langle1 2 \rangle , \wc\langle2 1 \rangle]   
 	\uplus \oplus [\wc^{+}\langle1 2 4 3 ,2 3 4 1 \rangle(1 2 ) , \wc\langle2 1 \rangle] \\
 	&\uplus&   
 	\ominus [\wc^{-}\langle1 2 3 \rangle(1 2 ) , \wc\langle1 2 \rangle] 
 	\uplus \ominus [\wc^{-}\langle1 2 \rangle , \wc\langle1 2 4 3 ,2 3 4 1 \rangle(1 2)] 
 	\\
 	&\uplus& 
 	\ominus [\wc^{-}\langle1 2 3 \rangle(1 2 ) , \wc\langle1 2 4 3 ,2 3 4 1 \rangle(1 2 )]
	\uplus 3 1 4 2 [\wc\langle1 2 \rangle , \wc\langle1 2 \rangle , \wc\langle1 2 \rangle , \wc\langle1 3 2 ,2 3 4 1 \rangle]\\
\wc^{+}\langle2 1 \rangle 
	&=& 1\\
\wc^{+}\langle1 3 2 ,2 3 4 1 \rangle 
	&=& 1  
	\uplus\ominus [\wc^{-}\langle1 3 2 ,1 2 3 \rangle , \wc\langle1 3 2 ,2 3 4 1 \rangle]\\
\wc\langle2 1 ,1 2 \rangle 
	&=& 1.
\end{eqnarray*}
\end{footnotesize}

The above specification contains $16$ equations. 
This should be compared to the number of equations in the specification that could be derived 
using the general method described in the proofs of~\cite{BHV08a} and reviewed in Section~\ref{subsec:QCS}
(and which would require to compute the sets $E_{\X,\si}$ defined in Section~\ref{subsec:QCS}).
The number of equations following this method would be \emph{a priori} $1+2^{k+2}$ where $k$ is the number of proper patterns of permutations in $B$ 
(in this case, $k=27$, giving $536\,870\,913$ equations), although we can expect many empty terms and hence fewer equations.

\subsubsection{Byproducts of the combinatorial specification}

The translation of the above specification into a system of equations for the generating function of $\C$ 
produces in this specific case a system that can be solved, giving access to a closed form for the generating function of \C: 

\[C(z) = \frac{(z^6-7z^5+20z^4-28z^3+20z^2-7z+1)z}{1-9z+32z^2-59z^3+62z^4-37z^5+13z^6-2z^7}.\]

We can also translate this specification into (Boltzmann or recursive) uniform random samplers of permutations in \C. 
We have used the recursive sampler\footnote{In this case (and in the one of Figure~\ref{fig:exemple_random}), we chose a recursive sampler over a Boltzmann one 
because it is more suitable when you want to sample many random permutations of the same medium size.} 
so obtained to uniformly generate permutations in \C. 
This is how Figure~\ref{fig:exemple900} (p.\pageref{fig:exemple900}) was obtained, 
as well as Figure~\ref{fig:exemple_random} (p.\pageref{fig:exemple_random}) for other examples of classes $\C$. 

\bigskip

\subsection*{Acknowledgments} 
We are grateful to the anonymous referees and to Mireille Bousquet-Mélou, 
whose comments have helped us make the relation between our work and~\cite{BHV08a} clearer. 

This work was completed with the support of the ANR project MAGNUM number 2010\_BLAN\_0204.

\bibliographystyle{alpha}

\end{document}